\documentclass[10pt,reqno]{amsart}
\usepackage[sort]{cite}
\usepackage{amsfonts} 
\usepackage{amsmath} 
\usepackage{amssymb} 
\usepackage{amsthm} 
\usepackage{latexsym} 
\usepackage{mathrsfs} 
\usepackage{mathtools} 
\usepackage{slashed}  

\usepackage{verbatim}
\usepackage{cases}
\usepackage[dvips]{epsfig}
\usepackage{epsf} 
\usepackage{url}
\usepackage{color}
\usepackage{xcolor}
\usepackage{graphicx}
\usepackage{enumitem} 

\usepackage{pifont} 
\usepackage{upgreek} 
\usepackage{fancyhdr} 
\usepackage{calligra} 
\usepackage{marvosym} 
\usepackage[percent]{overpic} 
\usepackage{pict2e} 
\usepackage[hypcap=false]{caption} 
\usepackage{wasysym} 
\usepackage[margin=3cm]{geometry}
\usepackage{lipsum}
\newcommand{\la}{\langle}
\newcommand{\ra}{\rangle}

\theoremstyle{plain}
\newtheorem{theorem}{Theorem}[section]

\newtheorem{lemma}[theorem]{Lemma}
\newtheorem{proposition}[theorem]{Proposition}
\newtheorem{corollary}[theorem]{Corollary}

\theoremstyle{definition}
\newtheorem{remark}[theorem]{Remark}

\newtheorem{definition}[theorem]{Definition}

\numberwithin{equation}{section}
\allowdisplaybreaks

\newcommand{\tr}{\operatorname{tr}}

\newcommand{\cA}{\mathcal{A}}
\newcommand{\cB}{\mathcal{B}}
\newcommand{\cC}{\mathcal{C}}
\newcommand{\cD}{\mathcal{D}}
\newcommand{\cE}{\mathcal{E}}
\newcommand{\cF}{\mathcal{F}}

\newcommand{\cH}{\mathcal{H}}

\newcommand{\cK}{\mathcal{K}}
\newcommand{\cL}{\mathcal{L}}
\newcommand{\cN}{\mathcal{N}}
\newcommand{\cM}{\mathcal{M}}
\newcommand{\cP}{\mathcal{P}}

\newcommand{\cS}{\mathcal{S}}

\newcommand{\bB}{\mathbb{B}}
\newcommand{\bC}{\mathbb{C}}
\newcommand{\bE}{\mathbb{E}}

\newcommand{\bJ}{\mathbb{J}}
\newcommand{\bK}{\mathbb{K}}
\newcommand{\bN}{\mathbb{N}}

\newcommand{\bP}{\mathbb{P}}
\newcommand{\bR}{\mathbb{R}}
\newcommand{\bS}{\mathbb{S}}

\newcommand{\bX}{\mathbb{X}}

\newcommand{\fF}{\mathfrak{F}}

\newcommand{\sD}{\mathsf{D}}

\newcommand{\sM}{\mathsf{M}}
\newcommand{\sN}{\mathsf{N}}
\newcommand{\sR}{\mathsf{R}}
\newcommand{\sT}{\mathsf{T}}

\newcommand{\gf}{g_{\flat}}
\newcommand{\gs}{g^{\sharp}}

\def\Re{\mathcal{R}}
\def\Rek{\mathcal{R}_k}
\def\Rc{\mathcal{R}^c}
\def\Rck{\mathcal{R}^c_k}
\def\kfk{\varphi^*_k}
\def\kbk{(\varphi_k)_*}

\newcommand{\PH}{\mathbb{P}_H}

\newcommand{\gd}{\mathsf{grad}\,}
\newcommand{\dv}{\mathsf{div}\,}
\newcommand{\Div}{\mathrm{div}}
\newcommand{\Ric}{\mathsf{Ric}}
\newcommand{\Aw}{A^\mathsf{w}}
\newcommand{\Fw}{F^\mathsf{w}}
\newcommand{\Xw}{X^\mathsf{w}}
\newcommand{\sw}{\mathsf{w}}
\newcommand{\Xwm}{X^\mathsf{w}_{\gamma,\mu_c^{\mathsf{w}}}}
\newcommand{\mwc}{\mu_c^\mathsf{w}}
\newcommand{\Lis}{\mathcal{L}is}
\begin{document}
\title[The Navier-Stokes equations on manifolds with boundary]{The Navier-Stokes equations  on manifolds with boundary}

\author{Yuanzhen Shao}
\address{The University of Alabama\\
	Tuscaloosa, Alabama \\
	USA}
\email{yshao8@ua.edu}

\author{Gieri Simonett}
\address{
        Vanderbilt University\\
        Nashville, Tennessee\\
        USA}
\email{gieri.simonett@vanderbilt.edu}

\author{Mathias Wilke}
\address{
        Martin-Luther-Universit\"at Halle-Wittenberg\\
        Halle (Saale)\\
        Germany}
\email{mathias.wilke@mathematik.uni-halle.de}

\thanks{This work was supported by a grant from the Simons Foundation (\#853237, Gieri Simonett), a grant from the National Science Foundation (DMS-2306991, Yuanzhen Shao), and a CARSCA grant from the University of Alabama (Yuanzhen Shao).}





\begin{abstract}
We consider the motion of an incompressible viscous fluid on a compact Riemannian manifold $\sM$ with boundary.
The motion  on $\sM$ is modeled by the incompressible Navier-Stokes equations, and the fluid is subject to pure or partial slip boundary conditions
of Navier type on $\partial\sM$.
We establish existence and uniqueness  of strong as well as weak (variational) solutions for initial data in critical spaces. Moreover, we show that the set of equilibria consists of Killing vector fields on $\sM$ that satisfy corresponding boundary conditions, and we prove that all equilibria are (locally) stable. In case $\sM$ is two-dimensional we show that solutions with divergence free initial condition in $L_2(\sM; T\sM)$ exist globally and converge to an equilibrium exponentially fast.
\end{abstract}

\maketitle

\section{Introduction}\label{S:Intro}
Suppose that $\sM$ is a  compact, smooth, connected and oriented $n$-dimensional Riemannian manifold with boundary $\Sigma=\partial\sM$.
It follows that
$\Sigma$ is a compact, smooth, orientable $(n-1)$-dimensional manifold.
$\Sigma$ is then provided with outward orientation with respect to $\sM$.
Let $(\cdot | \cdot)_g$ denote the Riemann metric on $\sM$.
In the sequel, we also use the notation $(\cdot | \cdot)_g$ for the induced Riemann metric on $\Sigma$.
We will study the motion of  an incompressible viscous fluid on $\sM$, modeled by the
{\em surface Navier-Stokes equations} with {\em Navier boundary conditions}
which can be stated as follows
\begin{equation}
\label{NS-sys}
\left\{\begin{aligned}
\varrho \left( \partial_t u + \nabla_u  u\right) -2\mu_s \dv D(u) + \gd\pi &=0 &&\text{on}&&\sM ,\\
\dv u&=0 &&\text{on}&&\sM ,\\
\alpha u + \cP_\Sigma \left((\nabla u + [\nabla u]^{\sf T})\nu_\Sigma \right)  &=0 &&\text{on}&&\Sigma  , \\
(u | \nu_\Sigma)_g &=0 &&\text{on}&&\Sigma  , \\
u(0) & = u_0 &&\text{on}&& \sM. \\
\end{aligned}\right.
\end{equation}
Here, the unknowns are the fluid velocity $u$  and the fluid pressure $\pi$.  $\varrho>0$ is the (constant) density, $\mu_s>0$ is the surface shear viscosity,
$\nu_\Sigma$ is the outward unit normal field of $\Sigma$,
while
 $\cP_\Sigma$ is the orthogonal projection onto the tangent bundle of $\Sigma$, and the constant $\alpha\geq 0$ is a given friction parameter.
In the following, we assume without loss of generality that $\varrho=1$.

Moreover, $\nabla_u v$ denotes the covariant derivative induced by the  Levi-Civita connection of $\sM$ for given tangent vectors $u,v$, and
$
D(u):= \frac{1}{2} (\nabla u+ [\nabla u]^{\sT} )^\sharp $
denotes the deformation tensor  (a definition of the operator $^\sharp$ is provided in Appendix \ref{Appendix A}),
given in local coordinates by
\begin{equation*}
D(u)   = \frac{1}{2} \left(g^{jk} u^i_{|k} + g^{ik} u^j_{|k}\right)\frac{\partial}{\partial x^i}\otimes \frac{\partial }{\partial x^j},
\end{equation*}
with $u^i_{|k}$  being covariant derivatives, that is,
\begin{equation*}
u^i_{|k} = \partial_k u^i + \Gamma^i_{k\ell}  u^\ell \quad \text{for} \quad u=u^i \frac{\partial}{\partial x^i}.
\end{equation*}
 Here and throughout this article, we are using the Einstein summation convention, indicating that terms with repeated indices are added.

\goodbreak
\noindent
We note here that
$$D_u:= \frac{1}{2} \left( \nabla u+ [\nabla u]^{\sT} \right)$$
is a $(1,1)$-tensor, while $D(u)$ is a $(2,0)$-tensor.
In case $\dv u=0$, it is well-known, see for instance \cite[Lemma 2.1]{SaTu20}, that
\begin{equation}
\begin{split}
\label{div-D-local}
2\, \dv D(u) = \Delta_{\sM} u +  \Ric^\sharp u,
\end{split}
\end{equation}
where $\Delta_\sM$ denotes the (negative) Bochner Laplacian (sometimes also called the connection Laplacian), and $\Ric^\sharp$
is the Ricci $(1,1)$-tensor.
In local coordinates, these operators are expressed by
\begin{equation}
\label{Bochner-Ricci-local}
\Delta_\sM u= g^{ij}(\nabla_i \nabla_j -\Gamma_{ij}^k \nabla_k)u , \quad
 \Ric^\sharp u = R^i_j u^j\frac{\partial} {\partial x^i}:=g^{ik} R_{kj} u^j \frac{\partial} {\partial x^i},
\end{equation}
with $\nabla_j=\nabla_{\frac{\partial}{\partial x^j}}$ being covariant derivatives, and where $\Ric = R_{ij} dx^i\otimes dx^j$ is the usual Ricci $(0,2)$-tensor.
 More details are given in Appendix A.

\eqref{NS-sys}$_3$ and \eqref{NS-sys}$_4$ is termed the {\em pure slip} boundary condition in case $\alpha=0$, or {\em partial slip} boundary condition in case $\alpha>0$.

In case $\sM=\bR^n_+:=\bR^{n-1}\times (0,\infty)$, the boundary conditions for $u=(u^1,\ldots, u^n)$ on $\Sigma=\bR^{n-1}$  result in
\begin{equation*}
\begin{aligned}
u^n=0,\qquad  (\partial_j u^n + \partial_n u^j) - \alpha u^j =0,\quad j=1,\ldots, n-1,
\end{aligned}
\end{equation*}
which, taking into account  the relation $u^n=0$, further reduce to
\begin{equation*}
\begin{aligned}
u^n=0,\qquad  \partial_n u^j - \alpha u^j =0,\quad j=1,\ldots, n-1.
\end{aligned}
\end{equation*}
This implies
$u^j(x^\prime,x^n)\approx (1+ \alpha x^n) u^j(x^\prime,0)$
for small $x^n>0$, showing the friction effect on $\Sigma$ for tangential velocity components in case $\alpha>0$.

\medskip\noindent
The topic of fluids on surfaces and Riemannian manifolds has recently attracted attention by numerous authors,
see for instance \cite{CCD17, JaOlRe17, MiMo09, OQRY18, Pri94, PrSiWi21, ReZh13, SaTu20, SiWi22} and the references contained
in these publications.

One application concerns the modeling of emulsion and biological membranes, see  \cite{SlSaLe07}.
In addition,  \eqref{NS-sys} may  be considered as a model for the motion
of a fluid on a planet's surface that is covered by water and landmasses (while the effect of Coriolis forces is being ignored).

\medskip
The main results in this manuscript establish existence, uniqueness, and qualitative properties of strong as well as weak (variational)
solutions to  \eqref{NS-sys}.
 The expression `(variational) weak solutions' is used here to distinguish our solutions from the class of Leray-Hopf weak solutions.
Our approach is based on the method of $L_p$-$L_q$ maximal regularity in time weighted spaces, see for instance~\cite{PruSim16}.

\medskip
In Sections \ref{Section:NS with perfect slip boundary condition} and \ref{Stokes-H-infity},
we demonstrate that the Stokes operator associated with \eqref{NS-sys}  admits a bounded $H^\infty$-calculus
with angle $<\pi/2$ (a property that implies maximal regularity)
in  $L_{q,\sigma}(\sM; T\sM)$ as well as in  $H^{-1}_{q,\sigma}(\sM;T\sM)$.
This property opens up the way to obtain unique solutions to \eqref{NS-sys} for initial data in critical spaces,
as shown in Section~\ref{Section:existence and uniqueness}, Theorem~\ref{Thm: local wellposed}, Corollary~\ref{Cor:existence}, and Remark~\ref{Rmk: two cases wellposed}.

\medskip
In Section~\ref{Section:large time behavior}, we show that the set of equilibria of \eqref{NS-sys} consists exactly of all Killing fields on $\sM$ which satisfy the boundary conditions imposed on solutions, see Proposition~\ref{Prop: equilibrium set}.
In particular,  we show that in case of a positive friction coefficient $\alpha$, equilibria correspond to the situation where the fluid is at rest.

One of the main results of this paper is contained in Theorem~\ref{Thm: gloabl existence and convergence 2D}.
It shows that in case ${\rm dim}\,\sM=2$, any solution with initial value
$u_0\in L_{2,\sigma}(\sM; T\sM)$ exists globally and converges to an equilibrium at an exponential rate.
Moreover, in case ${\rm dim}\,\sM>2$,  we show in Theorem~\ref{Thm: stability near killing 1} and Corollary~\ref{cor:stability}
that all equilibria are locally stable: solutions that start out close to an equilibrium exist globally and converge
at an exponential rate to a (possibly different) equilibrium.

\medskip\noindent
We add three examples to  illustrate the scope of our results for two-dimensional surfaces in $\bR^3$.
For $\alpha\ge 0$, let $\cE_\alpha$ denote the set of equilibria of \eqref{NS-sys}.\\
\goodbreak
{\bf Examples:}
\begin{itemize}
\item[(a)] Let $\sM= \bS^2_+=\{x=(x_1,x_2,x_3)\in \bS^2: x_3>0\}$ be the upper hemisphere in $\bR^3$.
Then
\begin{equation*}
\cE_\alpha =
\left\{
\begin{aligned}
&\{0\}\quad &&\text{in case } \alpha>0, \\
&\{\omega e_3 \times x : \omega\in \bR,\ x\in \bS^2_+\} &&\text{in case } \alpha=0.
\end{aligned}\right.
\end{equation*}
That is, in case of pure slip boundary conditions, the equilibria correspond to the situation where the fluid rotates with constant angular speed $\omega$ about the $z$-axis.

Theorem~\ref{Thm: gloabl existence and convergence 2D} says that in case $\alpha>0$, any solution with initial value $u_0\in L_{2,\sigma}(\sM; T\sM)$ exists globally and converges at
an exponential rate to the equilibrium state $u_\infty=0$.

If $\alpha=0$, any solution with initial value $u_0\in L_{2,\sigma}(\sM; T\sM)$ exists globally and converges at
an exponential rate to an equilibrium state $u_\infty = \omega e_3 \times x$, for some $\omega \in \bR$
which is determined by Theorem~\ref{Thm: gloabl existence and convergence 2D}.
\vspace{1mm}
\item[(b)] Analogous results hold in case $\sM$ is a disk in $\bR^2$ with center at the origin (embedded in $\bR^3$).
\vspace{1mm}
\item[(c)] Let $\sM=\{x=(x_1,x_2,x_3)\in \bR^3: x_1^2 + x_2^2=1, 0< x_3 <1\}$ be a cylinder of finite height.
Then analogous results to Example (a) hold.
\end{itemize}
It is interesting to note that even in the simple Euclidean setting of Example~(b), the results seem to be new, at least in the case where $\alpha=0$.

For surfaces, in  case $\alpha>0$, the global convergence results is based on the
 fact that all Killing fields are trivial, see Proposition \ref{Prop: characterizing Ealpha}.
In case $\alpha=0$, the results follow from the somewhat surprising observation that the evolution equation
leaves the orthogonal space to Killing fields invariant, see Lemma \ref{Lem: aux global}(a),
and  from Korn's inequality, see Lemma \ref{Appendix Lem: korn}.

\medskip\noindent
As another application of Theorem~\ref{Thm: stability near killing 1}
 we consider the three-dimensional manifold $\sM $ consisting of a solid ball in $\bR^3$ with center at the origin.
 Theorem~\ref{Thm: stability near killing 1} and Corollary~\ref{cor:stability} then show that rotations about any axis through the origin are stable:
 solutions that start close to a rotation exist globally and converge to a (possibly different) rotation. We are not aware of a corresponding result in the literature.

\medskip
In the appendices A through D we collect and prove results concerning Riemannian manifolds (with boundary), Green's formula, Korn's inequality,
solvability of elliptic problems
and the existence of the Helmholtz projection, interpolation for mixed boundary conditions, sectorial operators and  the $H^\infty$-calculus.
These results are used throughout the manuscript and are also of independent interest.

\medskip
In case $\sM$ is an embedded hypersurface in $\bR^{n+1}$ without boundary,
the motion of an incompressible fluid has been considered in the literature by several authors.
Here we refer to the article  \cite{CCD17} for a survey and a comprehensive list of references.
We also mention that the equations in \eqref{NS-sys}$_1$ and \eqref{NS-sys}$_2$
 coincide with the system
 \begin{equation*}
\left\{\begin{aligned}
\partial_t u + \cP_\sM( u \cdot \nabla_\sM u) -\cP_\sM \Div_\sM (2\mu_s \cD_\sM(u)-\pi \cP_\sM) &=0 &&\text{on}&&\sM ,\\
\Div_\sM u &=0 && \text{on}&&\sM ,\\
\end{aligned}\right.
\end{equation*}
considered in \cite{PrSiWi21, SiWi22}, see for instance \cite[Remarks A.3]{PrSiWi21}.
In addition, we mention the publications~\cite{JaOlRe17, OQRY18, ReZh13,ReVo18} and the references
contained therein for interesting numerical investigations for embedded surfaces in $\bR^3$ without boundary.

\medskip
For the case of a Riemannian manifold with boundary, we are aware of the publications \cite{Pri94,MiMo09}.
The author in \cite{Pri94} considers Navier boundary conditions, and  he examines the equations in a variational framework,
mostly concentrating on the stationary linear case.
In \cite{MiMo09}, the authors show that the Hodge-Laplacian subject to Neumann-type boundary conditions on a Lipschitz subdomain
of a smooth, compact, boundaryless Riemannian manifold generates an analytic semigroup
on $L_q$ for $q$ in some open interval containing $(3/2,3)$.

In case of a domain contained in Euclidean space, the Navier-Stokes equations with Navier boundary conditions have been considered
by numerous  authors, and we refer to \cite{PrWi18} for a discussion.

\medskip
The novelty of this manuscript lies in the fact that we consider the behavior of fluids on surfaces, or manifolds, with boundaries.
This situation extends traditional fluid dynamics analysis, which typically focuses on the Euclidean space. By studying fluids on manifolds, we are addressing a more complex scenario that also has applications.

For instance, this  is  relevant when analyzing the motion of water on a  planet that is covered by both oceans and continents. In such a context, the surface of the planet can be modeled as a manifold with boundaries, representing land  and sea.

In this situation, the analysis becomes considerably more complex than in the Euclidean case. Unlike in flat space, one must account for the manifold's geometric properties, which introduce additional mathematical challenges. Specifically, we need to handle geometric quantities such as the Ricci curvature,
which incorporates how the manifold's shape deviates from being flat. These geometric considerations play a role for describing the behavior of fluids on curved surfaces.

When dealing with an impermeable boundary  $\Sigma$, the most widely employed boundary condition in the literature is the no-slip condition, expressed as
\begin{equation}
\label{Dirichlet BC}
u=0 \quad \text{on}\quad \Sigma.
\end{equation}
In contrast, the Navier boundary condition \eqref{NS-sys}$_3$ and \eqref{NS-sys}$_4$ permits tangential slip along the boundary.
Over recent decades, a growing debate has emerged concerning the choice between the no-slip  condition and the Navier  condition, primarily due to the so-called no-collision paradox.
Consider a rigid body in free fall within a fluid bounded by a solid wall.
 In case the rigid body and the wall have a smooth boundary, previous research \cite{GeradHill, Hill07, HillTaka} has demonstrated that under the assumption  \eqref{Dirichlet BC}, the rigid body does not reach the fluid-solid interface in finite time, regardless of the relative densities of the fluid and the object. In contrast, assuming a Navier boundary condition circumvents such a situation \cite{GeradHillWang}.

Although we would expect similar results for the no-slip boundary conditions \eqref{Dirichlet BC} as for the case of partial slip with $\alpha>0$,
the  approach used here does not cover  \eqref{Dirichlet BC}.

\bigskip\noindent
{\bf Notation.}
Given $q\in (1,\infty)$, $q'=q/(q-1)$ always denotes the H\"older conjugate of $q$.

Let $X$ and $Y$ be two Banach spaces and $T: X\to Y$. We denote by $\sD(T)$, $\sN(T)$ and $\sR(T)$ the domain, null space and   range of $T$, respectively.
The notation $\cL(X,Y)$ stands for the set of all bounded linear operators from $X$ to $Y$ and $\cL(X):=\cL(X,X)$.
$\Lis(X,Y)$ denotes the subset of $\cL(X,Y)$ consisting of linear isomorphisms from $X$ to $Y$.
Moreover, we denote by $X^\prime=\cL(X,\bR)$ the dual of $X$.

For any $0\leq t_1<t_2<\infty$, $p\in (1,\infty)$ and $\mu\in (1/p,1]$, the $X$-valued $L_p$-spaces with temporal weight are defined by
$$
L_{p,\mu}((t_1,t_2);X):=\left\{ f: (t_1,t_2)\to X: \, t\mapsto t^{1-\mu}f(t)\in L_p((t_1,t_2);X)  \right\}.
$$
Similarly,
$$
H^k_{p,\mu}((t_1,t_2);X):=\left\{ f \in    W^k_{1,loc}((t_1,t_2);X):\,  \partial_t^j f\in L_{p,\mu}((t_1,t_2);X), \, j=0,1,\ldots,k \right\}.
$$

\section{The surface Stokes operator with Navier boundary conditions}
\label{surface Stokes operator}
To analyze \eqref{NS-sys},
we introduce the {\em surface Helmholtz projection}, defined by
$$
\PH u= u-\gd \psi_u, \quad u\in L_q(\sM; T\sM),
$$
where $\gd\psi_u \in L_q(\sM; T\sM)$ is the unique solution of
$$
(\gd \psi_u | \gd \phi)_{\sM}=(u|\gd \phi)_{\sM} ,\quad \forall \phi\in \dot{H}^1_{q'}(\sM),
$$
cf. Lemma~\ref{Apendix Lemma Helmholtz}.
Here,
$$
 (u | v)_\sM:=  \int_\sM (u|v)_g \, d\mu_g,\quad (u,v)\in L_r(\sM;T\sM)\times L_{r^\prime}(\sM; T\sM),
$$
denotes the  duality pairing between $L_q(\sM; T\sM)$ and $L_{q^\prime}(\sM; T\sM)$.
We note that  in case $q=2$,  the pairing $(\cdot | \cdot )_\sM$ defines an inner product on $L_2(\sM; T\sM)$.

\medskip\noindent
For any $u\in L_q(\sM; T\sM)$ and $v\in L_{q'}(\sM; T\sM)$ it holds
\begin{equation}
\label{PH-symmetric}
\begin{aligned}
(\PH u | v)_{\sM}&= (u - \gd \psi_u | v)_{\sM} = (u|v)_{\sM} - (\gd \psi_u | v)_{\sM}\\
&= (u|v)_{\sM} - (\gd \psi_u | \gd \psi_v)_{\sM}= (u|v)_{\sM} - (u | \gd \psi_v)_{\sM}\\
&= (u| \PH v)_{\sM}
\end{aligned}
\end{equation}
as $\psi_u\in \dot{H}^1_q(\sM)$ and $\psi_v \in \dot{H}^1_{q'}(\sM)$.
Note that the definition of $\PH$ implies
\begin{equation}
\label{sigma-u-nu}
 \text{$(u|\nu_\Sigma)_g=0$ on $\Sigma$} \ \ \text{in case}\ \  u\in H^s_{q,\sigma}(\sM;T\sM) \text{ and }s>1/q.
\end{equation}
With these preparations, we can introduce the function spaces used in this article
\begin{align}
\begin{split}\label{Def sigma spaces}
L_{q,\sigma}(\sM;T\sM):&= \PH L_q(\sM;T\sM) \\
H^s_{q,\sigma}(\sM;T\sM): &= H^s_q(\sM;T\sM) \cap L_{q,\sigma}(\sM;T\sM)  \\
B^s_{qp,\sigma}(\sM;T\sM): &= B^s_{qp}(\sM;T\sM) \cap L_{q,\sigma}(\sM;T\sM)   \\
H^{-s}_{q,\sigma}(\sM;T\sM): &=  (H^s_{q',\sigma}(\sM;T\sM))' \\
B^{-s}_{qp,\sigma}(\sM;T\sM): &=  (B^s_{q' p',\sigma}(\sM;T\sM))'
\end{split}
\end{align}
for   $s\geq 0$ and $1<p,q<\infty$,
where the respective duality parings
\begin{equation*}
\begin{aligned}
&\la \cdot | \cdot \ra_\sM:  H^{-s}_{q,\sigma}(\sM;T\sM)\times H^s_{q',\sigma}(\sM;T\sM)) \to \bR, \\
&\la \cdot | \cdot \ra_\sM: B^{-s}_{qp,\sigma}(\sM;T\sM)\times B^s_{q' p',\sigma}(\sM;T\sM)\to \bR,
\end{aligned}
\end{equation*}
are induced by $(\cdot | \cdot )_\sM$.
We would like to point out that our definition of the `negative' spaces $H^{-s}_q$ and $B^{-s}_{qp}$
differs from the usual definition in case $-s< -1/{q'}$. This allows for a more streamlined presentation of our results.
 As the spaces involved will be clear from the context they will not be explicitly
referenced in our notation  $\la \cdot | \cdot \ra_\sM$. Note that
\begin{equation*}
\la u | v \ra_\sM = (u |v)_\sM \quad\text{in case}\quad (u,v)\in L_q(\sM;T\sM)\times L_{q^\prime}(\sM; T\sM).
\end{equation*}

Now we can define the {\em strong surface Stokes operator with Navier boundary conditions}, $A_N: X_1\to X_0$, by
\begin{equation}
\label{Strong-Stokes-Navier}
A_N u:=-2\mu_s \PH\, \dv D(u)= -\mu_s \PH (\Delta_\sM u + \Ric^\sharp u )
\end{equation}
with $X_0:= L_{q,\sigma}(\sM;T\sM)$ and
\begin{equation}
\label{D(AN)}
X_1:=\sD(A_N):=\{u\in H^2_{q,\sigma}(\sM;T\sM):  (u|\nu_\Sigma)_g=0,\ \ \alpha u + \cP_\Sigma \left( (\nabla u + [\nabla  u]^{\sT} )  \nu_\Sigma \right)=0 \text{ on } \Sigma\}.
\end{equation}
Although the condition $(u | \nu_\Sigma)_g=0$ is already contained in the stipulation $u\in H^2_{q,\sigma}(\sM;T\sM)$, see \eqref{sigma-u-nu},
we include it in the definition for extra emphasis.

Next, we will derive a simpler expression of the boundary conditions of \eqref{NS-sys}.
We first note that in local coordinates
\begin{equation}
\label{boundary condition normal}
\nu_\Sigma=\sum_{j=1}^n \frac{1}{\sqrt{g^{nn}}}g^{nj} \frac{\partial}{\partial x^j}=\frac{1}{\sqrt{g^{nn}}}g^{nj} \frac{\partial}{\partial x^j}.
\end{equation}
In addition, we set
 $\displaystyle\cP_\Sigma=I_{T \sM} - \frac{1}{g^{nn}} g^{nj} \frac{\partial}{\partial x^j } \otimes dx^n.$
Hence,
\begin{equation}
\label{P-Sigma-properties}
\cP_\Sigma \frac{\partial}{\partial x^i}=  \frac{\partial}{\partial x^i},\quad i=1, \ldots  ,n-1,
\qquad  \cP_\Sigma \, \left(  g^{nj}\frac{\partial}{\partial x^j} \right)=0.
\end{equation}
Then we have for any $u\in H^1_{q,\sigma}(\sM;T\sM)$,
using the metric property  of $(\cdot | \cdot)_g$,
 \eqref{P-Sigma-properties},
the boundary condition $(u|\nu_\Sigma)_g=0$, and \eqref{ST}
\begin{align}
\cP_\Sigma \left( [\nabla u]^\sT \nu_\Sigma \right)
\notag
& =  \cP_\Sigma \left( g^\sharp(dx^i \otimes \nabla_i u )  g_\flat \nu_\Sigma \right)
= \cP_\Sigma \left(  g^\sharp dx^i ( \nabla_i u |  \nu_\Sigma)_g \right) \\
\notag
& = \sum_{i,j=1}^n\left[ \nabla_i(  u | \nu_\Sigma)_g -(  u | \nabla_i\nu_\Sigma)_g \right]  \, \cP_\Sigma \, g^{ij}\frac{\partial}{\partial x^j} \\
\notag
& = \sum_{i=1}^{n-1}\sum_{j=1}^n\left[ \nabla_i(  u | \nu_\Sigma)_g -(  u | \nabla_i\nu_\Sigma)_g \right]  \, \cP_\Sigma \, g^{ij}\frac{\partial}{\partial x^j} \\
\label{boundary condition 1}
& = \sum_{i=1}^{n-1}\sum_{j=1}^n ( L_\Sigma u \big| \frac{\partial}{\partial x^i})_g \, \cP_\Sigma \, g^{ij}\frac{\partial}{\partial x^j} \\
\notag
& = \sum_{i,j=1}^n ( L_\Sigma u | \frac{\partial}{\partial x^i})_g \, \cP_\Sigma \, g^{ij}\frac{\partial}{\partial x^j}
 = \sum_{i,j=1}^n ( L_\Sigma u | g^{ij}\frac{\partial}{\partial x^i})_g \, \cP_\Sigma \, \frac{\partial}{\partial x^j} \\
\notag
& = \sum_{j=1}^n \la dx^j , L_\Sigma u\ra_{g_\Sigma}  \, \cP_\Sigma \, \frac{\partial}{\partial x^j}
 = \sum_{j=1}^{n-1} \la  dx^j, L_\Sigma u\ra_{g_\Sigma}   \frac{\partial}{\partial x^j}
 =L_\Sigma u ,
\end{align}
where $L_\Sigma$ is the  Weingarten tensor induced by $g|_\Sigma$, with $L_\Sigma u\in \Gamma(\Sigma, T\Sigma)$, and
$$
\la{\cdot \, , \cdot}\ra_{g_\Sigma}:T^*\Sigma \times T \Sigma  \rightarrow \bR^\Sigma \quad\text{is the (fiber-wise defined) duality pairing on $\Sigma$.}
$$
Moreover,
\begin{align}
\label{boundary condition 2}
\begin{split}
\cP_\Sigma \left(  \nabla u  \nu_\Sigma \right)&= \frac{1}{\sqrt{g^{nn}}}g^{nj} \cP_\Sigma \left(  \nabla_j u \right)
= \frac{1}{\sqrt{g^{nn}}}g^{nj} \cP_\Sigma \left( (\partial_j u^i + \Gamma^i_{kj} u^k)\frac{\partial}{\partial x^i} \right) \\
&= \frac{g^{nj}}{\sqrt{g^{nn}}} (\partial_j u^i + \Gamma^i_{kj} u^k)\frac{\partial}{\partial x^i}
 - \frac{g^{nj} g^{ni} }{( g^{nn})^{3/2}}(\partial_j u^n + \Gamma^n_{kj} u^k)\frac{\partial}{\partial x^i}\\
&= \sum\limits_{i=1}^{n-1}\left[ \frac{g^{nj}}{\sqrt{g^{nn}}} (\partial_j u^i + \Gamma^i_{kj} u^k)
 - \frac{g^{nj} g^{ni} }{( g^{nn})^{3/2}}(\partial_j u^n + \Gamma^n_{kj} u^k) \right] \frac{\partial}{\partial x^i}  .
\end{split}
\end{align}
By  \eqref{PH-symmetric}, \eqref{boundary condition 1}, and Lemma~\ref{Appendix Lem: divergence thm}(b)(ii),
for any $u\in \sD(A_N)$ and $v\in H^1_{q',\sigma}(\sM;T\sM)$,
\begin{align}
\begin{split}\label{weak formulation Navier 1}
 (A_N u | v)_{\sM}
& = \mu_s (\nabla u   | \nabla v)_{\sM} -\mu_s (\Ric^\sharp\, u   |  v)_{\sM} - \mu_s (  \nabla u    \nu_\Sigma  | v)_{\Sigma}  \\
& = \mu_s (\nabla u   | \nabla v)_{\sM} -\mu_s (\Ric^\sharp\, u   |  v)_{\sM} -   \mu_s (\cP_\Sigma (\nabla u )  \nu_\Sigma  | v)_{\Sigma} \\
&=  \mu_s (\nabla u   | \nabla v)_{\sM} -\mu_s (\Ric^\sharp\, u   |  v)_{\sM} +  ( \alpha \mu_s u +  \mu_s L_\Sigma u |   v)_{\Sigma},
\end{split}
\end{align}
where   $(\cdot | \cdot)_{\Sigma}$ denotes the duality pairing between $L_q(\Sigma, T\Sigma)$ and $L_{q\prime} (\Sigma; T\Sigma)$.
By setting
\begin{equation}
\label{X12}
X_{1/2}:=H^1_{q,\sigma}(\sM;T\sM) \quad \text{and}\quad X_{-1/2}:=\left( H^1_{q',\sigma}(\sM;T\sM) \right)'=:H^{-1}_{q,\sigma}(\sM; T\sM),
\end{equation}
the above computations motivate  us to define the {\em weak surface Stokes operator with Navier boundary conditions} $\Aw_N: X_{1/2} \to X_{-1/2}$ by
\begin{align*}
 \la \Aw_N u | v\ra_{\sM} =  \mu_s (\nabla u   | \nabla v)_{\sM} -\mu_s (\Ric^\sharp u   |  v)_{\sM} +  ( \alpha \mu_s u +  \mu_s L_\Sigma u |   v)_{\Sigma}
\end{align*}
for all $(u,v)\in X_{1/2} \times (X_{-1/2})'= H^1_{q,\sigma}(\sM;T\sM) \times H^1_{q',\sigma}(\sM;T\sM)$.

\medskip\noindent
The next result states that the surface Stokes operators $\Aw_N$ and $A_N$ both admit a bounded
$H^\infty$-calculus.
\begin{theorem}
\label{thm:Stokes-H-infinty}
There exists a number $\omega_0>0$ such that
\begin{itemize}
\item[{\rm(a)}]
$\omega+ \Aw_N \in H^\infty(X_{-1/2}) \text{ with $H^\infty$-angle }  < \pi/2 \quad \text{for all  }\omega>\omega_0.$
\vspace{1mm}
\item[{\rm(b)}]
$\omega+ A_N \in H^\infty(X_0) \text{ with $H^\infty$-angle } < \pi/2 \quad \text{for all  }\omega>\omega_0.$
\end{itemize}
\end{theorem}
\begin{proof}
For a proof we refer to Section \ref{Stokes-H-infity}.
\end{proof}

\section{$H^\infty$-calculus of surface Stokes operator with perfect slip boundary conditions}
\label{Section:NS with perfect slip boundary condition}
In order to prove Theorem~\ref{thm:Stokes-H-infinty}, we take a detour and first consider the Stokes operator with perfect
slip boundary conditions, which is also of interest in its own right.
On a technical level, we are aided by the fact that in this case, the Helmholtz projection commutes with the Laplacian, which (temporarily) allows us to
ignore the pressure and the divergence condition.  The same strategy was also employed in \cite{PrWi18}.

 \medskip\noindent
We start this section by providing the necessary tools to localize differential equations and tangent fields that are defined on the manifold $\sM$.
\medskip
We define an atlas $\{(U_k, \varphi_k)\}_{k\in \cK}$  of  $\sM$  with $\cK=\cK_0\sqcup \cK_1$ such that
$k\in \cK_0$ if $U_k\cap \Sigma=\emptyset$ and $k\in \cK_1$ if $U_k\cap \Sigma \neq \emptyset$. Moreover,
$$
\varphi_k(U_k)=
\bB^n_k(0,R)
:=
\begin{cases}
\bB^n(0,R), \quad & \text{if } k\in \cK_0, \\
\bB^n(0,R) \cap \bR^n_+, \quad & \text{if } k\in \cK_1 .
\end{cases}
$$
Given $k\in\cK$, we set
\begin{align*}
\bX_k =
\begin{cases}
\bR^n, \quad & \text{if } k\in \cK_0, \\
\bR^n_+, \quad & \text{if } k\in \cK_1 ,
\end{cases}
\end{align*}
endowed with the Euclidean metric in $\bR^n$.
Let $\{\xi^2_k\}_{k\in \cK}$ be a partition of unity  subject to $\{ U_k \}_{k\in \cK}$.
Furthermore, let  $\zeta\in C^\infty_0(\bB^n(0,R);[0,1])$ be chosen such that
$$
\zeta  \equiv 1 \quad \text{on }{\rm supp}((\varphi_k)_* \xi_k)  \quad \text{for all } k\in \cK,
$$
where   $(\varphi_k)_* \phi :=\phi\circ \varphi_k^{-1}$  is the pushforward of a function
 $\phi: {\sM} \to \bR $ by $\varphi_k$.
Given $u\in \Gamma(\sM; T\sM)$, we define
\begin{align*}
(\varphi_k)_*u:= ((\varphi_k)_* u^i)_{1\le i\le n}, \quad\text{ where } u=u^i \frac{\partial}{\partial x^i}.
\end{align*}
For $\fF\in \{H_q ,W_q \}$, $1<q<\infty$, and $s\ge 0$, we define:
\begin{equation*}
\begin{aligned}
\Rck &: \fF^s( \sM ;T \sM)\rightarrow \fF^s( \bX_k; \bR^n), \quad  && u\mapsto  (\varphi_k)_*({\xi_k }u) ,\\
\Rek &: \fF^s({\bX_k}; \bR^n)\rightarrow \fF^s({\sM};T\sM), \quad && u_{\kappa}\mapsto{\xi_k }  (\varphi^*_k u_{\kappa}).
\end{aligned}
\end{equation*}
Here and in the following, it is understood that a partially defined and compactly supported vector field is
automatically extended over the whole base manifold by identifying it to be the zero section outside its original domain.
\\
With a slight abuse of notation, we define the pullback of a  vector field $v: \bX_k\to \bR^n$  by  means of
$$
\varphi_k^* v := \left(v^i \circ \varphi_k  \right) \frac{\partial}{\partial x^i}.
$$
Finally, we define
\begin{equation*}
\begin{aligned}
&\Rc: \fF^s( \sM ;T\sM )\rightarrow \boldsymbol{\fF}^s(\bX; \bR^n),\quad u\mapsto (\Rck u)_{k\in \cK}, \\
&\Re:\ \boldsymbol{\fF}^s(\bX; \bR^n)\rightarrow \fF^s({\sM};T\sM),\quad \boldsymbol{v}=(v_k)_{k\in \cK}\mapsto \sum_k\Rek v_k
\end{aligned}
\end{equation*}
with $\boldsymbol{\fF}^s(\bX;\bR^n):=\prod_{k\in \cK}\fF^s({\bX_k};\bR^n)$, equipped with the norm
$$
\| \boldsymbol{v} \|_{\boldsymbol{\fF}^s}=  \sum_{k\in \cK}\|v_k\|_{\fF^s(\bX_k)}  , \quad \boldsymbol{v}=(v_k)_{k\in \cK}.
$$
Then one shows that
\begin{equation*}
\Rc\in \cL(\fF^s( \sM ;T\sM), \boldsymbol{\fF}^s(\bX; \bR^n)),  \quad
\Re \in \cL(\boldsymbol{\fF}^s(\bX; \bR^n), \fF^s( \sM ;T\sM)),
\end{equation*}
see for instance \cite{Ama13}.
Moreover,
\begin{equation*}
(\Re \circ \Rc) u= u,\quad u\in \fF^s( \sM ;T\sM),
\end{equation*}
that is,
$\Re$ is a retraction from $\boldsymbol{\fF}^s(\bX; \bR^n)$ onto ${\fF}^{s}({\sM};T\sM)$, and $\Rc$ is a coretraction.

\subsection{Strong formulation}\label{Section:NS with perfect slip boundary condition strong}
Following the ideas of \cite{PrWi18}, we will first study the Stokes operator with {\em perfect slip boundary conditions}.
To this end,
we consider first the elliptic boundary value problem,
\begin{equation}
\label{perfect slip-elliptic}
\left\{\begin{aligned}
(\lambda -  \Delta_\sM  +\Ric^\sharp ) u   &=f &&\text{on}&&\sM ,\\
 \cP_\Sigma \left( (\nabla u -[\nabla  u]^{\sT} )  \nu_\Sigma \right)   &=h_1 &&\text{on}&&\Sigma  , \\
(u| \nu_\Sigma)_g  &=h_2 &&\text{on}&&\Sigma,
\end{aligned}\right.
\end{equation}
for suitable $\lambda\in \bC$ and
$$(f,h_1,h_2)\in L_{q}(\sM;T\sM)\times W_q^{1-1/q}(\Sigma;T\Sigma)\times W_q^{2-1/q}(\Sigma).$$

\noindent
We should like to briefly explain our rationale for using the terminology \emph{perfect slip boundary conditions}.
In three-dimensional Euclidean space, it can be shown that
$$\cP_\Sigma \left( (\nabla u -[\nabla  u]^{\sT} )  \nu_\Sigma \right) = \nu_\Sigma \times  {\rm curl}\, u,  $$
see for instance \cite[Section 4.1]{PrSiWi18}.
In applications in (magneto)hydrodynamics, the boundary conditions
$$(u | \nu_\Sigma)_g=0,\quad  {\rm curl}\, u \times \nu_\Sigma =0,$$
 are sometimes referred to as \emph{perfect wall conditions,} see for instance \cite{AKST04}.
In addition, these conditions are also known as \emph{Neumann boundary conditions} or  \emph{free boundary conditions,}
see for instance \cite{MiMo09, MiMo09b} and the reference therein.
For lack of a better name and following  \cite{PrSiWi18, PrWi18}, we will use the same terminology also in
the general situation of manifolds of arbitrary dimension.

\medskip
We have the following result about existence and uniqueness of solutions to \eqref{perfect slip-elliptic}.
\begin{proposition}
\label{pro:wellposedPS}
Let $1<q<\infty$ and $\phi\in (0,\pi/2)$. Then, there exists a number $\lambda_0>0$ such that for all $\lambda \in \lambda_0+\Sigma_{\pi-\phi}$ problem \eqref{perfect slip-elliptic} has a unique solution
$u\in H_q^2(\sM;T\sM)$ if and only if
$$(f,h_1,h_2)\in L_{q}(\sM;T\sM)\times W_q^{1-1/q}(\Sigma;T\Sigma)\times W_q^{2-1/q}(\Sigma).$$
Furthermore, there exists a constant $C>0$ such that for all $\lambda \in \lambda_0+\Sigma_{\pi-\phi}$ the estimate
\begin{multline}
\label{resolvent estimate Lps}
|\lambda| \|u\|_{L_q(\sM)} + \|u\|_{H^2_q(\sM)}
\leq   C   \Big( \|f\|_{L_q(\sM)} + \|H_1\|_{H^1_q(\sM)}  + |\lambda|^{1/2}\|H_1\|_{L_q(\sM)}  \\
 +\|{H}_2\|_{H^{2}_q(\sM)} + |\lambda|^{1/2}\|{H}_2\|_{H_q^1(\sM)} + |\lambda|\|{H}_2\|_{L_q(\sM)} \Big)
\end{multline}
holds, where $H_j$ is any extension of ${h}_j$ from $W_q^{j-1/q}(\Sigma)$ to $H_q^j(\sM)$.
\end{proposition}
\begin{proof}
In short form,  \eqref{perfect slip-elliptic} can be formulated as
\begin{equation}\label{eq:perfect slip-elliptic-abstract0}
L_\lambda u=F,
\end{equation}
where $L_\lambda: H_q^2(\sM;T\sM)\to L_{q}(\sM;T\sM)\times W_q^{1-1/q}(\Sigma;T\Sigma)\times W_q^{2-1/q}(\Sigma)$ is defined by the left side of \eqref{perfect slip-elliptic} and $F:=(f,h_1,h_2)$.

\medskip\noindent
In the following, we will show that the operator $L_\lambda$ is invertible for $\lambda$ appropriately chosen.
We start by establishing a priori estimates for solutions of \eqref{perfect slip-elliptic}.
Suppose $u\in H^2_q(\sM; T\sM)$ is a solution of \eqref{perfect slip-elliptic}. We then set
\begin{equation*}
\bar{u}_k:=\Rck u =(\bar{u}_k^1, \bar{u}_k^2,\cdots,\bar{u}_k^n)^\sT
\end{equation*}
and
\begin{equation*}
\bar{G}_{(k)}= [\bar{g}^{ij}_{(k)}]^i_j = \zeta  G_k + (1-\zeta) I_n,
\end{equation*}
where $G_k := [\kbk g^{ij}]^i_j$.
Using these notations  and \eqref{Bochner-Ricci-local}, we can write the first line in \eqref{perfect slip-elliptic} in local coordinates as
\begin{equation}
\label{localization of perfect slip-elliptic equation}
(\lambda - \bar{g}^{ij}_{(k)} \partial_i \partial_j ) \bar{u}_k     =\bar{f}_k + P_k  (u)  \quad \text{in } \bX_k ,
\end{equation}
where   the matrices $\bar{G}_{(k)} $ belong to  $BC^\infty(\bX_k; \bR^{n\times n})$  and $\bar{f}_k:= \Rck f$.
Up to translations and rotations, {$\|\bar{G}_{(k)}-I_n\|_{\infty}$  can be made arbitrarily small by shrinking the radius  $R>0$ of $\bB^n_k(0,R)$}.
The linear operator $P_k$ is of first order; in particular
$$
P_k\in \cL(H^1_q(\sM;T\sM), L_q(\bX_k;\bR^n)).
$$
Next, we will localize the boundary conditions in \eqref{perfect slip-elliptic}.
First, in view of  \eqref{boundary condition normal}, the boundary condition $(u|\nu_\Sigma)_g=h_2$   can be restated as
\begin{equation}
\label{boundary condition normal local}
 \frac{1}{\sqrt{\bar{g}_{(k)}^{nn}}}\bar{u}_{k}^n=\bar{h}_{2,k}  \quad \text{on }   \bR^{n-1},\quad k\in \cK_1.
\end{equation}
Using \eqref{boundary condition 1} and \eqref{boundary condition 2}, the remaining boundary condition $\cP_\Sigma \left( (\nabla u - [\nabla  u]^{\sT} )  \nu_\Sigma \right)    =h_1 $ can be rewritten as
\begin{equation}
\label{boundary condition perfect slip local}
  \frac{1}{\sqrt{\bar{g}_{(k)}^{nn}}}\bar{g}_{(k)}^{nj} \partial_j \bar{u}_k^i - \frac{\bar{g}_{(k)}^{nj} \bar{g}_{(k)}^{ni} }{( \bar{g}_{(k)}^{nn})^{3/2}}\partial_j \bar{u}_k^n
  =\bar{h}_{1,k}^{i}+ {\rm tr}_{\bR^{n-1}} Q_k^i (u)   \quad \text{on }    \bR^{n-1},
\end{equation}
for $ k\in \cK_1,\, i=1,2,\ldots,n-1$.  We note here that
$Q^i_k(u)$, in particular, contain an extension of the (localized) term $L_\Sigma u$ in \eqref{boundary condition 1} to
$H^{2}_q(\bX_k )$.
It follows that $Q_k^i\in \cL(H^2_q(\sM;T\sM), H^{2}_q(\bX_k ))$ with
$$\|Q_k^i(u)\|_{H^{s}_q(\bX_k )}\le C(s)\|u\|_{H^s_q(\sM;T\sM)},\quad u\in H^2_q(\sM;T\sM),$$
for any $s\in [0,2]$.
We define
$$L_{\lambda,k}^\#: H_q^2(\bX_k;\bR^n)\to {L_q}(\bX_k;\bR^n)$$
for $k\in\cK_0$ by $L_{\lambda,k}^\# v :=
 \lambda v-\bar{g}^{ij}_{(k)} \partial_i \partial_j v$
and
$$L_{\lambda,k}^\#: H_q^2(\bX_k;\bR^n)\to {L_q}(\bX_k;\bR^n)\times W^{1-1/q}_q(\bR^{n-1};\bR^{n-1}) \times W^{2-1/q}_q(\bR^{n-1}) $$
for $k\in\cK_1$ by
\begin{equation*}
 L_{\lambda,k}^\# v :=
\left(
\begin{array}{l}
 \lambda v-\bar{g}^{ij}_{(k)} \partial_i \partial_j v\\
 {\rm tr}_{\bR^{n-1}}(\frac{1}{\sqrt{\bar{g}_{(k)}^{nn}}}\bar{g}_{(k)}^{nj}  \partial_j v^1 -   \frac{\bar{g}_{(k)}^{nj} \bar{g}_{(k)}^{n1} }{( \bar{g}_{(k)}^{nn})^{3/2}}\partial_j v^n )   \\
 \hspace{2cm} \vdots \\
 {\rm tr}_{\bR^{n-1}}( \frac{1}{\sqrt{\bar{g}_{(k)}^{nn}}}\bar{g}_{(k)}^{nj} \partial_j v^{n-1} -  \frac{\bar{g}_{(k)}^{nj} \bar{g}_{(k)}^{n,n-1} }{( \bar{g}_{(k)}^{nn})^{3/2}}{\partial_j} v^n )   \\
   {\rm tr}_{\bR^{n-1}} \frac{1}{\sqrt{\bar{g}_{(k)}^{nn}}}v^n
 \end{array}
 \right),
\end{equation*}
where ${\rm tr}_{\bR^{n-1}}$ is the trace operator from $\bR^n_+$ to $\bR^{n-1}$.

Let us denote by $L_{\lambda,k}^{\#,0}$ the corresponding operator in the planar case, i.e. $G_k=I_n$. Then it holds that $L_{\lambda,k}^{\#,0} v=\lambda v-\Delta_{\bR^n} v$ if $k\in\cK_0$ and
$$
L_{\lambda,k}^{\#,0} v= \left(\lambda v-\Delta_{\bR^n_+} v,
 {\rm tr}_{\bR^{n-1}}\partial_n v^1 , \cdots ,   {\rm tr}_{\bR^{n-1}}\partial_n v^{n-1}, {\rm tr}_{\bR^{n-1}}  v^n \right)^{\sf T}
$$
if $k\in\cK_1$. It is well-known that for each $\lambda \in \Sigma_{\pi-\phi}$, $\phi\in (0,\pi/2)$, the operators $L_{\lambda,k}^{\#,0}$ are
isomorphisms between the corresponding spaces defined above. It follows from \cite[Theorem 3.1.3]{Lun95} that there exists a constant $C>0$ being independent of $\lambda$, such that
the unique solution $v$ to the elliptic problems
$$L_{\lambda,k}^{\#,0} v=\bar{f}$$
for $k\in \cK_0$, $\bar{f}\in L_q(\bR^n;\bR^n)$ and
$$L_{\lambda,k}^{\#,0} v=(\bar{f},\bar{h}_1^1,\ldots,\bar{h}_1^{n-1},\bar{h}_2)^{\sf T}$$
for $k\in \cK_1$ and
$$(\bar{f},\bar{h}_1,\bar{h}_2)\in L_q(\bR^n_+;\bR^n)\times W^{1-1/q}_q(\bR^{n-1};\bR^{n-1}) \times W^{2-1/q}_q(\bR^{n-1})$$
satisfies
\goodbreak
\begin{multline}
\label{resolvent estimate Lk}
 |\lambda| \|v\|_{L_q(\bX_k)} + \|v\|_{H^2_q(\bX_k)} \le C \Big( \|\bar{f}\|_{L_q(\bX_k)} +\|\bar{H}_1\|_{H^{1}_q(\bX_k)}  + |\lambda|^{1/2}\|\bar{H}_1\|_{L_q(\bX_k)}  \\
 +\|\bar{H}_2\|_{H^{2}_q(\bX_k)} + |\lambda|^{1/2}\|\bar{H}_2\|_{H_q^1(\bX_k)} + |\lambda|\|\bar{H}_2\|_{L_q(\bX_k)}
 \Big)
\end{multline}
for all $\lambda \in \Sigma_{\pi-\phi}$, $\phi\in (0,\pi/2)$ and any extension $\bar{H}_j$ of $\bar{h}_j$ from $W_q^{j-1/q}(\bR^{n-1})$ to $H_q^j(\bX_k)$. In case $k\in \cK_0$, the terms in
\eqref{resolvent estimate Lk} containing $\bar{H}_1$ and $\bar{H}_2$ are omitted.

We will show in the sequel, that \eqref{resolvent estimate Lk} still holds for the general geometry by means of a perturbation argument. To this end, we write $$L_{\lambda,k}^{\#}=L_{\lambda,k}^{\#,0}+L_{\lambda,k}^{\#}-L_{\lambda,k}^{\#,0},$$
wherefore the equation $L_{\lambda,k}^{\#}v=(\bar{f},\bar{h}_1^1,\ldots,\bar{h}_1^{n-1},\bar{h}_2)^{\sf T}$ is equivalent to
$$v+[L_{\lambda,k}^{\#,0}]^{-1}(L_{\lambda,k}^{\#}-L_{\lambda,k}^{\#,0})v=[L_{\lambda,k}^{\#,0}]^{-1}(\bar{f},\bar{h}_1^1,\ldots,\bar{h}_1^{n-1},\bar{h}_2)^{\sf T}.$$
In case $k\in\cK_0$, it holds that
$$L_{\lambda,k}^{\#}v-L_{\lambda,k}^{\#,0}v=\Delta_{\bR^n}v-\bar{g}^{ij}_{(k)} \partial_i \partial_j v,$$
and
\begin{align*}
\|\Delta_{\bR^n}v-\bar{g}^{ij}_{(k)} \partial_i \partial_j v\|_{L_q(\bX_k)}&\le \|\bar{G}_{(k)}-I_n\|_{\infty}\|v\|_{H^2_q(\bX_k)}\\
&\le \|\bar{G}_{(k)}-I_n\|_{\infty}\left(|\lambda| \|v\|_{L_q(\bX_k)} + \|v\|_{H^2_q(\bX_k)}\right),
\end{align*}
where we recall that $\|\bar{G}_{(k)}-I_n\|_{\infty}$  can be made as small as we wish. Therefore we may achieve
$$\|[L_{\lambda,k}^{\#,0}]^{-1}(L_{\lambda,k}^{\#,0}-L_{\lambda,k}^{\#})v\|_{H_{q,\lambda} ^2(\bX_k)}\le\frac{1}{2}\|v\|_{H_{q,\lambda} ^2(\bX_k)},$$
where $H_{q,\lambda} ^2(\bX_k)$ denotes the space $H_q^2(\bX_k) $ equipped with the norm $ |\lambda| \,\|\cdot \|_{L_q(\bX_k)} + \|\cdot\|_{H^2_q(\bX_k)}$.
A Neumann series argument then yields that the linear operator
$$I+[L_{\lambda,k}^{\#,0}]^{-1}(L_{\lambda,k}^{\#}-L_{\lambda,k}^{\#,0}):H_{q,\lambda} ^2(\bX_k)\to H_{q,\lambda} ^2(\bX_k)$$
is invertible and the estimate
$$\|v\|_{H_{q,\lambda} ^2(\bX_k)}\le 2C\|\bar{f}\|_{L_q(\bX_k)}$$
holds, whenever $k\in\cK_0$.

If $k\in\cK_1$, then perturbations on the boundary $\partial\bX_k=\mathbb{R}^{n-1}$ have to be taken into account. We show this exemplarily for the last component of $L_{\lambda,k}^{\#}v-L_{\lambda,k}^{\#,0}v$, given by
$$ {\rm tr}_{\bR^{n-1}} ((\bar{g}_{(k)}^{nn})^{-1/2}v^n-  v^n)\in W_q^{2-1/q}(\bR^{n-1}).$$
The term $((\bar{g}_{(k)}^{nn})^{-1/2}-1)v^n\in H_q^2(\bX_k)$ is an extension and will be estimated with respect to the norm
$$\|\cdot\|_{H^{2}_q(\bX_k)} + |\lambda|^{1/2}\|\cdot\|_{H_q^1(\bX_k)} + |\lambda|\|\cdot\|_{L_q(\bX_k)}.$$
For the sake of readability, we write $a=(\bar{g}_{(k)}^{nn})^{-1/2}- 1$ and we recall that $\|a\|_{L_\infty(\bX_k)}$ can be made as small as we wish, while $\|a\|_{W_\infty^2(\bX_k)}$ is bounded.
Then, it holds that
\begin{align*}
|\lambda|\|av^n\|_{L_q(\bX_k)}&\le |\lambda|\|a\|_{L_\infty(\bX_k)}\|v^n\|_{L_q(\bX_k)}\le |\lambda|\|a\|_{L_\infty(\bX_k)}\|v\|_{L_q(\bX_k)}\le \|a\|_{L_\infty(\bX_k)}\|v\|_{H_{q,\lambda} ^2(\bX_k)}
\end{align*}
by the definition of the norm in $H_{q,\lambda}^2(\bX_k)$.
Furthermore, we have
$$
\|av^n\|_{H_q^1(\bX_k)}
\le \|a\|_{W_\infty^1(\bX_k)}\|v\|_{L_q(\bX_k)}+\|a\|_{L_\infty(\bX_k)}\|v\|_{H_q^1(\bX_k)},
$$
where
$$\|v\|_{L_q(\bX_k)}\le |\lambda|^{-1}\|v\|_{H_{q,\lambda}^2(\bX_k)}.$$
We make use of complex interpolation
$$H_q^{1}(\bX_k)=[L_q(\bX_k),H_q^{2}(\bX_k)]_{1/2}$$
and Young's inequality to obtain
$$\|v\|_{H_q^1(\bX_k)}\le C|\lambda|^{-1/2}\|v\|_{H_{q,\lambda}^{2}(\bX_k)}.$$
This then implies that
$$
|\lambda|^{1/2}\|av^n\|_{H_q^1(\bX_k)}\le  C\left(\|a\|_{W_\infty^1(\bX_k)}|\lambda|^{-1/2}+\|a\|_{L_\infty(\bX_k)}\right)\|v\|_{H_{q,\lambda}^2(\bX_k)}.
$$
Finally, to estimate $av^n$ in $W_q^{2}(\bX_k)$, we observe
\begin{align*}
\|av^n\|_{H_q^2(\bX_k)}\le \|a\|_{L_\infty(\bX_k)}\|v\|_{H_q^2(\bX_k)}+\|a\|_{W_\infty^1(\bX_k)}\|v\|_{H_q^1(\bX_k)}+\|a\|_{W_\infty^2(\bX_k)}\|v\|_{L_q(\bX_k)}.
\end{align*}
By the estimates for $\|v\|_{H_q^s(\bX_k)}$, $s\in\{0,1\}$, from above, we obtain
$$\|av^n\|_{H_q^2(\bX_k)}\le C\left(\|a\|_{L_\infty(\bX_k)}+\|a\|_{W_\infty^1(\bX_k)}|\lambda|^{-1/2}+\|a\|_{W_\infty^2(\bX_k)}|\lambda|^{-1}\right)\|v\|_{H_{q,\lambda}^2(\bX_k)}.$$
This shows that, for any given $\eta>0$, choosing first $\|a\|_{L_\infty(\bX_k)}$ sufficiently small and then $|\lambda|$ sufficiently large, we may achieve that
$$\|av^n\|_{H^{2}_q(\bX_k)} + |\lambda|^{1/2}\|av^n\|_{H_q^1(\bX_k)} + |\lambda|\|av^n\|_{L_q(\bX_k)}\le\eta\|v\|_{H_{q,\lambda}^2(\bX_k)}.$$
The estimates for the remaining boundary conditions can be derived in the same spirit and are therefore omitted. By a Neumann series argument as in case $k\in\cK_0$, it follows that \eqref{resolvent estimate Lk} holds true for the general geometry, with a possibly larger constant $C>0$.

We split the solution $\bar{u}_k$ of \eqref{localization of perfect slip-elliptic equation}, \eqref{boundary condition normal local}, \eqref{boundary condition perfect slip local} into $\bar{u}_k=\tilde{u}_k+\hat{u}_k$ in such a way that $\tilde{u}_k$ solves
\begin{equation*}
 L_{\lambda,k}^\# \tilde{u}_k=(\bar{f}_k,\bar{h}_{1,k},\bar{h}_{2,k})^\sT
\end{equation*}
and $\hat{u}_k$ solves
\begin{equation*}
L_{\lambda,k}^\#  \hat{u}_k=(P_k(u), {\rm tr}_{\bR^{n-1}} Q_k (u), 0)^\sT, \quad Q_k(u)=(Q_k^1(u),\ldots,Q_k^{n-1}(u) )
\end{equation*}
if $k\in \cK_1$. For $k\in \cK_0$, we introduce a similar decomposition $\bar{u}_k=\tilde{u}_k+\hat{u}_k$  with
$$
L_{\lambda,k}^\#\tilde{u}_k=\bar{f}_k ,\quad  L_{\lambda,k}^\#\hat{u}_k=P_k(u).
$$
For the solution $u$ of \eqref{perfect slip-elliptic} we therefore obtain
\begin{align}
\label{resolvent exp Lps}
\begin{split}
u =  u_{(1)}+u_{(2)}  &:=   \Re \left(({\tilde u)}_{k\in\cK}\right) + \Re \left( ({\hat u)}_{k\in\cK}\right)  \\
 & = \sum\limits_{k\in \cK}  \xi_k  \kfk  \tilde{u}_k  +\sum\limits_{k\in \cK}  \xi_k  \kfk  \hat{u}_k.
\end{split}
\end{align}
Employing \eqref{resolvent estimate Lk} yields
\goodbreak
\begin{align*}
  |\lambda| \|u_{(1)}\|_{L_q(\sM)} &+ \|u_{(1)}\|_{H^2_q(\sM)}
 \\
 & \le C \sum_{k\in \cK} \|\bar{f}_k\|_{L_q(\bX_k)} +  C \sum_{k\in \cK_1}\Big( \|\bar{H}_{1,k}\|_{H^{1}_q(\bX_k)}  + |\lambda|^{1/2}\|\bar{H}_{1,k}\|_{L_q(\bX_k)} \\
 &
\quad +\|\bar{H}_{2,k}\|_{H^{2}_q(\bX_k)} + |\lambda|^{1/2}\|\bar{H}_{2,k}\|_{H_q^1(\bX_k)} + |\lambda|\|\bar{H}_{2,k}\|_{L_q(\bX_k)}
 \Big)\\
 &\le C \Big( \|f\|_{L_q(\sM)} + \|H_1\|_{H^{1}_q(\sM)}  + |\lambda|^{1/2}\|H_1\|_{L_q(\sM)}  \\
 &\quad +\|{H}_2\|_{H^{2}_q(\sM)} + |\lambda|^{1/2}\|{H}_2\|_{H_q^1(\sM)} + |\lambda|\|{H}_2\|_{L_q(\sM)} \Big)
\end{align*}
for all $\lambda \in \lambda_0+\Sigma_{\pi-\phi}$, where $\bar{H}_{j,k}\in H_q^j(\bX_k)$ are the localized versions of $H_j\in H_q^j(\sM)$. Here we have used the fact that if $H_j$ is any extension of ${h}_j$ from $W_q^{j-1/q}(\Sigma)$ to $H_q^j(\sM)$, then $\bar{H}_{j,k}$ is an extension of $\bar{h}_{j,k}$ from $W_q^{j-1/q}(\bR^{n-1})$ to $H_q^j(\bX_k)$.
\goodbreak
To estimate $ u_{(2)}$, note that \eqref{resolvent estimate Lk} implies
\begin{align}\label{eq:est_u2}
  |\lambda|& \|u_{(2)}\|_{L_q(\sM)} + \|u_{(2)}\|_{H^2_q(\sM)}\nonumber\\
 &\le C \sum_{k\in \cK} \|P_k(u)\|_{L_q(\bX_k)} +  C \sum_{k\in \cK_1}\left( \|Q_k(u)\|_{H^{1}_q(\bX_k)}  + |\lambda|^{1/2}\|Q_k(u)\|_{L_q(\bX_k)}  \right) \nonumber\\
 &\leq  C \left(  \| u\|_{H^1_q(\sM)}  +  |\lambda|^{1/2} \| u\|_{L_q(\sM)} \right).
\end{align}
By complex interpolation and Young's inequality, there exists a constant $C>0$ such that
\begin{equation}\label{eq:interpolation}
\|u\|_{H^{1}_q(\sM)}\le C\| u\|_{L_q(\sM)}^{1/2}\cdot \| u\|_{H^{2}_q(\sM)}^{1/2}\le |\lambda| ^{-1/2}C\left(|\lambda|\| u\|_{L_q(\sM)}+ \| u\|_{H^{2}_q(\sM)}\right).
\end{equation}
Furthermore,
$$|\lambda|^{1/2}\|u\|_{L_q(\sM)}\le |\lambda| ^{-1/2}\left(|\lambda|\| u\|_{L_q(\sM)}+ \| u\|_{H^{2}_q(\sM)}\right).$$
By possibly further increasing $\lambda_0>0$,  we can always achieve
$$
|\lambda| \|u_{(2)}\|_{L_q(\sM)} + \|u_{(2)}\|_{H^2_q(\sM)} \leq \frac{1}{2}\left(|\lambda| \|u \|_{L_q(\sM)} + \|u \|_{H^2_q(\sM)}\right)
$$
for all $\lambda \in \lambda_0+\Sigma_{\pi-\phi}$.
Combining with the estimate for $u_{(1)}$,
this  yields~\eqref{resolvent estimate Lps}
for all $\lambda \in \lambda_0+\Sigma_{\pi-\phi}$. This estimate implies in particular that the operator $L_\lambda$ defined in \eqref{eq:perfect slip-elliptic-abstract0} has a left inverse $S_\lambda$, provided $\lambda \in \lambda_0+\Sigma_{\pi-\phi}$.

We can even give an explicit formula for the left inverse $S_\lambda$. To this end, we use again \eqref{resolvent exp Lps}, i.e.
$$u= \sum\limits_{k\in \cK}  \xi_k  \kfk  \tilde{u}_k  +\sum\limits_{k\in \cK}  \xi_k  \kfk  \hat{u}_k$$
and define
$$H_\lambda^\ell u:=u_{(2)}=\sum\limits_{k\in \cK}  \xi_k  \kfk  \hat{u}_k.$$
 It follows from the considerations above  that the linear operator
$$
H_\lambda^\ell:H_{q,\lambda} ^2(\sM;T\sM)\to H_{q,\lambda}^2(\sM;T\sM),\quad u \mapsto u_{(2)},
$$  satisfies the norm estimate
$$\|H_\lambda^\ell\|\le\frac{1}{2},$$
provided $\lambda \in \lambda_0+\Sigma_{\pi-\phi}$,
where  $H_{q,\lambda} ^2(\sM;T\sM)$ denotes the space
 $H_q^2(\sM;T\sM) $ equipped with the norm $ |\lambda| \,\|\cdot \|_{L_q(\sM)} + \|\cdot\|_{H^2_q(\sM)}$.
By definition of $\tilde{u}_k$ we then obtain
$$u= \sum\limits_{k\in \cK}  \xi_k  \kfk  (L_{\lambda,k}^\#)^{-1}(\bar{f}_k,\bar{h}_{1,k},\bar{h}_{2,k})+H_\lambda^\ell u,$$
where $(\bar{f}_k,\bar{h}_{1,k},\bar{h}_{2,k})=\bar{f}_k$ if $k\in\cK_0$.
Therefore, it follows that
$$S_\lambda (f,h_1,h_2)=u=(I-H_\lambda^\ell)^{-1}\sum\limits_{k\in \cK}  \xi_k  \kfk  (L_{\lambda,k}^\#)^{-1}(\bar{f}_k,\bar{h}_{1,k},\bar{h}_{2,k}).$$
It remains to prove the existence of a right inverse for the operator $L_\lambda$ defined in  \eqref{eq:perfect slip-elliptic-abstract0}. To this end, let
$$(f,h_1,h_2)\in L_{q}(\sM;T\sM)\times W_q^{1-1/q}(\Sigma;T\Sigma)\times W_q^{2-1/q}(\Sigma)$$
be given and define $u:=S_\lambda (f,h_1,h_2)\in H_q^2(\sM;T\sM)$ with the left inverse $S_\lambda$ from above. In the sequel, we denote by
$$L_{\lambda,k}: H_q^2(\bX_k;\bR^n)\to L_q(\bX_k;\bR^n)$$
for $k\in\cK_0$
and by
$$L_{\lambda,k}: H_q^2(\bX_k;\bR^n)\to {L_q}(\bX_k;\bR^n)\times W^{1-1/q}_q(\bR^{n-1};\bR^{n-1}) \times W^{2-1/q}_q(\bR^{n-1}) $$
for $k\in\cK_1$ the full operator $L_\lambda$ from \eqref{eq:perfect slip-elliptic-abstract0} in local coordinates,
that is, $L_{\lambda,k}$ satisfies the relation
$L_\lambda (\xi_k \varphi^*_k v) = \varphi^*_k L_{\lambda,k}(\psi^*_k \xi_k v)$  for $v\in H^2_q(\bX_k; \bR^n),$
where $\psi_k:=\varphi_k^{-1}$.
It follows that
$$L_k^1:=L_{\lambda,k}-L_{\lambda,k}^\#$$
is of lower order, since the terms of highest order are already included in $L_{\lambda,k}^\#$. Applying $L_{\lambda}$ to $u-H_\lambda^\ell u$ yields
\begin{align*}
L_{\lambda}(u-H_\lambda^\ell u)&=L_{\lambda}\sum\limits_{k\in \cK}  \xi_k  \kfk  (L_{\lambda,k}^\#)^{-1}(\bar{f}_k,\bar{h}_{1,k},\bar{h}_{2,k})\\
&=\sum\limits_{k\in \cK}  \xi_k  \kfk  L_{\lambda,k}(L_{\lambda,k}^\#)^{-1}(\bar{f}_k,\bar{h}_{1,k},\bar{h}_{2,k})\\
&+\sum\limits_{k\in \cK} \kfk [L_{\lambda,k},\psi_k^*\xi_k]   (L_{\lambda,k}^\#)^{-1}(\bar{f}_k,\bar{h}_{1,k},\bar{h}_{2,k})\\
&=\sum\limits_{k\in \cK}  \xi_k  \kfk  (\bar{f}_k,\bar{h}_{1,k},\bar{h}_{2,k})\\
&+\sum\limits_{k\in \cK}  \xi_k  \kfk  L_k^1(L_{\lambda,k}^\#)^{-1}(\bar{f}_k,\bar{h}_{1,k},\bar{h}_{2,k})\\
&+\sum\limits_{k\in \cK} \kfk [L_{\lambda,k},\psi_k^*\xi_k]   (L_{\lambda,k}^\#)^{-1}(\bar{f}_k,\bar{h}_{1,k},\bar{h}_{2,k}).
\end{align*}
We note on the go that
$$\sum\limits_{k\in \cK}  \xi_k  \kfk  (\bar{f}_k,\bar{h}_{1,k},\bar{h}_{2,k})=(f,h_1,h_2)$$
and we define
\begin{equation*}
H_\lambda^r(f,h_1,h_2):=
\sum\limits_{k\in \cK} \Big( \xi_k  \kfk  L_k^1(L_{\lambda,k}^\#)^{-1}+ \kfk [L_{\lambda,k},\psi_k^*\xi_k]   (L_{\lambda,k}^\#)^{-1}\Big)
(\bar{f}_k,\bar{h}_{1,k},\bar{h}_{2,k}).
\end{equation*}
Since both operators $L_k^1$ and $[L_{\lambda,k},\psi_k^*\xi_k]$ are of lower order, it follows that
$$\|H_\lambda^r\|\le \frac{1}{2}\quad $$
for $\lambda \in \lambda_0+\Sigma_{\pi-\phi}$ and by possibly further increasing $\lambda_0>0$ if necessary,
where the space
$$L_{q}(\sM;T\sM)\times W_q^{1-1/q}(\Sigma;T\Sigma)\times W_q^{2-1/q}(\Sigma) $$
is equipped with the norm
on the right hand side of \eqref{resolvent estimate Lps}.
This in turn implies that
$$(S_\lambda-H_\lambda^\ell S_\lambda)(I+H_\lambda^r)^{-1}$$
is a right inverse for $L_\lambda$.  Hence, $L_\lambda$ is invertible.
\end{proof}
In a next step, we consider homogeneous boundary conditions
\begin{equation}\label{perfect slip boundary condition}
\cP_\Sigma \left( (\nabla u - [\nabla  u]^{\sT} )  \nu_\Sigma \right)    =0  \quad \text{and}\quad (u|\nu_\Sigma)_g=0 \quad \text{on } \Sigma
\end{equation}
in \eqref{perfect slip-elliptic} and we define an operator $L_{ps}: \sD(L_{ps})\to L_{q}(\sM;T\sM)$ by
$$
L_{ps} u = -   \Delta_\sM u  +\Ric^\sharp   u, \quad u\in \sD(L_{ps}):=\{u\in H^2_q(\sM;T\sM) : \, u \text{ satisfies } \eqref{perfect slip boundary condition} \}.
$$
Note that by Proposition \ref{pro:wellposedPS}, the operator $(\lambda+L_{ps})$ is invertible for any $\lambda\in \lambda_0+\Sigma_{\pi-\phi}$.

We can then show the following stronger result.
\begin{proposition}
There exists $\omega_0>0$ such that for all $\omega>\omega_0$
\begin{equation}
\label{H-cal of Lsp}
\omega+ L_{ps}\in H^\infty(L_q(\sM;T\sM)) \text{ with $H^\infty$-angle }  < \pi/2  .
\end{equation}
\end{proposition}
\begin{proof}
We define  the linear operator $L_k: \sD(L_k)\to {L_q}(\bX_k;\bR^n)$ by
\begin{equation*}
 L_k u := - \bar{g}^{ij}_{(k)} \partial_i \partial_j u    \quad \text{in }\bX_k
\end{equation*}
and for $k\in \cK_1$,
\begin{equation*}
 T_k u :=
 \left(\begin{array}{l}
 {\rm tr}_{\bR^{n-1}}
 (\frac{1}{\sqrt{\bar{g}_{(k)}^{nn}}}\bar{g}_{(k)}^{nj}  \partial_j u^1 -   \frac{\bar{g}_{(k)}^{nj} \bar{g}_{(k)}^{n1} }{( \bar{g}_{(k)}^{nn})^{3/2}}\partial_j u^n  )   \\
 \hspace{2cm} \vdots \\
  {\rm tr}_{\bR^{n-1}}
  ( \frac{1}{\sqrt{\bar{g}_{(k)}^{nn}}}\bar{g}_{(k)}^{nj} \partial_j u^{n-1} -  \frac{\bar{g}_{(k)}^{nj} \bar{g}_{(k)}^{n,n-1} }{( \bar{g}_{(k)}^{nn})^{3/2}}\partial_j u^n ) \\
   {\rm tr}_{\bR^{n-1}} \frac{1}{\sqrt{\bar{g}_{(k)}^{nn}}}u^n
 \end{array}\right),
\end{equation*}
where
$$
\sD(L_k)=
\begin{cases}
H^2_q(\bR^n;\bR^n) \quad &\text{if } k\in \cK_0,\\
\{u\in H^2_q(\bR^n_+;\bR^n): \, T_ku=0 \text{ on }\bR^{n-1} \}  &\text{if } k\in \cK_1.
\end{cases}
$$
We claim that  there exists some $\omega_0>0$ and  $\phi^\infty\in (0,\pi/2)$ such that
\begin{equation}
\label{H-cal of Lk}
\omega+ L_k\in H^\infty(L_q(\bX_k;\bR^n)) \text{ with $H^\infty$-angle }  <\phi^\infty \quad \text{for all }\omega>\omega_0.
\end{equation}
It is well-known that \eqref{H-cal of Lk} holds true for $k\in \cK_0$,
see for instance \cite[Theorem~6.1]{DuoSim97} or \cite[Theorem~4.1]{DDHP04},
as long as $R>0$ is sufficiently small.

In the case of $k\in \cK_1$,  in the planar case, i.e. $G_k=I_n$, one can check that
$$
L_k = -{\rm diag}\,[\Delta_N, \cdots, \Delta_N, \Delta_D] : \sD(L_k)\to L_p(\bX_k;\bR^n),
$$
where  $\Delta_N$ and $\Delta_D$ are the Neumann and Dirichlet Laplacian in $\bR^n_+$, respectively.
Then it follows from well-known results that \eqref{H-cal of Lk} holds, see \cite[Theorem~7.4]{DenHiePru03}. For a general geometry, using a similar perturbation argument to that in \cite{GHT13}, one can show that, by making $R>0$ sufficiently small,  \eqref{H-cal of Lk} is at our disposal.

We seek to find an expression for the resolvent $(\lambda+L_{ps})^{-1}$. To this end, consider the splitting \eqref{resolvent exp Lps} for the solution $u$ of \eqref{perfect slip-elliptic} with homogeneous boundary conditions. This yields
\begin{align}
\label{resolvent exp Lps2}
\begin{split}
(\lambda+L_{ps})^{-1} f=u ={ u_{(1)}+u_{(2)}} &:= { \Re \left(({\tilde u)}_{k\in\cK}\right) + \Re \left( ({\hat u)}_{k\in\cK}\right) }\\
 & = \sum\limits_{k\in \cK}  \xi_k  \kfk  \tilde{u}_k  +\sum\limits_{k\in \cK}  \xi_k  \kfk  \hat{u}_k \\
 & = {  \Re \left(((\lambda+L_k)^{-1}\bar{f}_k)_{k\in\cK}\right) } +  R(\lambda)(f),
\end{split}
\end{align}
where
\begin{equation*}
R(\lambda)(f)  := \sum\limits_{k\in \cK_0 }  \xi_k  \kfk (\lambda+L_k)^{-1}P_k(u)
+ \sum\limits_{k\in \cK_1 } \xi_k  \kfk (L_{\lambda,k}^\#)^{-1}(P_k(u), {\rm tr}_{\bR^{n-1}}Q_k(u), 0)^\sT .
\end{equation*}
The estimates \eqref{eq:est_u2}, \eqref{eq:interpolation} and \eqref{resolvent estimate Lps} then yield the existence of $\omega_0>0$ such that
\begin{equation}
\label{est for Rl}
\| R(\lambda)f \|_{L_q(\sM)} \leq C |\lambda|^{-3/2} \|f\|_{L_q(\sM)}
\end{equation}
for all $\lambda \in \omega+\Sigma_{\pi-\phi^\infty}$, $\omega\ge \omega_0$.

Note that the $H^\infty$-bound $K_{\phi^\infty}$, cf. \eqref{Appendix B: cHi}, can be chosen uniformly for $\omega+L_k$,
where $\omega>\omega_0$ by possibly further increasing $\omega_0$.
Given any $h\in \cH_0(\Sigma_{\pi-\phi^\infty})$,  see Appendix \ref{Appendix D} for a definition of $\cH_0(\Sigma_\phi$),
in view of \eqref{resolvent exp Lps2} we obtain
\begin{align}
\notag
&\|h(\omega+L_{ps})f\|_{L_q(\sM)}\\
\notag
&= \left\|  \frac{1}{2 \pi i} \int_{\Gamma} h(\lambda)(\lambda +\omega+ L_{ps})^{-1}f  \, d\lambda\right\|_{L_q(\sM)}\\
\notag
& \le c \left\| \frac{1}{2\pi i} \sum_{k\in \cK} \xi_k \kfk   \left[\int_{\Gamma} h(\lambda)  (\lambda + \omega +L_k)^{-1}\bar{f}_k  \, d\lambda \right] \right\|_{L_q(\sM)}
\hspace{-5mm} + { c \left\|  \int_{\Gamma} h(\lambda) R(\lambda) f  \, d\lambda  \right\|_{L_q(\sM)} } \\
\notag
& \le   M \sum_{k\in \cK} \left\|  \frac{1}{2 \pi i} \int_{\Gamma} h(\lambda)  (\lambda + \omega +L_k)^{-1}\bar{f}_k  \, d\lambda \right\|_{L_q(\bX_k )}   + M \|h\|_\infty     \int_{\Gamma} \left\|  R(\lambda) f\right\|_{L_q(\sM)}   \, d s     \\
\label{use est Rl}
& \le M \|h\|_\infty \sum_{k\in \cK}  \|\bar{f}_k \|_{L_q(\bX_k )}   + M \|h\|_\infty \|f\|_{L_q(\sM)}  \\
\notag
& \le M  \|h\|_\infty \|f\|_{L_q(\sM)},
\end{align}
where the integral contour $\Gamma$ is defined as in \eqref{Def integral contour} and \eqref{use est Rl}   follows from \eqref{H-cal of Lk} and \eqref{est for Rl}.
\end{proof}

We have shown in Proposition~\ref{pro:wellposedPS} that, for every $\lambda\in \omega+\Sigma_{\pi-\phi^\infty}$ and $f\in L_q(\sM;T\sM)$, the equation $\lambda u-L_{ps} u=f$ has a unique solution $u\in  \sD(L_{ps})$.
More can be said about the solution $u$ if, in addition, $f\in L_{q,\sigma}(\sM;T\sM)$.
\begin{proposition}
\label{Prop: Lsp preserve divergence free}
There exists $\lambda_0>0$ such that for all $\lambda>\lambda_0$ and $f\in L_{q,\sigma}(\sM;T\sM)$, the equation $\lambda u-L_{ps} u=f$ has a unique solution $u\in  H^2_{q,\sigma}(\sM;T\sM)$.
\end{proposition}
\begin{proof}
We consider the Neumann problem
\begin{equation}
\label{Neumann pb}
\left\{\begin{aligned}
\Delta_B \phi &=  \lambda\,\dv u   &&\text{on}&&\sM ,\\
(\gd\phi | \nu_\Sigma)_g    &=0 &&\text{on}&&\Sigma,
\end{aligned}\right.
\end{equation}
where $\Delta_B$ is the Laplace-Beltrami operator on $(\sM,g)$. By Lemma~\ref{Apendix Lemma Helmholtz},
\eqref{Neumann pb} has a unique (up to a constant) solution $\phi\in H^3_q(\sM)$.
Employing Lemma~\ref{Appendix Lem: divergence thm}  repeatedly, we obtain
\begin{align}
\notag
& \|\gd\phi \|_{L_2(\sM)}^2   \\
\notag
&=  (-\Delta_B \phi | \phi )_{\sM} = -\lambda ( \dv u | \phi )_{\sM} = \lambda (  u |  \gd \phi )_{\sM}\\
\notag
&=  (\Delta_\sM u | \gd \phi)_{\sM} + (f |\gd \phi)_{\sM} - (\Ric^\sharp u | \gd \phi)_{\sM} \\
\label{Neumann pb eq1}
&= -(\nabla u | \nabla \gd \phi)_{\sM} + ( \nabla_{\nu_\Sigma} u | \gd \phi)_{\Sigma}- (\Ric^\sharp u | \gd \phi)_{\sM} \\
\notag
&= (u | \Delta_\sM \gd \phi)_{\sM}  + ( \nabla_{\nu_\Sigma} u | \gd \phi)_{\Sigma}   - (u | \nabla_{\nu_\Sigma} \gd \phi)_{\Sigma}- (\Ric^\sharp u | \gd \phi)_{\sM} \\
\label{Neumann pb eq2}
&=  (u| \gd \Delta_B \phi)_{\sM}  + ( \nabla_{\nu_\Sigma} u | \gd \phi)_{\Sigma}   - (u | \nabla_{\nu_\Sigma} \gd \phi)_{\Sigma} \\
\label{Neumann pb eq3}
&=  -\lambda  \|\dv u\|_{L_2(\sM)}^2   + ( \nabla_{\nu_\Sigma} u | \gd \phi)_{\Sigma}   - (u | \nabla_{\nu_\Sigma} \gd \phi)_{\Sigma}.
\end{align}
By the definition of $\PH$ and the fact that $f\in L_{q,\sigma}(\sM;T\sM)$, we conclude that
$$
(f | \gd \phi)_{\sM}=(\PH f | \gd \phi)_{\sM}=(f-\gd\psi_f | \gd \phi)_{\sM}=0.
$$
Therefore,  \eqref{Neumann pb eq1} follows.
In  \eqref{Neumann pb eq2},
we have used the property
\begin{equation}
\label{Def Ricci}
\Delta_\sM \gd \phi=\gd \Delta_B \phi +\Ric^\sharp \gd \phi,
\end{equation}
see Lemma~\ref{lem: commuator-Ricci}.

Since $\phi$ is a scalar function, we have $\nabla \gd \phi=(\nabla \gd \phi)^\sT$.
Employing \eqref{boundary condition 1}, with $u$ replaced by $\gd \phi$, we then obtain
\begin{align*}
& \cP_\Sigma \nabla_{\nu_\Sigma} \gd \phi =   \cP_\Sigma (\nabla \gd \phi  )\nu_\Sigma  = \cP_\Sigma (\nabla \gd \phi  )^\sT\nu_\Sigma
= L_\Sigma \gd \phi.
\end{align*}
Hence, the last   term on the RHS of \eqref{Neumann pb eq3} can be rewritten as
$$
(u | \nabla_{\nu_\Sigma} \gd \phi)_{\Sigma} =(u | \cP_\Sigma \nabla_{\nu_\Sigma} \gd \phi)_{\Sigma} = (L_\Sigma u | \gd \phi )_{\Sigma}
$$
in view of \eqref{perfect slip boundary condition}. We thus infer that
\begin{align*}
\| \gd\phi\|_{L_2(\sM)}^2 + \lambda \|\dv u\|_{L_2(\sM)}^2  &= ( \nabla_{\nu_\Sigma} u  -L_\Sigma u | \gd \phi)_{\Sigma} \\
&= ((\nabla u - [\nabla u]^{\sf T})\nu_\Sigma | \gd \phi)_{\Sigma} =0,
\end{align*}
where we have used \eqref{perfect slip boundary condition} once more.
We thus have $\dv u =0 $ and  $(u | \nu_\Sigma)_g=0$ and this, in turn, implies
 $u\in H^2_{q,\sigma}(\sM;T\sM)$,
in virtue of Lemma~\ref{Apendix Lemma Helmholtz}.
\end{proof}

Proposition~\ref{Prop: Lsp preserve divergence free} reveals that for $\lambda>\lambda_0$
\begin{equation}
\label{perfect slip range cond}
(\lambda + L_{ps})^{-1} \sR(\PH) \subseteq \sR(\PH) .
\end{equation}
We will further show that
\begin{equation}
\label{perfect slip ker cond}
(\lambda + L_{ps})^{-1} \sN(\PH) \subseteq \sN(\PH) .
\end{equation}
Indeed, if $f=\gd g$ for some $g\in H^1_q(\sM)$, then we consider
\begin{equation}
\label{Ker pb}
\left\{\begin{aligned}
(\lambda-\Delta_B) \phi &= g   &&\text{on}&&\sM ,\\
(\gd \phi | \nu_\Sigma)_g   &=0 &&\text{on}&&\Sigma.
\end{aligned}\right.
\end{equation}
For sufficiently large $\lambda_0>0$ and all $\lambda>\lambda_0$, \eqref{Ker pb} has a unique solution $\phi\in H^3_q(\sM)$
by means of a localization argument as in Section~\ref{Section:NS with perfect slip boundary condition strong}.
Let $v=\gd \phi$. Then $v$ satisfies
$$
\lambda v - \gd \Delta_B \phi = \gd g =f .
$$
Using \eqref{Def Ricci} once more, it is easy task to check that $v$ solves
\begin{equation*}
\left\{\begin{aligned}
(\lambda -\Delta_\sM +\Ric^\sharp)  v &= f   &&\text{on}&&\sM ,\\
(v | \nu_\Sigma)_g   &=0 &&\text{on}&&\Sigma.
\end{aligned}\right.
\end{equation*}
On the other hand, since $\nabla v = [\nabla v]^\sT$, the boundary condition
$$
\cP_\Sigma \left( (\nabla v - [\nabla  v]^{\sT} )  \nu_\Sigma \right)    =0   \quad \text{on } \Sigma
$$
is automatically satisfied. Hence $v$ is indeed a solution of \eqref{perfect slip-elliptic}.
Uniqueness of  solutions of \eqref{perfect slip-elliptic} implies that $v=u$ and thus \eqref{perfect slip ker cond} is proved.

By \eqref{perfect slip range cond} and \eqref{perfect slip ker cond}, given any $u\in \sD(L_{ps})$, one has
\begin{equation}
\label{exchange PH and Lps}
L_{ps} \PH u \in \sR(\PH) \quad  \text{and} \quad L_{ps} (I-\PH) u \in \sN(\PH).
\end{equation}
Now we are ready to study  the {\em strong surface Stokes operator with perfect slip boundary conditions} $A_{ps}: \sD(A_{ps})\to L_{q,\sigma}(\sM;T\sM)$, defined by
$$
A_{ps}u=  -  \mu_s  \PH (\Delta_\sM u + \Ric^\sharp u )
$$
with
$$
\sD(A_{ps})=\{u\in H^2_{q,\sigma}(\sM;T\sM):   (u| \nu_\Sigma)_g=0,\ \  \cP_\Sigma \left( (\nabla u - [\nabla  u]^{\sT} )  \nu_\Sigma \right)=0 \text{ on } \Sigma\}.
$$
In spite of \eqref{sigma-u-nu}, we include the condition  $(u| \nu_\Sigma)_g=0$ for extra emphasis.

Let $\widetilde{A}_{ps}:= A_{ps} +   2\mu_s   \PH \Ric^\sharp: \sD(A_{ps})\to L_{q,\sigma}(\sM;T\sM)$. Since for any $u\in \sD(L_{ps})$, one can deduce
\begin{align*}
 \PH L_{ps} u = \PH L_{ps} \PH u + \PH L_{ps} (I-\PH) u = L_{ps} \PH u
\end{align*}
from \eqref{exchange PH and Lps},
it holds that
$$
\widetilde{A}_{ps}=  \mu_s  L_{ps} |_{\sD(A_{ps})}.
$$
Therefore, for sufficiently large $\omega>0$, $\omega+ \widetilde{A}_{ps} \in H^\infty(L_{q,\sigma}(\sM;T\sM))$, by \eqref{H-cal of Lsp}.
By possibly enlarging $\omega_0>0$, the following theorem is an immediate consequence of \cite[Corollary 3.3.15]{PruSim16}.
\begin{theorem}
There exists $\omega_0>0$ such that for all $\omega>\omega_0$
\begin{equation}
\label{cHi Aps}
\omega+ A_{ps} \in H^\infty(L_{q,\sigma}(\sM;T\sM)) \text{ with $H^\infty$-angle }  < \pi/2 .
\end{equation}
\end{theorem}

\subsection{Weak formulation}\label{Section:NS with perfect slip boundary condition weak}
For notational brevity, let $A_0=\omega+A_{ps}$, $\omega>\omega_0$ with $\omega_0$ being defined in \eqref{cHi Aps}.
Note that $\omega+A_{ps}$ is invertible.
We set
$$
Z_0=X_0  =L_{q,\sigma}(\sM;T\sM) \quad \text{and} \quad Z_1:=\sD(A_{ps}).
$$
By \cite[Theorems V.1.5.1 and V.1.5.4]{Ama95}, the pair $(Z_0, A_0)$ generates an interpolation-extrapolation scale $(Z_\beta, A_\beta)$, $\beta\in \bR$, with respect to the complex interpolation functor.
In particular, when $\beta\in (0,1)$, $A_\beta$ is the $Z_\beta$-realization of $A_0$, where
$$
Z_\beta= \sD(A_0^\beta)=[Z_0,Z_1]_\beta
$$
due to \eqref{cHi Aps}.
Let $Z_0^\sharp:=(Z_0)'=L_{q',\sigma}(\sM;T\sM)$ and
\begin{equation*}
\begin{aligned}
A_0^\sharp: &=(A_0)'= \omega- \mu_s\PH (\Delta_\sM - \Ric^\sharp) : \sD(A_0^\sharp) \to Z_0^\sharp, \\
\sD(A_0^\sharp) &=Z_1^\sharp:=\{u\in H^2_{q',\sigma}(\sM;T\sM) : \, \cP_\Sigma \left( (\nabla u - [\nabla  u]^{\sT} )  \nu_\Sigma \right)    =0   \text{ on } \Sigma \}.
\end{aligned}
\end{equation*}
Then $(Z_0^\sharp, A_0^\sharp)$ generates an interpolation-extrapolation scale $(Z_\beta^\sharp, A_\beta^\sharp)$, $\beta\in \bR$, the dual scale.
By \cite[Theorem  V.1.5.12]{Ama95}, it holds that
\begin{equation}
\label{dual inter-extra}
(Z_\beta)'=Z_{-\beta}^\sharp \quad \text{and}\quad (A_\beta)'=A_{-\beta}^\sharp
\end{equation}
for $\beta\in \bR$. Particularly, when $\beta=-1/2$, the operator $A_{-1/2}:Z_{1/2}\to Z_{-1/2}$ satisfies
$$
\sD(A_{-1/2})=Z_{1/2}=[Z_0,Z_1]_{ 1/2}=H^1_{q,\sigma}(\sM;T\sM),
$$
see  Proposition~\ref{Appendix Prop interpolation 4},
and $Z_{-1/2}=(Z_{1/2}^\sharp)'$.
Note that
$$
Z_{1/2}^\sharp= [Z_0^\sharp,Z_1^\sharp]_{1/2}=H^1_{q',\sigma}(\sM;T\sM).
$$
Therefore,
$$
Z_{1/2}=X_{1/2} \quad \text{and} \quad Z_{-1/2}=X_{-1/2},
$$
where $X_{1/2}$ and $X_{-1/2}$ were introduced in  \eqref{X12}.
By the definitions in \eqref{Def sigma spaces}, one can follow the arguments in \cite[Propositions~2.3 and 2.4]{PrWi18} and show that for any $\theta\in (0,1)$
\begin{equation}
\label{interpolation of X}
[Z_{-1/2}, Z_{ 1/2}]_\theta = H^{2\theta-1}_{q,\sigma}(\sM;T\sM) \quad (Z_{-1/2}, Z_{ 1/2})_{\theta,p} = B^{2\theta-1}_{q p ,\sigma}(\sM;T\sM),
\end{equation}
see also Proposition~\ref{Appendix Prop interpolation 3}.
By replacing $q$ by $q'$, we infer from  Section~\ref{Section:NS with perfect slip boundary condition strong} that
$A_0^\sharp \in H^\infty(Z_0^\sharp)$ with $H^\infty$-angle $\phi^\infty_{A_0^\sharp}<\pi/2$.
Since $A_{1/2}^\sharp$ is the $Z_{1/2}^\sharp$-realization of $A_0^\sharp$, it follows from \cite[Proposition~3.3.14]{PruSim16} and \eqref{dual inter-extra} that
$$\text{ $A_{-1/2}=(A_{1/2}^\sharp)'\in H^\infty( Z_{-1/2})$ with $H^\infty$-angle $< \pi/2$. }$$
We call the operator $A_{-1/2}: Z_{1/2}\to Z_{-1/2}$ the {\em weak surface Stokes operator with perfect slip boundary conditions}.
\goodbreak

Since $A_{-1/2}$ is the closure of $A_0$ in $Z_{-1/2}$, it follows that $A_{-1/2} u=A_0 u$ for all $u\in \sD(A_0)=Z_1$.
Thus for any $v\in Z_{1/2}^\sharp$, it follows from \eqref{boundary condition 1},
\eqref{perfect slip boundary condition} and Lemma~\ref{Appendix Lem: divergence thm}(b)(ii)  that
\begin{align*}
&
 \la A_{-1/2} u | v\ra_{\sM}
=  (A_0 u | v)_{\sM}\\
&= \omega (u | v)_{\sM} +  \mu_s (\nabla u   | \nabla v)_{\sM} -\mu_s (\Ric^\sharp u   |  v)_{\sM} - \mu_s (  \nabla u    \nu_\Sigma  | v)_{\Sigma}  \\
&= \omega (u | v)_{\sM} +  \mu_s (\nabla u   | \nabla v)_{\sM} -\mu_s (\Ric^\sharp u   |  v)_{\sM} - \mu_s ( \cP_\Sigma ([\nabla u ]^\sT   \nu_\Sigma)  | v)_{\Sigma} \\
&= \omega (u | v)_{\sM} +\mu_s (\nabla u   | \nabla v)_{\sM} -\mu_s (\Ric^\sharp u   |  v)_{\sM} - \mu_s (     L_\Sigma u |   v)_{\Sigma}.
\end{align*}
By the density of $Z_1$ in $Z_{1/2}$, we infer that  for all
$$(u,v)\in Z_{1/2} \times Z_{1/2}^\sharp=H^1_{q,\sigma}(\sM;T\sM)\times   H^1_{q',\sigma}(\sM;T\sM),  $$
\begin{equation}
\label{Perfect slip weak formulation}
\la A_{-1/2} u | v\ra_{\sM} =\omega (u | v)_{\sM} +\mu_s (\nabla u   | \nabla v)_{\sM} -\mu_s (\Ric^\sharp u   |  v)_{\sM} - \mu_s (     L_\Sigma u |   v)_{\Sigma}.
\end{equation}

\section{$H^\infty$-calculus of surface Stokes operator with Navier boundary conditions}\label{Section:NS with Navier boundary condition}
\label{Stokes-H-infity}
Recall the definition of the {\em weak Stokes operator with Navier boundary conditions} $\Aw_N: X_{1/2} \to X_{-1/2}$ provided in  Section~\ref{surface Stokes operator}.
In view of \eqref{Perfect slip weak formulation}, easy computations  show  that
\begin{align*}
\la (\omega + \Aw_N) u | v\ra_{\sM} = \la A_{-1/2} u | v\ra_{\sM} + 2\mu_s (L_\Sigma u   | v )_{\Sigma} + \alpha\mu_s (u|v)_{\Sigma}
\end{align*}
is valid for all $(u,v)\in  H^1_{q,\sigma}(\sM;T\sM)\times   H^1_{q',\sigma}(\sM;T\sM)  $.
We define the operator $B_N: \sD(B_N)\to X_{-1/2}$ by
$$
\la B_N u | v\ra_{\sM}
=2\mu_s (L_\Sigma u   | v )_{\Sigma} + \alpha\mu_s (u|v)_{\Sigma}
$$
for all $(u,v)\in \sD(B_N)\times   H^1_{q',\sigma}(\sM;T\sM)  $. The domain $\sD(B_N)$ will be specified in the following calculations.
By trace theory and H\"older's inequality, we have
	\begin{align*}
\left| \la B_N u | v\ra_{\sM} \right|  & \le C \|u\|_{L_q(\Sigma)}\|v\|_{L_{q'}(\Sigma)} \leq C \|u\|_{H^s_{q,\sigma}(\sM)}\|v\|_{H^1_{q',\sigma}(\sM)}
\end{align*}
for any $s>1/q$.
Thus, by choosing $\sD(B_N)=H^s_{q,\sigma}(\sM;T\sM)=[X_{-1/2},X_{1/2}]_{\theta}$ with $s\in (1/q,1)$ and $\theta=(s+1)/2 $, $B_N\in \cL(\sD(B_N), X_{-1/2})$ and thus is a lower order perturbation of $A_{-1/2}: X_{1/2}\to X_{-1/2}$.
Then it again follows from \cite[Corollary 3.3.15]{PruSim16} that, by possibly enlarging $\omega_0>0$,
\begin{equation}
\label{cHi ANw}
\omega+ \Aw_N \in H^\infty(X_{-1/2}) \text{ with $H^\infty$-angle } < \pi/2 \quad \text{for all  }\omega>\omega_0.
\end{equation}
Next, we will show that $A_N$ also admits bounded $H^\infty$-calculus.
Take $\omega>0$ sufficiently large so that $\omega+\Aw_N$ is invertible.
Given any $u\in \sD(A_N)$, see \eqref{Strong-Stokes-Navier} and \eqref{D(AN)},
there exists a unique $w\in X_{1/2}$ such that $(\omega+A_N ) u =( \omega+\Aw_N ) w$.
Because of \eqref{weak formulation Navier 1}, for all $v\in H^1_{q',\sigma}(\sM;T\sM)$  we have
\begin{align*}
(( \omega+A_N )u | v)_{\sM} =  \la ( \omega+\Aw_N) w | v\ra_{\sM} = \la ( \omega+\Aw_N) u | v\ra_{\sM}.
\end{align*}
The injectivity of $ \omega+\Aw_N$ implies that $u=w$. Thus, ${\rm gr}( A_N)\subset {\rm gr}(\Aw_N)$,
where ${\rm gr}(\cA)$ is the graph of an operator $\cA$ in $X_{-1/2}$.
Lemma \ref{lem: mat} further implies  that  $\omega+A_N: \sD(A_N)\subset L_{q,\sigma}(\sM;T\sM) \to L_{q,\sigma}(\sM;T\sM)$ is closed and bijective. Therefore, $\omega+A_N$ is the $L_{q,\sigma}(\sM;T\sM)$-realization of $\omega+\Aw_N$ and it inherits the bounded $H^\infty$-calculus property, i.e.,
there exists $\omega_0>0$ such that
\begin{equation}
\label{cHi AN}
\omega+ A_N \in H^\infty(X_0) \text{ with $H^\infty$-angle }   < \pi/2 \quad \text{for all  }\omega>\omega_0.
\end{equation}

\begin{lemma}
\label{lem: mat}
There exists $\lambda_0>0$ such that for every $\lambda>\lambda_0$ and $f\in L_{q,\sigma}(\sM;T\sM)$
\begin{equation}
\label{surject AN}
\left\{\begin{aligned}
(\lambda -\PH \Delta_\sM  - \PH \Ric^\sharp) u &= f &&\text{on}&&\sM ,\\
\alpha u + \cP_\Sigma \left( (\nabla u + [\nabla  u]^{\sT} )  \nu_\Sigma \right)   &=0 &&\text{on}&&\Sigma, \\
(u| \nu_\Sigma)_g  &=0 &&\text{on}&&\Sigma
\end{aligned}\right.
\end{equation}
has a unique solution $u\in H^2_{q,\sigma}(\sM;T\sM)$.
\end{lemma}
\begin{proof}
We proved in Proposition~\ref{pro:wellposedPS} that there exists $\lambda_0>0$ such that for all $\lambda\ge\lambda_0$, for all $f\in L_q(\sM;T\sM)$ and all $h\in W_q^{1-1/q}(\Sigma;T\Sigma)$ there exists a unique
solution $u\in H_q^2(\sM;T\sM)$ of the problem
\begin{equation}
\label{eq:surj1}
\left\{\begin{aligned}
(\lambda -  \Delta_\sM  +\Ric^\sharp ) u &=f &&\text{on}&&\sM ,\\
 \cP_\Sigma \left( (\nabla u -[\nabla  u]^{\sT} )  \nu_\Sigma \right)   &=h &&\text{on}&&\Sigma  , \\
(u| \nu_\Sigma)_g  &=0 &&\text{on}&&\Sigma,
\end{aligned}\right.
\end{equation}
and, in addition, there exists a constant $C=C(\lambda_0)>0$ such that the estimate
\begin{equation}\label{eq:surj1a}
\lambda\|u\|_{L_q(\sM)}+\|u\|_{H_q^2(\sM)}\le C\left(\|f\|_{L_q(\sM)}+\lambda^{1/2}\|H\|_{L_q(\sM)}+\|H\|_{H_q^{1}(\sM)}\right)
\end{equation}
holds for the solution $u\in H_q^2(\sM;T\sM)$ of \eqref{eq:surj1}, where $H$ is any extension of $h$ from $W_q^{1-1/q}(\Sigma)$ to $H_q^1(\sM)$.

In a first step, we will show that there
exists $\lambda_0>0$ such that for all $\lambda\ge\lambda_0$, for all $f\in L_q(\sM;T\sM)$ and all $h\in W_q^{1-1/q}(\Sigma;T\Sigma)$ there exists a unique solution $(v,\pi)\in H_q^2(\sM;T\sM)\times \dot{H}_q^1(\sM)$ of
the Stokes problem
\begin{equation}
\label{eq:surj2}
\left\{\begin{aligned}
(\lambda -  \Delta_\sM  +\Ric^\sharp ) v+ \gd\,\pi  &=f &&\text{on}&&\sM ,\\
\dv\,v&=0 &&\text{on}&&\sM ,\\
 \cP_\Sigma \left( (\nabla v -[\nabla  v]^{\sT} )  \nu_\Sigma \right)   &=h &&\text{on}&&\Sigma  , \\
(v| \nu_\Sigma)_g  &=0 &&\text{on}&&\Sigma,
\end{aligned}\right.
\end{equation}
satisfying the estimate
\begin{equation}\label{eq:surj2a}
\lambda\|v\|_{L_q(\sM)}+\|v\|_{H_q^2(\sM)}+\|\gd\,\pi\|_{L_q(\sM)}
\le C \left(\|f\|_{L_q(\sM)}+\lambda^{1/2}\|H\|_{L_q(\sM)}+\|H\|_{H_q^{1}(\sM)} \right).
\end{equation}
Indeed, let $v:=\PH u=u-\gd\,\psi_u$, where $u\in H_q^2(\sM;T\sM)$ is the unique solution of \eqref{eq:surj1}
and  $\gd\psi_u \in H_q^2(\sM; T\sM)$ is the unique solution of
\begin{equation*}
\left\{\begin{aligned}
\Delta_B \psi_u&=\dv\, u &&\text{on}&&\sM ,\\
(\gd\,\psi_u| \nu_\Sigma)_g  &=0 &&\text{on}&&\Sigma,
\end{aligned}\right.
\end{equation*}
see Lemma~\ref{Apendix Lemma Helmholtz}.
Defining $\pi:=\lambda\psi_u-\dv\, u$, it follows from \eqref{Def Ricci} that the pair $(v,\pi)$ is a solution of \eqref{eq:surj2}.
Moreover, by \eqref{est Helmholtz}  and \eqref{eq:surj1a}, we have the estimates
\begin{align*}
\|\gd\,\psi_u\|_{H_q^2(\sM)}&  \le C\|u\|_{H_q^2(\sM)} \\
&\le C\left(\|f\|_{L_q(\sM)}+ \lambda^{1/2}\|H\|_{L_q(\sM)}+\|H\|_{H_q^{1}(\sM)} \right)
\end{align*}
and
\goodbreak
\begin{align*}
\|\gd\,\pi\|_{L_q(\sM)}&\le \lambda\|\gd\,\psi_u\|_{L_q(\sM)}+\|\gd\,\dv\,u\|_{L_q(\sM)}\\
&\le C\left(\lambda\|u\|_{L_q(\sM)}+\|u\|_{H_q^2(\sM)}\right)\\
&\le C\left(\|f\|_{L_q(\sM)}+ \lambda^{1/2}\|H\|_{L_q(\sM)}+\|H\|_{H_q^{1}(\sM)}\right),
\end{align*}
and therefore, the functions $(v,\pi)$ satisfy the estimate \eqref{eq:surj2a}.

Uniqueness of the solution $(v,\pi)$ to \eqref{eq:surj2} can be seen as follows. Assume that $f=0$ and $h=0$ in \eqref{eq:surj2}. Then $u:=\PH v=v$ solves \eqref{eq:surj1} with $f=0$ and $h=0$, as $\PH L_{ps}=L_{ps}\PH$
 (see Section \ref{Section:NS with perfect slip boundary condition}) and $\bP_H \gd \pi=0$.
 Since the solution to \eqref{eq:surj1} is unique, it follows that $v=u=0$. Inserting $v=0$ into \eqref{eq:surj2} we obtain $\gd\,\pi=0$, and therefore $\pi=0$ in $\dot{H}_q^1(\sM)$.

Having the unique solvability  of \eqref{eq:surj2} and the estimate \eqref{eq:surj2a} at hand, we may apply a perturbation argument as in the proof of Proposition \ref{pro:wellposedPS} in order to replace the left hand side of $\eqref{eq:surj2}_1$ by
$$\lambda v-\Delta_{\sM} v-\Ric^\sharp v=(\lambda v-\Delta_{\sM} v+\Ric^\sharp v)-2\Ric^\sharp v$$
and the boundary condition $\eqref{eq:surj2}_3$ by
$$\cP_\Sigma \left( (\nabla v +[\nabla  v]^{\sT} )  \nu_\Sigma \right)+\alpha v=\cP_\Sigma \left( (\nabla v -[\nabla  v]^{\sT} )  \nu_\Sigma \right)+2L_\Sigma v+\alpha v.$$
Indeed, for $v\in H_q^2(\sM;T\sM)$ it holds that
$$\|\Ric^\sharp v\|_{L_q(\sM)}\le C\|v\|_{L_q(\sM)}\le \lambda^{-1} C\left(\lambda\|v\|_{L_q(\sM)}+\|v\|_{H_q^2(\sM)}\right).$$
Concerning the boundary condition, we observe that $2L_\Sigma v+\alpha v\in W_q^{2-1/q}(\Sigma)$ for $v\in H_q^2(\sM;T\sM)$. Hence, there exists an extension $Q(v)\in H_q^2(\sM)$ of $2L_\Sigma v+\alpha v$ such that
$$\|Q(v)\|_{H^{s}_q(\sM)}\le C(s)\|v\|_{H^s_q(\sM)},\quad v\in H^2_q(\sM;T\sM),\ s\in[0,2].$$
Then, by complex interpolation and Young's inequality,
$$
\|Q(v)\|_{H_q^{1}(\sM)}\le C\|v\|_{H_q^{1}(\sM)}\le \lambda^{-1/2} C\left(\lambda\|v\|_{L_q(\sM)}+\|v\|_{H_q^2(\sM)}\right),
$$
and
$$
\lambda^{1/2}\|Q(v)\|_{L_q(\sM)}\le \lambda^{1/2}C\|v\|_{L_q(\sM)}\le  \lambda^{-1/2}C\left(\lambda\|v\|_{L_q(\sM)}+\|v\|_{H_q^2(\sM)}\right).
$$
Therefore, a very similar Neumann series argument as in the proof of Proposition \ref{pro:wellposedPS} yields the existence of a number $\lambda_0>0$ such that for all $\lambda\ge\lambda_0$ and for all $f\in L_q(\sM;T\sM)$ there exists a unique solution
$(u,\pi)\in H_q^2(\sM;T\sM)\times \dot{H}_q^1(\sM)$ of the problem
\begin{equation*}
\left\{\begin{aligned}
(\lambda -  \Delta_\sM  -\Ric^\sharp ) u+ \gd\,\pi  &=f &&\text{on}&&\sM ,\\
\dv\,u&=0 &&\text{on}&&\sM ,\\
 \cP_\Sigma \left( (\nabla u +[\nabla  u]^{\sT} )  \nu_\Sigma \right)+\alpha u   &=0 &&\text{on}&&\Sigma  , \\
(u| \nu_\Sigma)_g  &=0 &&\text{on}&&\Sigma,
\end{aligned}\right.
\end{equation*}
satisfying the estimate
\begin{equation*}
\lambda\|u\|_{L_q(\sM)}+\|u\|_{H_q^2(\sM)}+\|\gd\,\pi\|_{L_q(\sM)}
\le C \|f\|_{L_q(\sM)}.
\end{equation*}
\end{proof}
\noindent
\emph{Proof of Theorem~\ref{thm:Stokes-H-infinty}.}
The assertions (a) and (b) of the Theorem are contained in
\eqref{cHi ANw} and \eqref{cHi AN}. \quad $\square$

\section{Existence and uniqueness of solutions}
\label{Section:existence and uniqueness}

Based on the bounded $H^\infty$-calculus property of $\Aw_N$ and $A_N$, the local well-posedness of  \eqref{NS-sys} can be proved  as in \cite{PrWi18, PrSiWi21, SiWi22}.
For the sake of completeness, we will nevertheless include a proof here.

By applying  the Helmholtz projection $\PH$ on \eqref{NS-sys}$_1$, one can readily see that the weak formulation of \eqref{NS-sys} is equivalent to the following abstract semilinear evolution equation
\begin{equation}
\label{NS sys abstract}
\left\{\begin{aligned}
\partial_t u + \Aw_N u &=\Fw(u),      && t>0 ,\\
u(0) & = u_0 ,&&
\end{aligned}\right.
\end{equation}
where for all $(u,v)\in H^1_{q,\sigma}(\sM;T\sM) \times H^1_{q',\sigma}(\sM;T\sM) $
$$
 \la \Fw (u) | v\ra_{\sM}  = ( u\otimes u_\flat | \nabla v )_{\sM},
$$
where we used Lemma~\ref{Appendix Lem: divergence thm}  and the fact that $\nabla_u u = \dv (u\otimes u)$ (which holds
since $\dv u=0)$.
For notational brevity, we put
$$
\Xw_0=X_{-1/2} \quad \text{and} \quad \Xw_1=X_{ 1/2}.
$$
Due to \eqref{interpolation of X},
$$
\Xw_\beta:= [\Xw_0, \Xw_1]_\beta= H^{2\beta-1}_{q,\sigma}(\sM;T\sM) \hookrightarrow L_{2q,\sigma}(\sM;T\sM)
$$
provided $2\beta-1\geq n/2q$. Then by taking $2\beta-1=n/2q$ and using H\"older's inequality, we infer that
\begin{align*}
\left|  \la \Fw (u) | v\ra _{\sM} \right| \leq \|u\| _{L_{2q}(\sM)}^2 \|\nabla v\|_{L_{q'}(\sM)} \leq C \|u\| _{\Xw_\beta}^2 \|v\|_{H^1_{q'}(\sM)}
\end{align*}
and hence
$$
\|\Fw  (u)\|_{\Xw_0} \leq C \|u\| _{\Xw_\beta}^2.
$$
With $2\beta-1=n/2q$, which means $q\in (n/2,\infty)$ as $\beta<1$, the critical weight $\mwc$ and the corresponding critical space in the weak setting read as
$$
\Xwm=(\Xw_0,\Xw_1)_{\mwc-1/p,p}=B^{n/q-1}_{qp,\sigma}(\sM;T\sM), \quad \mwc=2\beta-1+\frac{1}{p}= \frac{1}{p}+\frac{n}{2q},
$$
with $2/p+n/q\le 2$, see \eqref{interpolation of X} and \cite[Proposition 2.4 \& Section 3.3]{PrWi18} (for the Euclidean case).

Next, let us compute the critical spaces in the strong setting.
To this end, we consider the following abstract evolution equation, which is equivalent to the strong formulation of \eqref{NS-sys}
\begin{equation}
\label{NS sys abstract strong}
\left\{\begin{aligned}
\partial_t u + A_N u &=F(u):=-\PH (\nabla_u u),     && t>0 ,\\
u(0) & = u_0 .&&
\end{aligned}\right.
\end{equation}
It follows from H\"older's inequality that
$$
\|F(u)\|_{L_q(\sM)} \leq C \|u\|_{L_{q r'}(\sM)} \|u\|_{H^1_{q r}(\sM)},
$$
where $1/r + 1/r'=1$.
We choose  $\displaystyle 1-\frac{n}{qr}= -\frac{n}{q r'}$, or equivalently $\displaystyle \frac{n}{qr}=\frac{1}{2}\left(1+\frac{n}{q} \right)$,
so that
$$H^1_{qr}(\sM; T\sM ) \hookrightarrow L_{qr'}(\sM; T\sM).$$
Note that this choice is feasible if $q\in (1,n)$.

By interpolation theory and  Sobolev's embedding theorem, $$
[X_0,X_1]_\beta \subset H^{2\beta}_q (\sM;T\sM)\hookrightarrow H^{1}_{qr} (\sM;T\sM),
$$
provided
$$
2\beta-\frac{n}{q}=1-\frac{n}{qr}, \quad \text{or equivalently} \quad \beta=\frac{1}{4} \left( \frac{n}{q}+1 \right).
$$
In summary, we have shown that
$$\| F(u)\|_{L_q}\le C \| u \|^2_{X_\beta},\quad u\in X_\beta. $$
The condition $\beta<1$ requires $q>n/3$. Hence, for $q\in (n/3,n)$, the critical weight in the strong setting is given by
$$
\mu_c=2\beta-1+\frac{1}{p}=\frac{1}{2}\left(\frac{n}{q}-1 \right) + \frac{1}{p}, \quad  \text{with} \quad \frac{2}{p}+ \frac{n}{q}\leq 3,
$$
and  the corresponding critical space in the strong setting is given by
$
X_{\gamma,\mu_c}:=(X_0,X_1)_{\mu_c-1/p,p},
$
where
\begin{equation}
\label{strong-critical}
X_{\gamma,\mu_c} 
= B^{n/q-1}_{qp ,\sigma, \cB} (\sM;T\sM)\ \ \text{in case}\ \ n/q -1\neq1+1/q,
\end{equation}
see Proposition~\ref{Appendix Prop interpolation 2}.

\medskip
The above discussions give  rise to the following theorem concerning the local well-posedness of \eqref{NS sys abstract}, respectively \eqref{NS-sys}.
\goodbreak
\begin{theorem}
\label{Thm: local wellposed}
{\rm (a)}
Let $p\in (1,\infty) $ and $q\in (n/2,\infty)$ such that $\frac{2}{p}+ \frac{n}{q}\leq 2$. Then for any initial value  $u_0\in B^{n/q-1}_{qp ,\sigma} (\sM;T\sM)$, there exists a unique weak solution
$$
u\in H^1_{p,\mwc}((0,t_+); H^{-1}_{q,\sigma}(\sM;T\sM))\cap L_{p,\mwc}((0,t_+); H^1_{q,\sigma}(\sM;T\sM))
$$
of \eqref{NS sys abstract}
for some $t_+=t_+(u_0)>0$ with $\mwc=1/p+n/2q$.
The solution exists on a maximal time interval $[0, t_{\max}(u_0))$ and depends continuously on $u_0$.
Moreover,
$$
u \in C([0, t_{\max}); B^{n/q-1}_{qp ,\sigma} (\sM;T\sM)) \cap C((0, t_{\max}); B^{1-2/p}_{qp ,\sigma} (\sM;T\sM)).
$$
 If, in addition, $q\geq n$, the solution $u$ satisfies
\begin{equation}
\label{additional reg}
u\in H^1_{p,loc}((0,t_{\max}); L_{q,\sigma}(\sM;T\sM))\cap L_{p,loc}((0,t_{\max}); H^2_{q,\sigma}(\sM;T\sM)).
\end{equation}
Hence, any solution regularizes instantaneously and becomes a strong solution  in case $q\ge n$.

\medskip\noindent
{\rm (b)}
If $p\in (1,\infty) $ and $q\in (n/3,n)$ with $\frac{2}{p}+ \frac{n}{q}\leq 3$, then for any
initial value
$u_0\in X_{\gamma, \mu_c}$, see~\eqref{strong-critical} for a characterization,
there exists a unique strong solution
$$
u\in H^1_{p,\mu_c}((0,t_+); L_{q,\sigma}(\sM;T\sM))\cap L_{p,\mu_c}((0,t_+); H^2_{q,\sigma}(\sM;T\sM))
$$
of \eqref{NS sys abstract strong} for some $t_+=t_+(u_0)>0$ with $\mu_c=1/p+n/2q-1/2$.
The solution exists on a maximal time interval $[0, t_{\max}(u_0))$ and depends continuously on $u_0$.
Moreover,
$$
u \in C([0, t_{\max}); B^{n/q-1}_{qp ,\sigma} (\sM;T\sM)) \cap C((0, t_{\max}); B^{2-2/p}_{qp ,\sigma} (\sM;T\sM)).
$$
\end{theorem}
\begin{proof}
Because of \eqref{cHi ANw} and \eqref{cHi AN}, the local existence and uniqueness of a solution is an
immediate consequence of   \cite[Theorem~1.2]{PrWi17}, see also \cite[Theorem~2.1]{PrSiWi18}.

It remains to show the additional regularity property~\eqref{additional reg}.
Suppose $q\ge n$ and choose $r=2p$. As $r>p$, we have
\begin{equation*}
B^{n/q-1}_{qp,\sigma}(\sM; T\sM)\hookrightarrow B^{n/q-1}_{qr,\sigma}(\sM; T\sM) = (\Xw_0, \Xw_1)_{\mu_r-1/r,r}
=:\Xw_{\gamma,\mu_r},
\quad\text{with} \ \ \mu_r=1/r+ n/2q.
\end{equation*}
Note that $\mu_r<1$.
We can now consider problem  \eqref{NS sys abstract} with initial value $u_0\in \Xw_{\gamma,\mu_r}$.
By  uniqueness and    \cite[Theorem~2.1]{PrSiWi18},
we conclude that the solution $u$ regularizes and satisfies
\begin{equation*}
u(t_0)\in B^{1-2/r}_{qr,\sigma}(\sM; T\sM)=(\Xw_0, \Xw_1)_{1-1/r,r}=:\Xw_{\gamma,1}
\end{equation*}
for any $t_0\in (0,t_{\max})$.
Choosing $\mu\in (1/p,1/2+1/2p)$,  we have the embedding
$$
B^{1-2/r}_{qr ,\sigma} (\sM;T\sM) \hookrightarrow B^{2\mu-2/p}_{qp ,\sigma} (\sM;T\sM)
$$
at our disposal. Next, we note that
$$
\|F(u)\|_{L_q(\sM)} \leq C \|u\|_{L_\infty(\sM)} \|u\|_{H^1_q(\sM)}
$$
for all
$$
u\in (X_0,X_1)_{\beta,p} \subset B^{2\beta}_{qp}(\sM;T\sM) \hookrightarrow L_\infty (\sM;T\sM) \cap H^1_{q } (\sM;T\sM),
$$
provided $2\beta>1$ and $q\ge n$.
Now we can solve \eqref{NS sys abstract strong} with initial value $u(t_0)\in B^{2\mu-2/p}_{qp ,\sigma} (\sM;T\sM)$
to obtain a strong solution by using \cite[Theorem~2.1]{LePrWi14}.
The asserted regularity~\eqref{additional reg} now follows from uniqueness of solutions.
\end{proof}

The following plot is helpful to illustrate the results in Theorem~\ref{Thm: local wellposed}.
\begin{figure}[hbt]
\includegraphics[scale=.25, angle=-90]{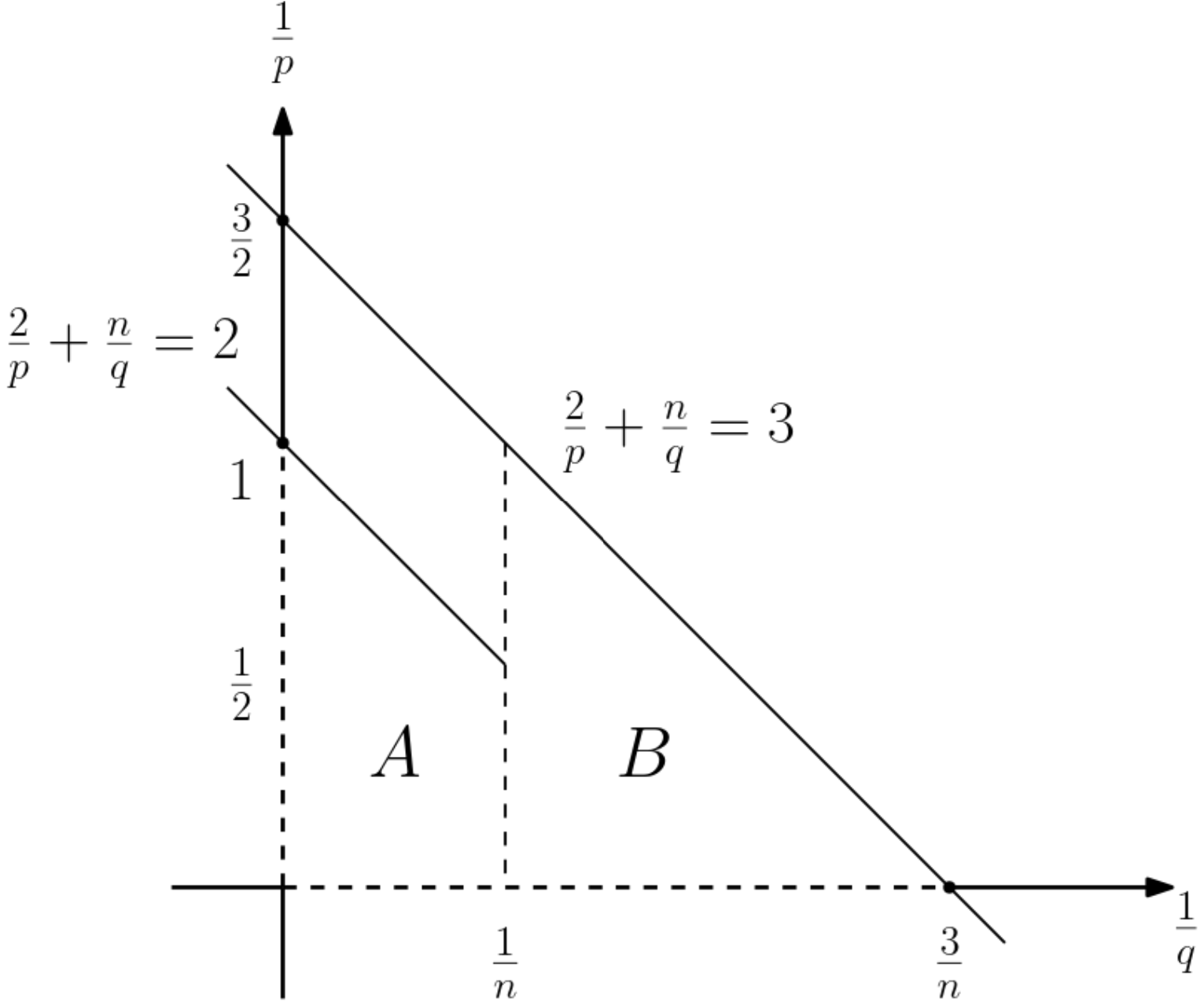}
\end{figure}
When $(\frac{1}{q}, \frac{1}{p})$ is
\medskip
\begin{itemize}
\item in region $A$ we have weak solutions, which immediately regularize  and become  strong solutions;
\item in region $B$, we have strong solutions.
\end{itemize}
\goodbreak

\medskip\noindent
The proof of Theorem~\ref{Thm: local wellposed} has an immediate byproduct.

\begin{corollary}
\label{Cor:existence}
Let $p\in (1,\infty) $,  $q\in [n,\infty)$ and $\mu\in (1/p,1]$.
Then for any initial value $u_0\in (X_0, X_1)_{\mu -1/p,p}$, where
$$
(X_0, X_1)_{\mu -1/p,p} = B^{2\mu-2/p}_{qp ,\sigma,\cB} (\sM;T\sM)\quad\text{in case}\quad 2\mu - 2/p \neq 1+1/q,
$$
there exists a unique strong solution
$$
u\in H^1_{p,\mu }((0,t_+); L_{q,\sigma}(\sM;T\sM))\cap L_{p,\mu }((0,t_+); H^2_{q,\sigma}(\sM;T\sM))
$$
of \eqref{NS sys abstract strong}
for some $t_+=t_+(u_0)>0$.
The solution exists on a maximal time interval $[0, t_{\max}(u_0))$ and depends continuously on $u_0$.
Moreover,
$$
u \in C([0, t_{\max}); B^{2\mu-2/p}_{qp ,\sigma} (\sM;T\sM)) \cap C((0, t_{\max}); B^{2-2/p}_{qp ,\sigma} (\sM;T\sM)).
$$
\end{corollary}

\begin{remark}
\label{Rmk: two cases wellposed}
Concerning Theorem \ref{Thm: local wellposed}, two cases are of particular interest.
\begin{itemize}
\item[(i)] Suppose that $n\geq 2$. Then for every $u_0\in L_{n,\sigma}(\sM;T\sM)$, \eqref{NS sys abstract} has a unique solution satisfying
the regularity properties stated in  Theorem~\ref{Thm: local wellposed}(a)  with $q=n$ for each   fixed $p\ge n$.
Therefore, Theorem \ref{Thm: local wellposed} reproduces the celebrated results by Giga and Miyakawa \cite{GigaMiyakawa85}
 (obtained for  no-slip boundary conditions) for Navier boundary conditions.
\item[(ii)] Suppose that $n= 2,3$. Choosing $p=q=2$
we can admit initial values $u_0\in H^{n/2-1}_{2,\sigma}(\sM;T\sM)$.
This generalizes the celebrated results by  Fujita and Kato \cite{KatoFujita62, KatoFujita64}.

In particular, if $n=2$, for any $u_0\in L_{2,\sigma}(\sM;T\sM)$, \eqref{NS sys abstract} has a unique solution satisfying
the regularity properties stated in Theorem~\ref{Thm: local wellposed}(a) with $q=p=2$. Moreover, the solution satisfies
$$
u\in H^1_{p,loc}((0,t_{\max}); L_{q,\sigma}(\sM;T\sM))\cap L_{p,loc}((0,t_{\max}); H^2_{q,\sigma}(\sM;T\sM))
$$
 for any fixed $p,q>1$.
\end{itemize}
\end{remark}
\begin{proof}
The assertions can be shown by following the proofs of \cite[Corollary 4.4 and Theorem~4.5]{PrSiWi21} line by line.
\end{proof}

\section{Large time behavior}
\label{Section:large time behavior}
\subsection{Characterization of equilibria}
We will begin the analysis of large time behavior by a characterization of the spectrum of $\Aw_N$.
Since $X_{1/2}$ is compactly embedded in $X_{-1/2}$, the spectrum $\sigma(\Aw_N)$ consists only of isolated eigenvalues and is independent of  the choice of $q$.
By  Green's first identity,   Lemma~\ref{Appendix Lem: divergence thm}, one obtains for all   $(u,v) \in \sD(A_{N,q}) \times H^1_{q',\sigma}(\sM;T\sM)$
\begin{equation}\label{weak formulation Navier 2}
\begin{split}
  \la \Aw_N u | v\ra_{\sM} &=  (A_N u | v)_{\sM}\\
&= 2\mu_s ( D_u  | D_v)_{\sM}  - \mu_s (\cP_\Sigma (\nabla u + [\nabla  u]^{\sT} )  \nu_\Sigma  | v)_{ \Sigma } \\
&= 2\mu_s ( D_u  |  D_v)_{\sM} +\alpha \mu_s  (\cP_\Sigma u | \cP_\Sigma v)_{ \Sigma }\\
&=  2\mu_s ( D_u | D_v)_{\sM} +\alpha \mu_s  (  u |   v)_{ \Sigma }.
\end{split}
\end{equation}
By a density argument, one readily sees that \eqref{weak formulation Navier 2} holds for all $(u,v)\in H^1_{q ,\sigma}(\sM;T\sM) \times H^1_{q' ,\sigma}(\sM;T\sM)$.
Suppose
$$
\Aw_N u = \lambda u,
$$
for some $\lambda \in \mathbb C$ and $u\in H^1_{2,\sigma}(\sM;T\sM_\bC)$, with $T\sM_\bC=T\sM+i T\sM$ denoting the complexified  tangent bundle.
 \eqref{weak formulation Navier 2} implies
\begin{equation}
\label{eigenvalue of ANw}
\lambda\|u\|_{ L_2(\sM) }^2= 2\mu_s\| D_u\|^2_{ L_2(\sM) } +\alpha \mu_s \|u\|_{L_2(\Sigma)}^2,
\end{equation}
and thus, $\lambda\geq 0$, i.e., $\sigma(\Aw_N)\subset [0,\infty)$.
It follows that $\omega+\Aw_N \in \cS(L_{q,\sigma}(\sM; T\sM))$ is invertible with spectral angle $\phi_{\omega+\Aw_N} <\pi/2$ for all $\omega>0$.
Applying  one more time \cite[Corollary~3.3.15]{PruSim16}  and taking advantage of \eqref{cHi ANw} yields
$\omega+ \Aw_N \in H^\infty(X_{-1/2}) $ with $H^\infty$-angle $< \pi/2$, for all $\omega>0$.
Following the discussion in Section~\ref{Section:NS with Navier boundary condition}, it is not hard to see that the same holds true for $A_N$.

Moreover, \eqref{eigenvalue of ANw} shows that $\sN(\Aw_N)\subseteq \cE_\alpha$, which is defined by
$$
\cE_\alpha=
\begin{cases}
 \{u\in H^1_{2,\sigma}(\sM;T\sM): \,    D_u =0  \text{ on }\sM \quad \text{and}\quad u=0 \text{ on }\Sigma\} \quad & \text{if } \alpha>0 \\
 \{u\in H^1_{2,\sigma}(\sM;T\sM): \,    D_u=0  \text{ on }\sM \quad \text{and}\quad (u|\nu_\Sigma)_g =0 \text{ on }\Sigma\} & \text{if } \alpha=0.
\end{cases}
$$
Conversely, if $u\in \cE_\alpha$, \eqref{weak formulation Navier 2} implies that
$\Aw_N u =0$.
Hence, $\sN(\Aw_N)=\cE_\alpha$.

\begin{remark}\label{Rmk:Killing}
Any element $u \in \cE_\alpha$ is a Killing vector field  on  $\sM$, that is, it satisfies
\begin{equation}
\label{Def killing}
(\nabla_v u | w)_g+ (\nabla_w u| v)_g =0 ,\quad \forall  v,w\in  C^\infty(\sM;T\sM).
\end{equation}
\end{remark}
\begin{proposition}\label{Prop: characterizing Ealpha}
$$
\cE_\alpha=
\begin{cases}
 \{0\} \quad & \text{if } \alpha>0 \\
 \{u\in C^\infty({\sM};T\sM): \,  D_u=0 \text{ on }\sM \quad \text{and}\quad (u|\nu_\Sigma)_g =0 \text{ on }\Sigma\} & \text{if } \alpha=0.
\end{cases}
$$
\end{proposition}
\begin{proof}
Suppose $u\in H^1_{2,\sigma}(\sM ; T\sM )$ and $D_u=0$. Then $(D_u)_\flat=0$ as well.
Let $u=u^k \frac{\partial }{\partial x^k}$ be a representation of $u$ in local coordinates.
Then, in local coordinates,
$$
(D_u)_\flat= \left(g_{ki} u^k_{ | j} + g_{kj} u^k_{| i} \right) dx^i \otimes dx^j =: \left( u_{i | j} + u_{j| i} \right) dx^i \otimes dx^j.
 $$
The relation $(D_u)_\flat =0$ then reads
$$
u_{i | j} + u_{j| i}= \frac{\partial u_i} {\partial x^j} + \frac{\partial u_j} {\partial x^i } - 2\Gamma^\ell_{ij} u_\ell=0,
\quad 1\le i,j\le n.
$$
Arguing as in the proof of  \cite[Lemma~3]{Pri94},
we conclude that $u\in C^\infty(\sM; T\sM)$.

\medskip
Suppose now that $\alpha>0$. Let  $u\in \cE_\alpha$. Since $u=0$ on $\Sigma$, one can immediately infer that
$  \nabla_v  u  =0$ on $\Sigma$
 for any  $v\in C^\infty(\Sigma;  T\Sigma) $.
To show   $ \nabla_{\nu_\Sigma} u  =0$, we first observe that
$( \nabla_{\nu_\Sigma} u | \nu_\Sigma)_g =0,$
as $u$ is a Killing field, see \eqref{Def killing}.
Next, we get for any  $v\in C^\infty(\Sigma; T\Sigma) $
$$ ( \nabla_{\nu_\Sigma} u | v)_g = -(  \nabla_v  u  | \nu_\Sigma)_g=0 ,$$
where we use, once more, the property that $u$ is a Killing field  and  $ \nabla_v u =0$.
This yields $ \nabla_v u =0$ for each vector field $v\in C^\infty(\Sigma; T\sM)$, and hence $\nabla u=0$  on $\Sigma$.
It then follows from \cite[Chapter 7, Proposition 28]{Petersenbook} that $u=0$.
\end{proof}
 \goodbreak

\begin{remark}\label{Rmk:Killing 2}
Although, in general, $\cE_0 $ is non-trivial,
$\cE_0 =\{0\}$    holds under some specific conditions,  for instance in case
\begin{itemize}
\item[(i)] $\Ric^\sharp  < 0$ and  $L_\Sigma \geq 0$, or
\item[(ii)] $\Ric^\sharp \leq  0$ and $L_\Sigma> 0$,
\end{itemize}
see the proof of Lemma~\ref{Appendix Lem: korn},    and also\cite{Yano59, YanoAko65}.
\end{remark}

\begin{proposition}\label{Prop: equilibrium set}
$\cE_\alpha$ is the set of equilibria of \eqref{NS sys abstract},   respectively \eqref{NS sys abstract strong}.
That is, the set of equilibria  of \eqref{NS sys abstract},   respectively \eqref{NS sys abstract strong},
are exactly the Killing vector fields as characterized in Proposition~\ref{Prop: characterizing Ealpha}.

\end{proposition}
\begin{proof}
Assume that $u_*$ is an equilibrium of \eqref{NS sys abstract}, i.e., $u\in \sD(A_N)$ and
$$
A_N u_* = F(u_*).
$$
Using the metric property of $(\cdot | \cdot)_g$
and the fact that $(u_* | \nu_\Sigma)_g=0$, it follows from \eqref{div-scalar} that
\begin{align*}
(F(u_*)|u_*)_{\sM }=- ( \nabla_{u_*} u_*  | u_*)_{ \sM } = -\frac{1}{2} \int_{\sM} \nabla_{u_*} |u_*|^2_g\, d \mu_g=0.
\end{align*}
Following the computations in \eqref{weak formulation Navier 2}, we get
$$
0= (F(u_*) | u_*)_{ \sM }=(\Aw_N (u_*) | u_*)_{ \sM }  =2\mu_s \| D_{u_*} \|_{ L_2(\sM) }^2 +\alpha \mu_s  \|  u_*\|_{L_2(\Sigma)}^2.
$$
Therefore, $u_*\in \cE_\alpha$.

Conversely, if $u_*\in \cE_\alpha$, then it is clear that
$A_N u_* = 0$. On the other hand, by the definition of $\PH$, one can show that
\begin{align*}
\PH \nabla_{u_*} u_* = \PH \left( (\nabla u_*) u_*  \right)= - \PH \left( [\nabla u_*]^\sT u_*  \right)
 = - \frac{1}{2} \PH \left( \gd  |u_*|_g^2 \right)=0.
\end{align*}
This shows that $u_*$ is an  equilibrium   of \eqref{NS sys abstract strong}.
\end{proof}

With the convention that $H^0_{2,\sigma}(\sM;T\sM):=L_{2,\sigma}(\sM;T\sM)$, we put
\begin{equation}
\label{E-orth}
V^j_2=\{u\in H^j_{2,\sigma}(\sM;T\sM): \, ( u|v)_{\sM}=0 \quad \forall v\in \cE_\alpha\}, \quad j=0,1.
\end{equation}
In \cite[Lemma~6]{Pri94}, it was shown that $\cE_\alpha$ is a finite dimensional space.
Thus, $V^j_2$ is a closed subspace of $H^j_{2,\sigma}(\sM;T\sM)$ and
\begin{equation}
\label{direct sum}
H^j_{2,\sigma}(\sM;T\sM)= \cE_\alpha \oplus V^j_2, \quad j=0,1,
\end{equation}
by a similar argument to \cite[Remark~4.10(a)]{SiWi22}.


\subsection{Global existence and convergence for 2D}

In  this subsection, we consider the case $n=2$.
Given any $u_*\in \cE_\alpha$, we consider the evolution equation
\begin{equation}
\label{NS sys abstract-aux}
\left\{\begin{aligned}
\partial_t u + A_N u &=G_*(u):=-\PH (\nabla_u u + \nabla_u u_* + \nabla_{u_*} u ),     && t>0 ,\\
u(0) & = u_0 &&
\end{aligned}\right.
\end{equation}
and its weak counterpart
\begin{equation}
\label{NS sys abstract-aux-weak}
\left\{\begin{aligned}
\partial_t u + \Aw_N u &=G^{\sw}_*(u),    && t>0 ,\\
u(0) & = u_0, &&
\end{aligned}\right.
\end{equation}
where
$$\langle G^{\sw}_*(u) | \phi\ra_\sM = (u\otimes u_\flat | \nabla \phi)_\sM + (u_* \otimes u_\flat | \nabla \phi)_\sM + (u\otimes (u_*)_\flat | \nabla \phi)_\sM,
\quad \phi\in \dot H^1_{q^\prime,\sigma}(\sM; T\sM).$$

\medskip
Note that, by choosing $p=q=2$, the critical weight is $\mwc=1$ and the corresponding critical trace space is
$$
\Xwm=(X_{-1/2},X_{1/2})_{1/2,2}=[X_{-1/2},X_{1/2}]_{1/2}=L_{2,\sigma}(\sM;T\sM),
$$
see \eqref{interpolation of X}.
Hence, by applying a similar argument to Remark~\ref{Rmk: two cases wellposed}(ii), one can immediately infer that for every $u_0\in L_{2,\sigma}(\sM;T\sM)$,  \eqref{NS sys abstract-aux-weak} has a unique solution
\begin{equation*}
u\in H^1_2((0,t_+); H^{-1}_{2,\sigma}(\sM;T\sM))\cap L_2((0,t_+); H^1_{2,\sigma}(\sM;T\sM))
\end{equation*}
for some $t_+=t_+(u_0)>0$.
The solution exists on a maximal time interval $[0,t_{\max}(u_0))$.
In addition, it holds that
\begin{equation}
\label{2D-regularization}
u\in H^1_{p,loc}((0,t_{\max}); L_{q,\sigma}(\sM;T\sM))\cap L_{p,loc}((0,t_{\max}); H^2_{q,\sigma}(\sM;T\sM))
\end{equation}
 for any fixed $p,q\in (1,\infty)$,  and $u$ also solves  \eqref{NS sys abstract-aux}.

Next we show that any solution of \eqref{NS sys abstract-aux-weak} with initial value $u_0\in L_{2,\sigma}(\sM)$ that is orthogonal to $\cE_\alpha$ remains orthogonal
for all later times. Moreover, we establish an energy estimate for such solutions.
\goodbreak
\begin{lemma}\label{Lem: aux global}
Assume that $n=2$. Given $u_0\in V^0_2$,  let $u$ be the unique solution of \eqref{NS sys abstract-aux-weak}. Then
\begin{itemize}
\item[{\em (a)}] $u(t)\in V^1_2$ for all $t\in (0,t_{\max}(u_0))$;
\item[{\em (b)}] there exists a constant $C>0$ such that
\begin{equation}
\label{integral ineq}
\|u(t)\|_{ L_2(\sM) }^2 + C\int_0^t \|u(s)\|_{H^1_2(\sM)}^2\, ds \leq \|u_0\|_{ L_2(\sM) }^2,\quad t\in (0,t_{\max}(u_0)) ;
\end{equation}
\item[{\em (c)}] $t_{\max} (u_0)=+\infty$. Moreover, there exists a constant $\beta>0$ such that
\begin{equation}
\label{exp growth}
\|u(t)\|_{ L_2(\sM) }\leq e^{-\beta t} \|u_0\|_{ L_2(\sM) },\quad t\ge 0.
\end{equation}
\end{itemize}
\end{lemma}
\begin{proof}
(a) Pick any  $z\in \cE_\alpha$.

In the sequel, we suppress the time variable and simply write $u$ in lieu of $u(t)$.
Following the computations in \eqref{weak formulation Navier 2}, we have
\begin{equation}
\label{weak formulation Navier 3}
 \langle A_N u | z \rangle_{\sM}
=   2\mu_s ( D_u | D_z)_{\sM} +\alpha \mu_s  (  u |   z)_{ \Sigma }=0.
\end{equation}
Moreover, it holds that
\begin{equation}
\label{weak formulation Navier 4}
 \langle G_*(u) | z \rangle_{\sM}
=  ( \nabla_u u | z)_{\sM} +  ( \nabla_{u_*} u | z)_{\sM} + ( \nabla_u u_* | z)_{\sM}=0 .
\end{equation}
We note  that in  \eqref{weak formulation Navier 3} and \eqref{weak formulation Navier 4} we may use the `strong' operators $A_N$ and $G_*$,
as solutions immediately regularize, see \eqref{2D-regularization}.
To show \eqref{weak formulation Navier 4}, we employ the metric property to obtain
\begin{align*}
  ( \nabla_{u_*} u | z)_g + ( \nabla_u u_* | z)_g
=    \nabla_{u_*}(  u | z)_g + \nabla_u (  u_* | z)_g  - (  u | \nabla_{u_*} z)_g - (  u_* | \nabla_u z)_g .
\end{align*}
Since $z$ is a Killing vector field, we infer that
\begin{equation}\label{weak formulation Navier 5}
(  u | \nabla_{u_*} z)_g + (  u_* | \nabla_u z)_g =0,
\end{equation}
see \eqref{Def killing}.
Meanwhile,    Lemma~\ref{Appendix Lem: divergence thm}, implies
$$
\int_\sM  \left[  \nabla_{u_*}(  u | z)_g + \nabla_u (  u_* | z)_g  \right] \, d\mu_g
 =0.
$$
Similar computations show that
\begin{align*}
( \nabla_u u | z)_{\sM}=  0.
\end{align*}
Combining \eqref{weak formulation Navier 3} and \eqref{weak formulation Navier 4} yields
$$
0= \la \partial_t u (t) | z \ra_{\sM} = \partial_t ( u (t)| z)_{\sM}.
$$
Hence $( u (t)  | z)_{\sM}=0$ and $u(t) \in V^1_2$ for all $t\in (0,t_{\max}(u_0))$.

(b)  Due to  \eqref{2D-regularization}, $u$ is a valid test function in \eqref{NS sys abstract-aux}.
Multiplying \eqref{NS sys abstract-aux}$_1$ by $u$ and integrating over $\sM$ yields
\begin{align*}
\frac{d}{dt} \|u(t)\|_{ L_2(\sM) }^2 = -4\mu_s \| D  (u)\|_{  L_2(\sM) }^2 -2\alpha \mu_s  \|  u \|_{L_2(\Sigma)}^2
+2 (G_*(u) | u)_{\sM}.
\end{align*}
Following a similar computation as in Part (a) and using the fact that $u_*$ is a Killing vector field, one can show that
$$
(G_*(u) | u)_{\sM}=0.
$$
We thus have
\begin{equation}
\label{integral ineq 2}
\frac{d}{dt} \|u(t)\|_{ L_2(\sM) }^2 =   -4\mu_s \| D  (u)\|_{ L_2(\sM) }^2 -2\alpha \mu_s  \|  u \|_{L_2(\Sigma)}^2
\leq   -C \|u\|_{H^1_2(\sM)}^2 ,
\end{equation}
where the last step follows from  Korn's inequality, cf. Lemma~\ref{Appendix Lem: korn}.
Integrating both side with respect to time gives \eqref{integral ineq}.

(c) Part (b) shows that
$$
u\in L_2((0,t_{\max}(u_0)); H^1_{2,\sigma} (\sM;T\sM)) .
$$
It follows from \cite[Theorem~2.4]{PrSiWi18} that $t_{\max}(u_0)=+\infty$.
An immediate consequence of \eqref{integral ineq 2}  is
$$
\frac{d}{dt} \|u(t)\|_{ L_2(\sM) }^2 + C \|u\|_{L _2(\sM)}^2  \leq 0 ,\quad \forall t>0.
$$
Solving the above ordinary differential inequality  gives \eqref{exp growth}.
\end{proof}

Now we are in a position to prove the main theorem of this subsection.
\begin{theorem}
\label{Thm: gloabl existence and convergence 2D}
Let $n=2$. Then for every $u_0\in L_{2,\sigma}(\sM;T\sM)$, the unique solution $u$ to \eqref{NS sys abstract}
with initial value $u_0$ exists globally and enjoys the regularity properties listed in Remark~\ref{Rmk: two cases wellposed}(ii).
Furthermore, for any fixed $q\in (1,\infty)$, $u$ converges to the equilibrium
$u_*:= \cP_{\cE_\alpha} u_0$  in the topology of $H^2_{q,\sigma}(\sM;T\sM)$  at an exponential rate as $t\to \infty$,
where $\cP_{\cE_\alpha}$ denotes the orthogonal projection from $L_{2,\sigma}(\sM;T\sM)$ onto $\cE_\alpha$.
\end{theorem}
\begin{proof}
In view of  \eqref{direct sum}, we can decompose $u_0$ into $u_0=u_* + v_0$ such that $v_0\in V^0_2$.
Let $v(t)$ be the (unique) solution to
\begin{equation*}
\left\{\begin{aligned}
\partial_t v + \Aw_N v &=G^{\sw}_*(v),     && t>0 ,\\
v(0) & = v_0 .&&
\end{aligned}\right.
\end{equation*}
By Lemma~\ref{Lem: aux global}, $v$ exists globally. Then it follows from Proposition~\ref{Prop: equilibrium set} that
$$
u(t) = u_*+v(t)
$$
is the unique global solution of \eqref{NS sys abstract} with initial value $u_0$.
As was proved in Lemma~\ref{Lem: aux global},
$$
\|u(t)-u_*\|_{ L_2(\sM) } = \|v(t) \|_{ L_2(\sM) } \leq e^{-\beta t}\|v_0 \|_{ L_2(\sM) }
=e^{-\beta t}\|u_0 - u_* \|_{ L_2(\sM) },\quad t>0,
$$
for some $\beta>0$.
The convergence in the stronger topology $H^2_{q,\sigma}(\sM;T\sM)$ can be proved in the same way as in \cite[Theorem~4.9]{SiWi22}.
\end{proof}

\subsection{Stability near Killing vector fields}

In this subsection, we will establish the stability of solutions of \eqref{NS sys abstract}, respectively \eqref{NS-sys}, with initial values close to a Killing vector field
 for $n>2$.

For  any fixed $u_*\in \cE_\alpha$, the linearization of the operator $[u\mapsto (\Aw_N u - \Fw(u))]$ is given by the operator
$\Aw_0: X_{1/2} \to X_{-1/2}$, defined by
\begin{align*}
\la \Aw_0 u | v\ra_{\sM} =  \la \Aw_N u | v\ra_{\sM} -   (u\otimes (u_*)_\flat + u_*\otimes u_\flat | \nabla v)_{\sM}
\end{align*}
for all $(u,v)\in H^1_{q,\sigma}(\sM;T\sM) \times H^1_{q',\sigma}(\sM;T\sM)$.
In other words, $\Aw_0=\Aw_N + B $, where $B$ is the linear operator from $ L_{q,\sigma}(\sM;T\sM)$
to $X_{-1/2}$ defined by
$$
\la B u | v\ra_{\sM}  = -(u\otimes (u_*)_\flat + u_*\otimes u_\flat | \nabla v)_{\sM} ,\quad  v\in H^1_{q',\sigma}(\sM;T\sM).
$$
Proposition~\ref{Prop: characterizing Ealpha} shows that $\cE_\alpha \subset C^\infty({\sM}) $.
Direct computations yield
\begin{align*}
\left|  \la B u | v\ra_{\sM} \right|  & \le C  \|u\|_{L_q(\sM)} \|v\|_{H^1_{q'}(\sM)}.
\end{align*}
Therefore, $B\in \cL(L_{q,\sigma}(\sM;T\sM),  X_{-1/2} )$.
From \cite[Corollary~3.3.15]{PruSim16}, we infer that for some sufficiently large $\omega_0>0$
$$
\omega+ \Aw_0 \in H^\infty( X_{-1/2} ) \text{ with $H^\infty$-angle } < \pi/2 \quad \text{for all  }\omega>\omega_0.
$$
Let $A_0\in \cL(\sD(A_{N,q}), L_{q,\sigma}(\sM;T\sM) )$ be the operator defined by
$$
A_0 u = 2\mu_s \PH \dv D  (u)     + \PH \left( \nabla_u u_* + \nabla_{u_*} u \right).
$$
By applying a similar argument to $A_N$, we can show that by possibly further increasing $\omega_0>0$
$$
\omega+ A_0 \in H^\infty(L_{q,\sigma}(\sM;T\sM)) \text{ with $H^\infty$-angle }  < \pi/2 \quad \text{for all  }\omega>\omega_0.
$$
Since $X_{1/2}$ is compactly embedded in $X_{-1/2}$, the spectrum of $\Aw_0$ consists only of isolated eigenvalues
and is independent of  the choice of $q$. Suppose
$$
\Aw_0 u = \lambda u,
$$
for some $\lambda\in \mathbb C$.
Following the computations in \eqref{weak formulation Navier 2}, it is not difficult to check that
\begin{equation}
\label{weak formulation Navier 6}
\begin{split}
 {\rm Re}\, \lambda \|u\|_{ L_2(\sM) }^2
&=  {\rm Re}  ( \langle  \Aw_0 u | \overline{u}\ra _{\sM} ) \\
&=   2\mu_s\| D_u \|_{ L_2(\sM) }^2 +\alpha \mu_s  \|u\|_{L_2(\Sigma)}^2 + {\rm Re} (B u  |   \overline{u})_{\sM}.
\end{split}
\end{equation}
We observe that for all   $v\in H^1_{2,\sigma}(\sM;T\sM_\bC)$, with $T\sM_\bC$ denoting the complexified  tangent bundle,
\begin{equation*}
\begin{aligned}
{\rm Re}(B v  |  \overline{v})_g
&= - {\rm Re}\, (v\otimes (u_*)_\flat + u_*\otimes v_\flat | \nabla \overline v  )_g = -{\rm Re}\, [ ( u_* |  \nabla_v \overline v)_g  + ( v | \nabla_{u_*} \overline v )_g] \\
&= -{\rm Re}\,\nabla_v (u_* | \overline v)_g - \frac{1}{2} \nabla_{u_*} |v |^2_g  + {\rm Re}\, (\nabla_v u_*  | \overline v)_g,
\end{aligned}
\end{equation*}
where we used the metric property of $(\cdot | \cdot)_g$.
It follows from \eqref{div-scalar} that
$$
\int_{\sM}\left( {\rm Re}\nabla_v (u_* | \overline v)_g  - \frac{1}{2} \nabla_{u_*} |v |^2_g\right) \, d\mu_g =0.
$$
Finally, the definition of Killing vector fields implies
$$
{\rm Re} (   \nabla_v u_*    | \overline{v} )_g =0.
$$
Therefore,  ${\rm Re} (B u  |   \overline{u})_{\sM} =0$ and this shows that ${\rm Re} \lambda \geq 0$.
When ${\rm Re} \lambda=0$, one can infer from \eqref{weak formulation Navier 6} that
$ u \in \cE_\alpha$. This implies $\sN(\Aw_0)\subseteq \cE_\alpha.$

Conversely, if $z\in \cE_\alpha$, then for any $v\in H^1_{q',\sigma}(\sM;T\sM)$, the above computations show
\begin{align*}
 \langle  \Aw_0 z | v\ra_{\sM} =  ( \PH (  \nabla_z u_* + \nabla_{u_*} z )  |v )_{\sM}.
\end{align*}
As $ D_{u_*}= D_z=0$,  we obtain
\begin{align*}
\PH (  \nabla_z u_* + \nabla_{u_*} z ) &= \PH (  (\nabla u_*)z + (\nabla z ) u_*  )
= - \PH (  (\nabla u_*)^\sT z + (\nabla z )^\sT u_*  ) \\
& =    - \PH \gd (u_* | z)_g   =0
\end{align*}
in virtue of the definition of $\PH$.
This implies that $z\in  \sN(\Aw_0)$.
In summary, we conclude that
$$
\sN(\Aw_0)=\cE_\alpha
$$
and $\sigma(\Aw_0)\cap i \bR=\{0\}.$
Next, we will show that the eigenvalue $0$ of $\Aw_0$ is semi-simple. Indeed, if
$$
\Aw_0 u =z \in \cE_\alpha,
$$
then it follows from similar computations as in \eqref{weak formulation Navier 3} and \eqref{weak formulation Navier 5} that
\begin{align*}
\|z\|_{ L_2(\sM)  }^2 & =  \la \Aw_0 u | z\ra_{\sM}  \\
&=  \la \Aw_N u | z \ra_{\sM}   - (\PH (\nabla_u u^* + \nabla_{u^*} u )|z)_{\sM}\\
&= 2\mu_s (D_u |D_z )_{\sM}+ \alpha \mu_s (u|z)_\Sigma - ( \nabla_u u^*   |z)_{\sM} -(   \nabla_{u^*} u  |z)_{\sM} = 0.
\end{align*}
This shows that $z=0$ and thus, $\sN(\Aw_0)=\sN((\Aw_0)^2)$.
As $\cE_\alpha$ is a linear space, we clearly have $T_{u_*} \cE_\alpha=\sN(\Aw_0)$.
From Proposition~\ref{Prop: equilibrium set} and \cite{PruSim16}, we learn that $\Aw_N$ is normally stable.
So we can apply \cite[Theorem~5.3.1]{PruSim16} to obtain the following theorem.

\begin{theorem}
\label{Thm: stability near killing 1}
Suppose  that $n>2$, $p\in (1,\infty) $ and $q\in (n/2,\infty)$ such that $\frac{2}{p}+ \frac{n}{q}\leq 2$. \\
Then  for each  $u_*\in \cE_0$, there exists some $\delta=\delta(u_*)>0$ such that the solution $u$ of \eqref{NS sys abstract} with
initial value $u_0\in B^{n/q-1}_{qp,\sigma}(\sM;T\sM)$ satisfying
$$
\|u_0-u_*\|_{B^{n/q-1}_{qp}} \leq \delta
$$
exists globally and converges at an exponential rate to some  $z\in \cE_0$
\begin{itemize}
\item[{\em (i)}]    in the topology of $B^{1-2/p}_{qp }(\sM;T\sM)$,
\item[{\em (ii)}]   in the topology of $B^{2-2/p}_{qp }(\sM;T\sM)$ if  $q\geq n$.
\end{itemize}
\end{theorem}
\begin{proof}
According to Theorem~\ref{Thm: local wellposed},  problem \eqref{NS sys abstract} has for each  $u_0\in B^{n/q-1}_{qp,\sigma}(\sM;T\sM)$
a unique solution $u$ in the regularity class asserted by the theorem.
In order to show assertions  (i) and (ii), we will employ \cite[Theorem~5.3.1]{PruSim16}
for initial values in  $(\Xw_0,\Xw_1)_{1-1/r,r}$
for the weak setting, or in $ (X_0,X_1)_{1-1/r,r}$ for the strong setting,
 with $r$ properly chosen.

\medskip\noindent
(i) We will first show that any solution $u$ to \eqref{NS sys abstract} with initial value $u_0$ close to $u_*$
 in $B^{n/q-1}_{qp,\sigma}(\sM; T\sM)$ will also be close to $u_*$ in $(\Xw_0, \Xw_1)_{1-1/r,r},$  for any fixed positive
 (sufficiently small) time and appropriate $r>p$.

 Suppose $r>p$. As in the proof  of Theorem \ref{Thm: local wellposed}, we have
\begin{equation}
\label{critical-2}
B^{n/q-1}_{qp,\sigma}(\sM; T\sM)\hookrightarrow B^{n/q-1}_{qr,\sigma}(\sM; T\sM) = (\Xw_0, \Xw_1)_{\mu_r-1/r,r}
=:\Xw_{\gamma,\mu_r},
\end{equation}
where $ \mu_r=1/r + n/2q$. Moreover,
$$u(t_0)\in B^{1-2/r}_{qr ,\sigma} (\sM;T\sM)= (\Xw_0, \Xw_1)_{1-1/r,r}=:\Xw_{\gamma,1}$$
for  any fixed time $t_0\in (0, t_{\max})$.
Using  Lipschitz continuity of solutions with respect to initial data 
and the regularization property, there exists a positive number $t_0$ and a
constant $C(t_0)$ such that
\begin{equation}
\label{small-1}
 \|u(t_0)-u_*\|_{B^{1-2/r}_{qr}}\le C(t_0) \|u_0- u_*\|_{B^{n/q-1}_{qr}}
\le C(t_0) \|u_0- u_*\|_{B^{n/q-1}_{qp}},
\end{equation}
for any initial value $u_0$ sufficiently close to $u_*$ in $B^{n/q-1}_{qp,\sigma}(\sM; T\sM)$.
Indeed, as solutions to \eqref{NS sys abstract} depend Lipschitz continuously on the initial data, see \cite[Theorem 1.2]{PrSiWi18},
there are numbers $t_0$ and $M>0$ such that
\begin{equation}
\label{first control}
\| u  -u_*  \|_{\bE^{\sw}_{1,\mu_r}(0,2t_0)} \leq M \|u_0 -u_* \|_{\Xw_{\gamma, \mu_r}}
\end{equation}
for any initial value $u_0$ sufficiently close to $u_*$ in
$\Xw_{\gamma,\mu_r}.$
Here we have set
$$
\bE^{\sw}_{1,\mu}(T_1,T_2) : =H^1_{p,\mu}((T_1,T_2); \Xw_0)\cap L_{p,\mu}((T_1,T_2); \Xw_1),$$
for $0\le T_1<T_2<\infty$.
Since
$ \bE^{\sw}_{1,\mu}(t_0,2t_0)\hookrightarrow \bE^{\sw}_{1,1}(t_0,2t_0)\hookrightarrow BU\!C((t_0,2t_0); X^\sw_{\gamma,1})$
for any $\mu\in (1/p, 1]$,
we obtain with \eqref{first control}
\begin{equation}
\label{2nd control}
\begin{aligned}
\|u(t_0)-u_*\|_{X^\sw_{\gamma,1}} &\le \sup_{t\in [t_0, 2t_0]} \|u(t)-u_*\|_{X^{\sw}_{\gamma,1}}\le
C \|u-u_*\|_{ \bE^{\sw}_{1,1}(t_0,2t_0) } \\
& \le  C t_0^{\mu_r-1}  \|u-u_*\|_{ \bE^{\sw}_{1,\mu_r}(t_0,2t_0) }
\le  C(t_0)  \|u_0 -u_* \|_{\Xw_{\gamma,\mu_r}}.
\end{aligned}
\end{equation}
The assertion in \eqref{small-1} follows now from \eqref{critical-2} and \eqref{2nd control}.

Hence  $\|u(t_0) -u_*\|_{B^{1-2/r}_{qr}}$ can be made as small as we wish by making $\|u_0- u_*\|_{B^{n/q-1}_{qp}} $ small.
It follows from the embedding
$$
B^{1-2/r}_{qr,\sigma}(\sM;T\sM) \hookrightarrow L_{2q,\sigma}(\sM;T\sM),
$$
 which holds true as $4/r + n/q <2$ (here we need $r>2p$), and  the estimate
$$
( u_1 \otimes (u_2)_\flat  | \nabla v)_{\sM}  \leq \|u_1\|_{L_{2q}(\sM)} \|u_2\|_{L_{2q}(\sM)}  \|v\|_{H^1_{q'}(\sM)}
$$
that $\Fw \in C^1(B^{1-2/r}_{qr,\sigma}(\sM;T\sM), \Xw_0)$.
We can now deduce from \cite[Theorem~5.3.1]{PruSim16} that each solution with initial value $u_0$ satisfying
$
\|u_0-u_*\|_{B^{n/q-1}_{qp}} \leq \delta,
$
with $\delta>0$ sufficiently small, exists globally and converges exponentially fast to some $z\in\cE_0$ in the topology of $B^{1-2/r}_{qr}(\sM;T\sM)$.
The embedding
$$B^{1-2/r}_{qr}(\sM;T\sM)\hookrightarrow B^{1-2/p}_{qp}(\sM;T\sM),$$
then yields the assertion in (i).

We note that by the embedding $B^{1-2/p}_{qp}(\sM;T\sM)\hookrightarrow B^{n/q-1}_{qp}(\sM;T\sM)$,
solutions also converge in the topology of  critical spaces.

\medskip
\noindent

(ii) The arguments in step (i) show that $u(t_0)\in B^{1-2/r}_{qr}(\sM; T\sM)$ and that
\eqref{small-1} holds true for any $r>p$ and any fixed  time $t_0\in (0, t_{\max})$.
In the following, we assume $r>\max\{p,2\}.$
We then have the embedding
$$B^{1-2/r}_{qr,\sigma}(\sM; T\sM)\hookrightarrow B^{2\mu-2/r}_{qr,\sigma}(\sM; T\sM)$$
for any fixed $\mu\in (1/r,1/2]$.
We can now consider problem \eqref{NS sys abstract strong} with initial value
$$u(t_0)\in (X_0, X_1)_{\mu-1/r,r} = B^{2\mu -2/r}_{qr,\sigma}(\sM; T\sM).$$
By regularization and uniqueness, we have
$$u(t_0+t_1)\in (X_0, X_1)_{1-1/r.r}\hookrightarrow B^{2-2/r}_{qr}(\sM; T\sM),$$
 for any fixed time $t_1>0$ such that $t_0+t_1<t_{\max}$.

An analogous argument to \eqref{2nd control}, with $\Xw_j$ replaced by $X_j$, $j=1,2$, and
$\mu_r$ replaced by $\mu$, shows that
$\| u(t_0+t_1)-u_*\|_{B^{2-2/r}_{qr}}$ can be made as small as we wish by choosing
$\| u_0 -u_*\|_{B^{n/q-1}_{qp}}$ small.

The condition $r>2$ and $q\geq n$  ensures
$$
(X_0, X_1)_{1-1/r,r} \hookrightarrow B^{2-2/r}_{qr}(\sM;T\sM)\hookrightarrow H^1_q(\sM; T\sM)\cap L_\infty(\sM; T\sM).
$$
Since
$
\|F(u)\|_{L_q(\sM)} \leq C \|u\|_{L_\infty(\sM)} \|u\|_{H^1_q(\sM)}
$
we conclude that  $F\in C^1 ((X_0, X_1)_{1-1/r,r}, X_0).$

\medskip
Theorem~5.3.1 in \cite{PruSim16} then implies
that each solution with initial value close to $u_*$ in $(X_0, X_1)_{1-1/r,r}$
exists globally and converges exponentially fast to some $z\in \cE_0$ in the topology of $B^{2-2/r}_{qr}(\sM; T\sM)$.
By the previous steps and the embedding $B^{2-2/r}_{qr}(\sM; T\sM)\hookrightarrow B^{2-2/p}_{qp}(\sM; T\sM)$,
which hold for ant $r>p$, we obtain assertion (ii).
\end{proof}
\begin{remark}
 In case $q\ge n$, analogous arguments as in \cite[Remarks 4.10] {SiWi22} show that
every global solution of \eqref{NS sys abstract}, respectively \eqref{NS sys abstract strong}, with initial value $u_0$ converges
exponentially fast to an equilibrium, namely to $\cP_{\cE_0}u_0$, where $\cP_{\cE_0}$ is the projection onto the finite dimensional space $\cE_0$.
Hence $z= \cP_{\cE_0}u_0$ in Theorem~\ref{Thm: stability near killing 1} in the particular case  $q\ge n$.
\end{remark}
\begin{corollary}
\label{cor:stability}
Suppose that $\cE_\alpha=\{0\}$,   $p\in (1,\infty) $ and $q\in (n/2,\infty)$ such that $\frac{2}{p}+ \frac{n}{q}\leq 2$. \\
Then there exists some $\delta>0$ such that the assertions (i) and (ii) of Theorem~\ref{Thm: stability near killing 1} hold true
with $z=0$
for any initial value $u_0\in B^{n/q-1}_{qp,\sigma}(\sM;T\sM)$ satisfying
$\|u_0\|_{B^{n/q-1}_{qp}} \leq \delta.$
\end{corollary}



\goodbreak

\appendix


\section{Tensor bundles and the Levi-Civita connection}\label{Appendix A}

Let $\sM$ be a compact, smooth,  and oriented $n$-dimensional Riemannian manifold with boundary $\Sigma=\partial\sM$
and let $(\cdot | \cdot)_g$ denote the Riemann metric on $\sM$.
We will use the same notation for the (induced) Riemann metric on $\Sigma$.

\medskip
Then $T{\sM}$ and $T^{\ast}{\sM}$ denote the tangent and the cotangent bundle of ${\sM}$, respectively,
and $T^{\sigma}_{\tau}{\sM}:=T{\sM}^{\otimes{\sigma}}\otimes{T^{\ast}{\sM}^{\otimes{\tau}}}$ stands for the $(\sigma,\tau)$-tensor bundle
of $\sM$ for $\sigma,\tau\in \bN$.
The notations $\Gamma(\sM ;T^{\sigma}_{\tau}{\sM})$ and $\mathcal{T}^{\sigma}_{\tau}{\sM}$ stand  for the set of all sections of   $T^{\sigma}_{\tau}{\sM}$ and the $C^{\infty}({\sM})$-module of all smooth sections of $T^{\sigma}_{\tau}\sM$, respectively.
For abbreviation, we put $\mathbb{J}^{\sigma}:=\{1,2,\ldots,n\}^{\sigma}$, and $\mathbb{J}^{\tau}$ is defined alike.

Given local coordinates $ \{x^1,\ldots,x^n\}$,
$$(i):=(i_1,\cdots,i_{\sigma})\in\bJ^{\sigma},\quad (j):=(j_1,\cdots,j_{\tau})\in\bJ^{\tau},$$
we set
\begin{align*}
\frac{\partial}{\partial{x}^{(i)}}:=\frac{\partial}{\partial{x^{i_1}}}\otimes\cdots\otimes\frac{\partial}{\partial{x^{i_{\sigma}}}}, \hspace*{.5em} dx^{(j)}:=dx^{j_1}\otimes{\cdots}\otimes{dx}^{j_{\tau}}.
\end{align*}
Suppose that $a\in \Gamma(\sM; T^\sigma_\tau \sM) $  is a $\bK$-valued, $\bK\in \{\bR,\bC\}$,  tensor bundle on $\sM$.
In this appendix, for notational brevity, we denote both $T^\sigma_\tau \sM$ and its complexification by $T^\sigma_\tau \sM$.
The local representation of $a$ with respect to these coordinates is given by
\begin{align*}
a=a^{(i)}_{(j)} \frac{\partial}{\partial{x}^{(i)}} \otimes dx^{(j)}, \hspace{1em}\text{ with } a^{(i)}_{(j)}: U_k \to \bK,
\end{align*}
where  $U_k\subset  \sM$  is a coordinate patch.

For $s\in \{1,\ldots, \sigma \}$,  $t\in \{1,\ldots, \tau\}$ and $a\in \Gamma(\sM; T^\sigma_\tau \sM)$,
${\sf C}^s_t (a)\in \Gamma(\sM; T^{\sigma-1}_{\tau-1} \sM)$ denotes the contraction of  $a$ with respect to the $(s,t)$-position.
This means that  in a local representation of $a$,
$$
a=a^{(i_1,\dots, i_s,\dots, i_\sigma)}_{(j_1,\dots, j_t,\dots, j_\tau)}
\frac{\partial}{\partial x^{i_1}}\otimes\cdots \otimes \frac{\partial}{\partial x^{i_s}}\otimes \cdots \otimes \frac{\partial}{\partial x^{i_\sigma}}
\otimes dx^{j_1}\otimes \cdots \otimes dx^{j_t}\otimes \cdots \otimes dx^{j_\tau},
$$
the terms $\frac{\partial}{\partial x^{i_s}}$ and $dx^{j_t}$ are deleted and
$
a^{(i_1,\dots, i_s,\dots, i_\sigma)}_{(j_1,\dots, j_t,\dots, j_\tau)}$ is replaced by $ a^{(i_1,\dots, k,\dots, i_\sigma)}_{(j_1,\dots, k,\dots, j_\tau)}$,
and the sum convention is used for $k$.

\medskip
Any
$S\in  \Gamma(\sM;T^1_1\sM)$ induces a linear map from $\Gamma(\sM;T\sM)$
to $\Gamma(\sM;T\sM)$ by virtue of
$$
S u=(S^i_j \frac{\partial}{\partial x^i}\otimes dx^j )u = S^i_j u^j  \frac{\partial}{\partial x^i},
\quad u=u^j \frac{\partial}{\partial x^j}\in\Gamma(\sM;T\sM).
$$
The dual $S^*$ of $S\in  \Gamma(\sM;T^1_1\sM)$  is a linear map from  $\Gamma(\sM;T^*\sM)$ to $\Gamma(\sM;T^*\sM)$,
defined by
$$
S^* \alpha=(S^i_j dx^j \otimes \frac{\partial}{\partial x^i})\alpha = S^i_j  \alpha_i dx^j,
 \quad \alpha  = \alpha_i dx^i\in \Gamma(\sM;T^*\sM).
$$
The adjoint $S^\sT$ of $S\in  \Gamma(\sM;T^1_1\sM)$
is the linear map from $\Gamma(\sM;T\sM)$ to $\Gamma(\sM;T\sM)$ defined by $S^{\sT} =\gs S^* \gf$, or more precisely,
\begin{equation}
\label{ST}
S^\sT u = \gs[ S^* (\gf u)], \quad u \in \Gamma (\sM; T\sM).
\end{equation}
It holds that $(Su|v)_g=(u|S^{\sf T}v)_g$ for  tangent fields $u,v$.
In local coordinates,
$
S^\sT = g^{i\ell} S^m_\ell g_{jm} \frac{\partial}{\partial x^i} \otimes dx^j.
$

\medskip\noindent
For $a\in\Gamma(\sM; T^\sigma_{\tau}\sM)$, $\tau\ge 1$, $a^\sharp \in\Gamma(\sM; T^{\sigma+1}_{\tau-1}\sM)$ is defined by
\begin{equation*}
a^\sharp := g^\sharp a:={\sf C}^{\sigma+2}_1(a \otimes   g^{*} ),
\end{equation*}
and for $a\in\Gamma(\sM; T^\sigma_{\tau}\sM)$, $\sigma\ge 1$,  $a_\flat \in\Gamma(\sM; T^{\sigma-1}_{\tau+1}\sM)$
is defined by
\begin{equation*}
a_\flat := g_\flat a:={\sf C}^{\sigma }_1( g\otimes  a  ).
\end{equation*}

Let $\nabla$ be the Levi-Civita connection on $\sM$.
For $u\in C^1(\sM; T\sM)$, the covariant derivative $\nabla u\in C(\sM; T^1_1\sM)$ is given in local coordinates by
$$
\nabla u= \nabla_j u\otimes dx^j= (\partial_j u^i +\Gamma^i_{jk}u^k)\frac{\partial}{\partial x^i}\otimes dx^j =:u^i_{|j}\frac{\partial}{\partial x^i}\otimes dx^j,
$$
where $u=u^i\frac{\partial}{\partial x^i}$,
$\nabla_j = \nabla_{\frac{\partial}{\partial x^j}}$,
 and $\Gamma^i_{jk}$ are the Christoffel symbols.
It follows that  $\nabla u + [\nabla u]^{\sf T} $ is given in local coordinates by
\begin{equation*}
\nabla u + [\nabla u]^{\sf T}= \big( u^i_{|j} + g^{i\ell} u^m_{|\ell} g_{jm}\big) \frac{\partial}{\partial x^i}\otimes dx^j
\end{equation*}
and
\begin{equation*}
(\nabla u + [\nabla u]^{\sf T})^\sharp =\left(g^{jk} u^i_{|k} + g^{ik} u^j_{|k}\right)\frac{\partial}{\partial x^i}\otimes \frac{\partial }{\partial x^j}.
\end{equation*}

The extension of the Levi-Civita connection on $C^1(\sM; T^{\sigma}_{\tau}{\sM})$  is again denoted by $\nabla:=\nabla_g$.
For $a\in C^1(\sM; T^{\sigma}_\tau \sM)$, $\nabla a \in C(\sM; T^{\sigma}_{\tau+1}\sM)$ is given in local coordinates by
$\nabla a = \nabla_j a \otimes dx^j$,  and
$$
\dv : C^1(\sM; T^{\sigma }_\tau \sM) \to C(\sM; T^{\sigma -1}_\tau \sM), \quad \sigma\ge 1,
$$
is the divergence operator, defined by
$
\dv a = {\sf C}^{\sigma}_{\tau +1}(\nabla a).
$
In particular,
$$
\dv u= u^i_{|i} \quad \text{for}\quad u=u^i \frac{\partial}{\partial x^i},\qquad
\dv S = S^{ik}_{|k} \frac{\partial} {\partial x^i}\quad\text{for}\quad S=S^{ij} \frac{\partial} {\partial x^i} \otimes  \frac{\partial} {\partial x^j}.
$$
For a scalar function $\phi \in C^1(\sM; \bK)$, the gradient vector $\gd\phi\in C(\sM; T\sM)$ is defined by the relation
$$
(\gd \phi | u)_g := \la \nabla\phi , u\ra_g= \nabla_u \phi , \quad u\in C(\sM; T\sM),
$$
where $\nabla\phi\in C(\sM; T^\ast \sM)$ is the covariant derivative of $\phi$. In local coordinates, we have
$$
(\gd\phi)^i = g^{ij} \partial_j \phi ,\quad 1\le i \le n.
$$

For the curvature tensor $R(u,v)w:= [\nabla_u, \nabla_v]w - \nabla_{[u,v]}w$, with $u,v,w\in \Gamma(\sM; T\sM)$,
we use the convention (as in \cite{Petersenbook, SaTu20}, for instance)
$$
R\left(\frac{\partial}{\partial x^i}, \frac{\partial}{\partial x^j}\right) \frac{\partial}{\partial x^k} = R^\ell_{ijk } \frac{\partial}{\partial x^\ell}.
$$
The Ricci tensor $\Ric\in T^0_2\sM$ is then defined by
$\Ric_{jk}=R^i_{ijk}.$

\medskip
The generalized metric $g^{\tau}_{\sigma}$ on $T^{\sigma}_{\tau}{\sM}$ is still written as $(\cdot|\cdot)_{g}$.
In addition,
\begin{align*}
|\cdot|_g:C^\infty(\sM; T^{\sigma}_{\tau}{\sM})\rightarrow{C^{\infty}}({\sM}),\hspace*{.5em} a\mapsto\sqrt{(a|a)_g}
\end{align*}
is called the (vector bundle) \emph{norm} induced by $g$.

\section{Some analysis on manifolds}\label{Appendix B}

\begin{lemma}
\label{Appendix Lem: divergence thm}  Let $1<q<\infty$.
\phantom{new line}
\begin{enumerate}
\item[]
\vspace{-4mm}
\item[{\rm (a)}]
Suppose that $u\in H^1_q(\sM; T \sM)$ and $\phi\in H^1_{q'}(\sM)$. Then
\begin{equation}
\label{div-scalar}
\begin{aligned}
 \int_{\sM} (\dv u)\phi \,d\mu_g &= -\int_{\sM} (u | \gd\phi)_g \, d\mu_g + \int_\Sigma ( u | \nu_\Sigma  )_g \phi\,  d\sigma_g   \\
& = - \int_{\sM} \nabla_u \phi\,d\mu_g  +  \int_\Sigma ( u | \nu_\Sigma  )_g \phi\,  d\sigma_g,
\end{aligned}
\end{equation}
where  $\mu_g$ ($\sigma_g$, respectively) is the volume element induced by $g$ or $g|_{\Sigma}$, respectively.
\vspace{2mm}
\item[{\rm (b)}] (Green's first identity).
Suppose  that  $S\in H^1_q(\sM; T^2_0\sM)$    and $v\in H^1_{q'}(\sM; T \sM)$.
Then
\begin{equation}
\label{Green-first}
( \dv S | v )_\sM = - ( S_\flat | \nabla v )_\sM + (  S_\flat \nu_\Sigma | v )_\Sigma.
\end{equation}

In particular,
\begin{enumerate}
\vspace{2mm}
\item[{\rm (i)}]  $((\Delta_\sM + \Ric^\sharp)u | v)_\sM = - 2(D_u | D_v)_\sM + 2(D_u \nu_\Sigma | v)_\Sigma$, with
$$ (u,v)\in H^2_{q,\sigma} (\sM; T\sM) \times  H^1_{q'}(\sM; T\sM) ;$$
\item[{\rm (ii)}]  $(\Delta_\sM u | v)_{\sM} = - (\nabla u | \nabla v)_{\sM} + (\nabla u\, \nu_\Sigma | v)_\Sigma$, with
$$ (u,v)\in H^2_q(\sM; T\sM) \times H^1_{q'}(\sM; T\sM) ;$$
\item[{\rm (iii)}]  $(\dv (u\otimes u) | v)_{\sM} =  - (u\otimes u_\flat | \nabla v)_{\sM} + (u |v)_\Sigma \,(u | \nu_\Sigma)_\Sigma$, with
$$(u,v)\in \left( H^1_q(\sM; T\sM)\cap L_\infty (\sM; T\sM) \right) \times  H^1_{q'}(\sM; T\sM) ,$$
where $D_u = \frac{1}{2} (\nabla u + [\nabla u]^{\sf T})$ and $D_v = \frac{1}{2} (\nabla v + [\nabla v]^{\sf T})$.
\end{enumerate}

\end{enumerate}
\end{lemma}
\begin{proof}
(a)
We first consider the case $u\in C^1(\sM; T \sM)$ and $\phi\in C^1(\sM)$.
The assertion follows from
$$
\dv (u\phi) = (\dv u)\phi + (u | \gd\phi)_g=(\dv u)\phi + \nabla_u \phi
$$
and  the divergence theorem on manifolds with boundary,
cf. \cite[Theorem~16.32]{LeeBook}.
In view of the fact that $\dv\in \cL(H^1_q(\sM;T\sM), L_q(\sM;T\sM))$, the assertion follows by a density argument.

\medskip\noindent
(b) As in Part (a), it suffices to prove the assertion for   $S\in C^1(\sM; T^2_0\sM)$
and $v\in C^1(\sM; T \sM)$.
Then we have in local coordinates
$$
S=S^{ij} \frac{\partial}{\partial x^i} \otimes \frac{\partial}{\partial x^j},\quad  v=v^i \frac{\partial}{\partial x^i}.
$$
One readily verifies that
$$
S_\flat^{\sf T} = g_{jk} S^{ji} \frac{\partial }{\partial x^i}\otimes dx^k,\quad
S_\flat^{\sf T} v= g_{jk} S^{ji} v^k\frac{\partial }{\partial x^i}.
$$
Direct computations show that in local coordinates
\begin{align*}
\dv(S_\flat^{\sf T}v) =  (g_{jk} S^{ji} v^k)_{|i} = (g_{jk} S^{ji})_{ | i} v^k  +  g_{jk} S^{ji} v^k_{|i}
= (\dv S | v)_g + (S_\flat | \nabla v)_g.
\end{align*}
By the divergence theorem on manifolds with boundary, cf. \cite[Theorem~16.32]{LeeBook},
\begin{align*}
\int_\sM \dv(S_\flat^{\sf T}v )\, d\mu_g= ( S_\flat^{\sf T}v | \nu_\Sigma )_\Sigma 	= (  S_\flat \nu_\Sigma | v )_\Sigma .
\end{align*}
Hence \eqref{Green-first} holds.

\medskip\noindent
The assertion in (i) then follows by choosing $S= 2D(u)$ and noting that $\dv 2D(u) = \Delta_\sM u +\Ric^\sharp u$   as $\dv u=0$,   see \eqref{div-D-local},
and
$$
(S_\flat | \nabla v)_g = (\nabla u + [\nabla u]^{\sf T} | \nabla v)_g= 2 (D_u | D_v)_g.
$$
 (ii) follows by choosing $S= (\nabla u)^\sharp$ and noting that $S_\flat =\nabla u$ and $\dv (\nabla u)^\sharp = \Delta_\sM u$.
Finally, the assertion in (iii) follows immediately by choosing $S= u\otimes u$.
\end{proof}

\begin{lemma} Suppose $\phi \in H^3_q(\sM)$. Then
\label{lem: commuator-Ricci}
$$
\Delta_\sM\, \gd \phi=\gd \Delta_B \phi +\Ric^\sharp \gd \phi.
$$
\end{lemma}
\begin{proof}
Let $u=u^k \frac{\partial }{\partial x^k}$. Then
$\Delta_\sM u = g^{ij} u^k_{| i\, | j}  \frac{\partial }{\partial x^k}$ and hence, $(\Delta_\sM u)^k = g^{ij} u^k_{| i\, | j}$.

\smallskip\noindent
In case $u=\gd \phi =g^{kl}\phi_{| l}  \frac{\partial }{\partial x^k}$ we obtain,
employing the property  that $g^{kl}_{|i}=0$ for all $1\le i,k,l\le n$,
\begin{equation*}
\begin{aligned}
(\Delta_\sM \,\gd\phi )^k &=g^{ij} (g^{kl } \phi_{ | l})_{| i\, | j} = g^{ij} g^{kl }(\phi_{ | l})_{| i\, | j}
   = g^{ij} g^{kl } (\phi_{|l \, |i})_{| j} =  g^{ij} g^{kl } (\phi_{|i \, |l})_{| j}  \\
 & =  g^{ij} g^{kl } (\phi_{|i})_{|l \, | j} = g^{kl } (g^{ij} (\phi_{|i})_{|l \, | j}
 =  g^{kl } u^j_{|l \, | j}
  = g^{kl } u^j_{|j\, | l } + g^{kl } ( u^j_{|l \, | j} -u^j_{|j\, | l }  ) \\
 & = (\gd \dv u)^k  +   g^{kl }\Ric_{ lm}u^m   = ( \gd \dv u + \Ric^\sharp u )^k \\
 & = (\gd \Delta_B \phi + \Ric^\sharp\,\gd \phi )^k,
 \end{aligned}
\end{equation*}
where we used  the fact that
$\phi_{| i\,| j} = \partial_j\partial_i \phi - \Gamma^k_{ji }\partial_k \phi = \phi_{| j\,| i} $ for scalar functions.
\end{proof}

\begin{lemma}[Korn's inequality]
\label{Appendix Lem: korn}
There exists some constant $C>0$ such that
\begin{equation}
\label{korn ineq}
\|u\|_{H^1_2(\sM)} \leq C \| D_u\|_{L_2(\sM)} ,\quad u\in  V^1_2,
\end{equation}
 where $V^1_2$ is defined in \eqref{E-orth}.
In particular, if
\begin{itemize}
\item[{\em (i)}]   $\alpha>0$, or
\item[{\em (ii)}] $\Ric^\sharp < 0$, $L_\Sigma \geq 0$ and $\alpha=0$, or
\item[{\em (iii)}] $\Ric^\sharp \leq  0$, $L_\Sigma> 0$ and $\alpha=0$,
\end{itemize}
then \eqref{korn ineq} holds for all $u\in H^1_{2,\sigma}(\sM;T\sM)$.
\end{lemma}
\begin{proof}
By combining  assertions (i) and (ii) of Lemma~\ref {Appendix Lem: divergence thm}(b),
employing~\eqref{boundary condition 1},  and using compactness of $\sM$, we obtain
 \begin{equation}
 \label{Korn-estimate}
 \begin{aligned}
2\|D_u \|_{ L_2(\sM) }^2
&= -(\Delta_M u| u)_\sM - ( \Ric^\sharp u | u)_{\sM}  +  ((\nabla u + [\nabla u]^\sT) \nu_\Sigma  | u)_{ \Sigma }\\
& =  \| \nabla u \|_{ L_2(\sM) }^2 -( \Ric^\sharp u | u)_{\sM} +  ((\nabla u + [\nabla u]^\sT) \nu_\Sigma  | u)_{ \Sigma }- (\nabla u \nu_\Sigma | u)_\Sigma\\
& =  \| \nabla u \|_{ L_2(\sM) }^2 -( \Ric^\sharp u | u)_{\sM} +  ( \cP_\Sigma ([\nabla u]^\sT \nu_\Sigma ) | u)_{ \Sigma }\\
& =  \| \nabla u \|_{ L_2(\sM) }^2 -( \Ric^\sharp u | u)_{\sM} + (L_\Sigma u | u)_{ \Sigma } \\
&\ge  \| \nabla u \|_{ L_2(\sM) }^2 - c_1 (\| u \|_{ L_2(\sM) } + \| u \|_{L_2(\Sigma)})
\end{aligned}
\end{equation}
for some constant $c_1$.
Hence,
\begin{equation}
\label{korn ineq 2}
\|u\|_{H^1_2(\sM)} \leq C \left( \| D_u\|_{ L_2(\sM) }  +  \| u\|_{L_2(\sM)}  +  \| u\|_{  L_2(\Sigma) }    \right)
\end{equation}
for some constant $C$.
By trace theory, interpolation theory, see for instance \cite[Theorem 10.1]{Ama13}, and Young's inequality,
we conclude that for every $\varepsilon >0$ there exists a constant $C(\varepsilon)>0$ such that
\begin{equation*}
\| u \|_{L_2(\Sigma)} \le \varepsilon \|u\|_{H^1_2(\sM)} + C(\varepsilon) \| u \|_{L_2(\sM)}.
\end{equation*}
Inequality \eqref{korn ineq 2} then becomes
\begin{equation}
\label{korn ineq 3}
\|u\|_{H^1_2(\sM)} \leq C \left( \|  D_u\|_{L_2(\sM)}  +  \| u\|_{L_2(\sM)}   \right)
\end{equation}
with a (possibly different) constant $C$.

The assertion in \eqref{korn ineq} then follows by a contradiction argument.
Suppose \eqref{korn ineq} does not hold.
Then there exists a sequence
$\{u_n\}_{n=1}^\infty  \subset V^1_2$ such that $\|u_n\|_{H^1_2(\sM)}=1$ and
$$
 \| D_{u_n}\|_{  L_2(\sM) }   \to 0, \quad\text{as  $n\to \infty$}.
$$
Since $V^1_2$ is a closed subspace of $H^1_2(\sM;T\sM)$,
there exist a subsequence of $\{u_n\}_{n=1}^\infty $, not relabelled, and some $u\in V^1_2$ such that $u_n\to u$ in $L_{2,\sigma}(\sM;T\sM)$ and $u_n \rightharpoonup u$ in $H^1_{2,\sigma}(\sM;T\sM)$.
It follows from \eqref{korn ineq 3} that $\{u_n\}_{n=1}^\infty $ is Cauchy in $V^1_2$ and thus, $u_n \to u$ in $V^1_2$.
We can now infer that  $\| D_{u_n}- D_u\|_{ L_2(\sM)  }  \to 0$  as $n\to \infty$, and consequently,  $u\in \cE_\alpha$.
Therefore, $u\in \cE_\alpha \cap V^1_2=\{0\}$.  However, this contradicts the assumption that $\|u \|_{H^1_2(\sM)}=1$.
This completes the proof for \eqref{korn ineq}.

\medskip\noindent
Let us consider the set $\cE_\alpha$ under conditions (i)-(iii). When $\alpha>0$, it follows from Proposition~\ref{Prop: characterizing Ealpha} that  $\cE_\alpha=\{0\}$.
Now we consider the case $\alpha=0$.
Let $u\in \cE_0$ be given.  Then $D_u=0$.
By the computations in \eqref{Korn-estimate}, we have
\begin{align*}
0= 2\|D_u \|_{ L_2(\sM) }^2
=  \| \nabla u \|_{ L_2(\sM) }^2 -( \Ric^\sharp u | u)_{\sM} + (L_\Sigma u | u)_{ \Sigma }.
\end{align*}

This shows that under assumptions (ii) or (iii), $u=0$, and hence $\cE_0=\{0\}$.
Therefore, in all three cases, we have $V^1_2=H^1_{2,\sigma}(\sM; T\sM)$.
\end{proof}
\begin{remark}
\label{Appendix Rem: korn-no-divergence}
\mbox{}\\
 (a)
In the Euclidean case, Korn's inequality for Navier boundary conditions was first
proved in \cite[Lemma 4]{SoSc73}.

\smallskip\noindent
(b)
The estimate  \eqref{korn ineq 2} remains valid for all $u\in H^1_2(\sM; T\sM)$ satisfying $(u | \nu_\Sigma)_g=0$,
that is, without assuming that $\dv u=0$.
Indeed, in this case, the assertion of  Lemma~\ref{Appendix Lem: divergence thm}(b)(i)   reads
$$
(\Delta_\sM u +  \Ric^\sharp u + \gd\dv u | u)_\sM = - 2(D_u | D_u)_\sM + 2(D_u \nu_\Sigma | u)_\Sigma, \quad u\in H^2_2(\sM; T\sM).
$$
Using the relation
$\dv ((\dv u)u) = (\dv u)^2 + (\gd\dv u | u)_g $ and the assumption $(u|\nu_\Sigma)_g=0$, we obtain by analogous arguments as above
\begin{equation*}
\begin{aligned}
2\|D_u \|_{ L_2(\sM) }^2
& =  \| \nabla u \|_{ L_2(\sM) }^2 -( \Ric^\sharp u | u)_{\sM} + (L_\Sigma u | u)_{ \Sigma } + \| \dv u\|^2_{L_2(\sM)}\\
&\ge  \| \nabla u \|_{ L_2(\sM) }^2 - c_1 (\| u \|_{ L_2(\sM) } + \| u \|_{L_2(\Sigma)}).
\end{aligned}
\end{equation*}
Hence the assertion \eqref{korn ineq 2} follows by a density argument.
\end{remark}
In order to construct the Helmholtz projection on $(\sM,g)$, we will need the following   lemma,
where we use the definition
$$
 H^{-1}_q(\sM) := \left( H^1_{q'}(\sM) \right)'  \quad \text{and} \quad W^{-1/q}_q(\Sigma):= \left(W^{ 1/q}_{q'}(\Sigma) \right)' ,
\quad 1/q+ 1/q'=1.
$$
 We note  that our definition of  $H^{-1}_q(\sM)$ differs from the usual definition used in the literature.
This abuse of notation allows for a more streamlined presentation of the results in the following two Lemmas.

\begin{lemma}
\label{Appendix Lemma Poisson}
Let $q\in (1,\infty)$ and $k\in \{-1,0,1\}$.
Then the  Poisson problem
\begin{equation}
\label{Poisson}
\left\{\begin{aligned}
\Delta_B \phi   &= f &&\text{on}&&\sM ,\\
(\gd \phi   | \nu_\Sigma)_g &=h &&\text{on}&&\Sigma
\end{aligned}\right.
\end{equation}
has a unique (up to a constant)  solution $\phi \in H^{k+2}_q(\sM)$ for each
$f\in H^k_q(\sM)$
and $h\in W^{k+1-1/q}_q(\Sigma)$ satisfying
the solvability condition
\begin{equation}
\label{solvability cond}
 \la f | 1 \ra_\sM =\la h | 1 \ra_\Sigma.
\end{equation}
Furthermore,
\begin{equation}
\label{gradient estimate}
\|\gd \phi \|_{H^{k+1}_q(\sM)}  \leq C \left( \| f\|_{H^k_q(\sM)} + \|h\|_{W^{k+1-1/q}_q(\Sigma)}   \right)
\end{equation}
for some constant $C>0$.
In case $k=-1$, equation \eqref{Poisson} is interpreted as
$$
 (\gd \phi | \gd v)_\sM =  \la \cM h  -f | v\ra_\sM,   \quad v \in H^1_{q'}(\sM),
$$
where $\cM\in \cL( W^{-1/q}_q(\Sigma), H^{-1}_q(\sM))$ is the dual of the trace operator ${\rm tr}_\Sigma \in  \cL(H^1_{q'}(\sM), W^{1-1/q'}_{q'} (\Sigma))$.
\end{lemma}
\goodbreak
\begin{proof}
For $k\in \{0,1\}$ and $u\in H^{k+2}_q(\sM)$, let $\cB u:=   ({\rm tr}_\Sigma\, \gd u | \nu_\Sigma)_g$, where ${\rm tr}_\Sigma $ denotes the trace operator.  Then
$ \cB\in \cL( H^{k+2}_q(\sM), W^{k+1-1/q}_q(\Sigma))$.
 Moreover, let
\begin{equation*}
\begin{aligned}
\cA_k: \sD(\cA_k)\to H^k_q(\sM), \  \sD(\cA_k)=\{ u\in H^{k+2}_q(\sM):\, \cB u =0 \text{ on } \Sigma \}, \   \cA_k u := -\Delta_B u.
\end{aligned}
\end{equation*}
Following a localization argument as in Section~\ref{Section:NS with perfect slip boundary condition strong}, one can show that there exists $\omega_0\in \bR$  such that for all $\omega>\omega_0$
$$
\omega+ \cA_k \in \Lis( \sD(\cA_k), H^k_q(\sM)).
$$
Since the embedding $H^{k+2}_q(\sM) \hookrightarrow H^k_q(\sM)$ is compact, the spectrum $\sigma(\cA_k)$  consists solely of isolated eigenvalues with finite multiplicity and the spectrum does not depend on $q\in (1,\infty)$. Let $\lambda\in \sigma(\cA_s)$ and consider the eigenvalue problem
$$
\lambda u= \cA_k u \quad \text{in }\sM.
$$
Multiplying the above equality by $\overline{u}$ and applying Lemma~\ref{Appendix Lem: divergence thm} yields
$$
\lambda \|u \|_{L_2(\sM)}= \| \gd u \|_{L^2(\sM)},
$$
which implies $\sigma(\cA_k) \subset [0,\infty)$.
In particular,  we have
$$
\sN(\cA_k)=\{ u\in \sD(\cA_k): \, u\equiv \text{constant}\} = \bR 1_\sM,
$$
where $1_\sM$ is the constant 1 function on $\sM$.
Next, we will show that $\lambda=0$ is in fact a semi-simple eigenvalue of $\cA_k$.
Assume that $u\in \sN(\cA_k^2)$ and let
$$
\cA_k u= \phi.
$$
Since $\phi\in \sN(\cA_k)$, it follows that $\phi\equiv$ constant.
Multiplying both sides of the equation above by $\phi$ and   using  Lemma~\ref{Appendix Lem: divergence thm} results in
$$
(\cA_k u | \phi)_\sM= (\gd u | \gd \phi)_{\sM}=0=    \|\phi\|_{L_2(\sM)}^2  ,
$$
which further yields $\phi=0$. Therefore, $\sN(\cA_k^2) = \sN(\cA_k)$. The assertion is thus established.
This further implies that
\begin{equation*}
H^k_q(\sM)= \sN(\cA_k) \oplus  \sR(\cA_k) = \bR 1_\sM \oplus  \sR(\cA_k).
\end{equation*}
Put $Y_0=L_q(\sM)$ and $Y_1=\sD(A_0)$,
where  $A_0:=\omega+\cA_0$ for a fixed number $\omega>0$.
We note that it follows from $\sigma(\cA_0) \subset [0,\infty)$ that
$ \omega+\cA_0\in \Lis( \sD(\cA_0), L_q(\sM)) $ for any $\omega>0$.

\medskip
The pair
$(Y_0,A_0)$ generates an interpolation-extrapolation scale with respect to the complex interpolation functor.
We recall that $Y_1:=D(A_0)=\{u\in H^2_q(\sM): \cB u=0\}.$
Let $Y^\sharp_0 = L_{q'}(\sM)$ and
\begin{equation*}
A^\sharp_0 := (A_0)'=\omega + \cA_0, \quad Y^\sharp_1:=D(A^\sharp_0)=\{u\in H^2_{q'}(\sM): \cB u=0\}.
\end{equation*}
Then $(Y^\sharp_0, A^\sharp_0)$ also generates an interpolation-extrapolation scale $(Y^\sharp_\beta, A^\sharp_\beta)$,
$\beta \in \bR$, the dual scale.

By \cite[Theorem  V.1.5.12]{Ama95}, it holds that
$(Y_\beta)'=Y_{-\beta}^\sharp$  and $(A_\beta)'=A_{-\beta}^\sharp$
for $\beta\in \bR$. In particular,  when $\beta=-1/2$,
\begin{equation*}
\begin{aligned}
& \sD(A_{-1/2})=Y_{1/2}=[Y_0,Y_1]_{ 1/2}=H^1_{q}(\sM;T\sM), \\
& Y_{-1/2}=(Y_{1/2}^\sharp)' = ([Y_0^\sharp,Y_1^\sharp]_{1/2})'=(H^1_{q'}(\sM;T\sM))'= H^{-1}_q(\sM),
\end{aligned}
\end{equation*}
see  Proposition~\ref{Appendix Prop interpolation 4}.
We have
$$A_{-1/2}=\omega+ \cA_{-1/2}:  H^1_q(\sM)= Y_{1/2} \to  Y_{-1/2}=H^{-1}_q(\sM),$$
where $\cA_{-1/2}$ is characterized by
\begin{align*}
 \la \cA_{-1/2}\,\phi | v\ra_\sM = (\gd \phi | \gd v )_\sM ,\quad v\in H^1_{q'}(\sM),
\end{align*}
see \eqref{div-scalar}, and satisfies
$\sN(\cA_{-1/2})=\bR 1_\sM$. Moreover,
$$
 H^{-1}_q(\sM) =  \bR 1_\sM\oplus  \sR(\cA_{-1/2}).
$$
Particularly, this implies that $  \cA_{-1/2} \in \Lis( H^1_q(\sM) \cap  \sR(\cA_{-1/2}) ,  \sR(\cA_{-1/2}))$.
In addition, observe that
$$
 \sR(\cA_{-1/2}) = \{  u\in  H^{-1}_q(\sM) : \, \la  u | 1_{\sM} \ra _\sM =0  \}.
$$
Since
${\rm tr}_\Sigma \in  \cL(H^1_{q'}(\sM), W^{1-1/q'}_{q'} (\Sigma))$, its dual
\begin{equation}
\label{dual-trace}
\cM :=( {\rm tr}_\Sigma |_{H^1_{q'}(\sM)})' \in \cL( W^{-1/q}_q(\Sigma),   H^{-1}_q(\sM) )
\end{equation}
is well-defined.
An important observation is that
$\phi$ is a weak solution of \eqref{Poisson} in $H^1_q(\sM)$ iff
\begin{equation}
\label{Poisson2}
\cA_{-1/2}\,\phi= \cM h - f  ,
\end{equation}
or equivalently,
$$
 (\gd \phi | \gd v)_\sM =  \la \cM h  -f | v\ra_\sM,   \quad v \in H^1_{q'}(\sM).
$$
Since $\cM h  -f \in H^{-1}_q(\sM) $, it suffices to show that $ \cM h  -f  \in \sR(\cA_{-1/2})$.
Indeed, due to \eqref{solvability cond}
\begin{align*}
\la \cM h - f  | 1\ra_\sM   = \la h | 1 \ra_\Sigma - \la f | 1 \ra_\sM =0.
\end{align*}
This implies that \eqref{Poisson2} has a unique (up to a constant) weak solution $\phi\in H^1_q(\sM)$.
Estimate \eqref{gradient estimate} in the case $k=-1$ follows from
\begin{align*}
 \|\gd \phi\|_{L_q(\sM)}  & \leq  \| \phi \|_{H^1_q(\sM)}  \leq C  \| \cM h - f \|_{ H^{-1}_q(\sM) }
 \leq C \big( \|f \|_{ H^{-1}_q(\sM) } + \| \cM h   \|_{   H^{-1}_q(\sM)}   \big) \\
& \leq C \big( \|f \|_{  H^{-1}_q(\sM) } + \|   h   \|_{W^{-1/q}_q(\Sigma)} \big) .
\end{align*}
When $f\in H^k_q(\sM)$ and $h\in W^{k+1-1/q}_q(\Sigma)$ with $k\in \{0,1\}$,
it follows from \cite[Theorem~10.1]{Ama13} that $\cB$ has a right inverse $ \cN_k \in \cL( W^{k+1-1/q}_q(\Sigma), H^{k+2}_q(\sM))$.
Observe that $\phi$ is a strong solution of \eqref{Poisson}
iff $\psi= \phi - \cN_k   h$ is a strong  solution of
\begin{equation}
\label{Poisson3}
\left\{\begin{aligned}
\cA_k \psi   &= \Delta_B  \cN_k  h -f  &&\text{on}&&\sM ,\\
(\gd \psi   | \nu_\Sigma)_g &=0 &&\text{on}&&\Sigma.
\end{aligned}\right.
\end{equation}
Since $h\in W^{k+1-1/q}_q(\sM)$, it is an easy task to verify that $\Delta_B  \cN_k h \in H^k_q(\sM)$. Condition~\eqref{solvability cond} and Lemma~\ref{Appendix Lem: divergence thm} imply that $\Delta_B \cN_k  h - f  \in \sR(\cA_k)$.

Therefore, \eqref{Poisson3} has a unique (up to a constant) strong solution $\psi\in H^{k+2}_q(\sM)$.
The remaining cases in \eqref{gradient estimate} can be established in a similar way to $k=-1$.
This completes the proof.
\end{proof}


\begin{lemma}
\label{Apendix Lemma Helmholtz}
Let $q\in (1,\infty)$ and $k\in \{-1,0,1\}$.
For every $u\in H^{k+1}_q(\sM;T\sM)$, the elliptic boundary value problem
\begin{equation}
\label{Helmholtz}
\left\{\begin{aligned}
\Delta_B \phi &=  \dv u  &&\text{on}&&\sM ,\\
(\gd \phi | \nu_\Sigma)_g &=(u  | \nu_\Sigma)_g &&\text{on}&&\Sigma
\end{aligned}\right.
\end{equation}
has a unique (up to a constant)  solution  $\phi\in H^{k+2}_q(\sM)$.
The solution satisfies
\begin{equation}
\label{est Helmholtz}
\| \gd \phi \|_{H^{k+1}_q(\sM)} \leq C \| u \|_{H^{k+1}_q(\sM)}
\end{equation}
for some constant $C>0$.
In case $k=-1$, equation \eqref{Helmholtz} is interpreted as
\begin{equation}
\label{Helmholtz-weak}
(\gd \phi | \gd v)_\sM = (u | \gd  v)_\sM ,\quad  v\in H^1_{q'}(\sM),
\end{equation}
while  \eqref{Helmholtz-weak} is always satisfied for solutions of \eqref{Helmholtz} in case $k=0,1$.

\medskip
Therefore, the Helmholtz projection $\PH\in \cL(H^{k+1}_q(\sM;T\sM), H^{k+1}_{q,\sigma}(\sM;T\sM))$
is well-defined.
\end{lemma}

\begin{proof}
Suppose first that $k\in \{0,1\}$ and let  $\cC u :=({\rm tr}_{\Sigma}u | \nu_\Sigma)_g $ for $u\in H^{k+1}_q(\sM;T\sM)$.
Then by Lemma~\ref{Appendix Lem: divergence thm}(a), the pair
$$(f,g)=(\dv u, \cC u )$$
satisfies the solvability condition~\eqref{solvability cond}.
Moreover, we have   $\dv u\in H^k_q(\sM)$ and $(u| \nu_\Sigma)_g \in W^{k+1-1/q}_q(\sM)$.
The latter follows from the trace theorem, cf. \cite[Theorem~10.1]{Ama13}.
Solvability of  \eqref{Helmholtz} in these two cases thus follows from Lemma~\ref{Appendix Lemma Poisson}.

\medskip
Suppose $\phi\in H^{k+2}_q(\sM)$ is a solution of \eqref{Helmholtz}.
Employing Lemma~\ref{Appendix Lem: divergence thm} twice, we obtain
\begin{equation}
\label{Helmoltz-weak-derivation}
\begin{aligned}
(\gd \phi | \gd v)_\sM
&= -(\dv u | v)_\sM + ((\gd \phi | \nu_\Sigma)_g | {\rm tr}_\Sigma v)_\Sigma  \\
&= -(\dv u | v)_\sM + (\cC u | {\tr}_\Sigma v)_\Sigma
= (u | \gd v)_\sM 
\end{aligned}
\end{equation}
for all $v\in H^1_{q'}(\sM)$, showing \eqref{Helmholtz-weak}.

\medskip
We now consider the case $k=-1$. Let $\cM$ be as in \eqref{dual-trace}.
Employing the same computation as in \eqref{Helmoltz-weak-derivation}, we obtain
\begin{equation*}
 ( \dv u | v )_\sM  =  ( \cC u   | {\rm tr}_\Sigma v)_\Sigma - ( u | \gd v )_\sM
= \la \cM ( \cC u ) | v\ra_\sM - (  u | \gd v )_\sM , \\
\end{equation*}
for each $ (u,v)\in H^1_q(\sM;T\sM)\times  H^1_{q'}(\sM)$.
Hence,
\begin{align*}
 \left|\la \cM (\cC u ) - \dv u | v \ra_\sM \right|
= |( u | \gd v )_\sM | \leq  \| u\|_{L_q(\sM)} \| v \|_{H
^1_{q'}(\sM)},
\end{align*}
which further implies
\begin{equation}
\label{boundedness of weak divergence}
[u \mapsto (\cM ( \cC u ) -\dv u ) ] \in \cL(H^1_q(\sM; T\sM) , H^{-1}_q(\sM) ).
\end{equation}
By the density of   $H^1_q(\sM; T\sM)$ in $ L_q(\sM; T\sM)$ and \eqref{boundedness of weak divergence}, the operator $[u \mapsto (\cM ( \cC u ) - \dv u) ]$  has a unique continuous extension in $\cL(L_q(\sM;T\sM),  H^{-1}_q(\sM;T\sM) )$, denoted by $\cF$.
The extension satisfies
$$
 ( u |  \gd v)_\sM=   \la\cF u | v \ra_\sM  ,  \quad (u,v)\in L_q(\sM;T\sM) \times H^1_{q'}(\sM).
$$
By \eqref {Helmoltz-weak-derivation} and a density argument, we have
$$
(\gd \phi | \gd v)_\sM  =  ( u |  \gd v)_\sM=   \la\cF u | v \ra_\sM   ,  \quad (u,v)\in L_q(\sM;T\sM) \times H^1_{q'}(\sM).
$$
Hence, \eqref{Helmholtz} can be interpreted as
 $$
 \cA_{-1/2}\,\phi =\cF u.
$$
By analogous arguments as in the proof of Lemma~\ref{Appendix Lemma Poisson}, this problem has (up to constants) a unique solution, which satisfies~\eqref{est Helmholtz},
as
$\|\cF u\|_{ H^{-1}_q(\sM)  }\le c \|u\|_{L_q(\sM)}.$
\end{proof}

\section{Interpolation spaces}\label{Appendix C}

As in Section~\ref{Section:NS with perfect slip boundary condition weak}, let $A_0=\omega+A_N:  X_1:=\sD(A_N) \to X_0$, for some $\omega>0$, with $X_0=L_{q,\sigma}(\sM; T\sM)$ and
$$
X_1= \{u\in H^2_{q,\sigma}(\sM;T\sM): \alpha u + \cP_\Sigma \left( (\nabla u + [\nabla  u]^{\sT} )  \nu_\Sigma \right)=0 \text{ on } \Sigma\}.
$$
Recall  that $A_0$ is invertible.
By \cite[Theorems V.1.5.1 and V.1.5.4]{Ama95}, the pair $(X_0, A_0)$ generates an interpolation-extrapolation scale $(X_\beta, A_\beta)$, $\beta\in \bR$, with respect to the complex interpolation functor.
When $\beta\in (0,1)$, $A_\beta$ is the $X_\beta$-realization of $A_0$, where
$$
X_\beta =[X_0,X_1]_\beta
$$
in view of \eqref{cHi AN}.
Let $X_0^\sharp:=(X_0)'=L_{q',\sigma}(\sM;T\sM)$ and
\begin{equation*}
\begin{aligned}
&A_0^\sharp:=(A_0)'=(\omega+A_N)' = \omega - \mu_s \PH (\Delta_\sM + \Ric^\sharp), \\
& \sD(A_0^\sharp)=X_1^\sharp:=\{u\in H^2_{q' , \sigma}(\sM;T\sM) : \, \alpha u +  \cP_\Sigma \left( (\nabla u + [\nabla  u]^{\sT} )  \nu_\Sigma \right)    =0   \text{ on } \Sigma \}.
\end{aligned}
\end{equation*}
Then $(X_0^\sharp, A_0^\sharp)$ generates an interpolation-extrapolation scale $(X_\beta^\sharp, A_\beta^\sharp)$, $\beta\in \bR$, the dual scale.

\medskip\noindent
In the following, we set
$$
H^2_{q,\cB}(\sM;T\sM) = \{u\in H^2_q(\sM;T\sM) : \, \cB u    =0   \text{ on } \Sigma \},
$$
where $\cB u  = (\cB_1 u , \cB_2 u) :=({\rm tr}_\Sigma \left(     \cP_\Sigma(   \nabla u \nu_\Sigma )  + (\alpha + L_\Sigma) \cP_\Sigma u \right) , {\rm tr}_\Sigma  (u | \nu_\Sigma)_g )$.
One readily verifies that
\begin{equation}
\label{X1 equals}
X_1 = H^2_{q,\cB}(\sM;T\sM) \cap L_{q,\sigma}(\sM;T\sM).
\end{equation}
Indeed,  given  any $u\in X_1$, we immediately have  $u\in H^2_q(\sM;T\sM) \cap L_{q,\sigma}(\sM;T\sM)$ and  the boundary condition $\cB_2 u=0$
is automatically satisfied,  see \eqref{sigma-u-nu}.
In view of  \eqref{boundary condition 1}, it holds that on $\Sigma$
\begin{align*}
0  & =  \alpha u +   2\cP_\Sigma ( D_u \nu_\Sigma )
 =     \cP_\Sigma (   \nabla u \nu_\Sigma )  + (\alpha + L_\Sigma) \cP_\Sigma u  =\cB_1 u.
\end{align*}
Therefore, we conclude that $X_1 \subset H^2_{q,\cB}(\sM;T\sM) \cap L_{q,\sigma}(\sM;T\sM)$.
The converse inclusion $H^2_{q,\cB}(\sM;T\sM) \cap L_{q,\sigma}(\sM;T\sM) \subset X_1   $ follows from
$$H^2_{q,\cB}(\sM;T\sM) \cap L_{q,\sigma}(\sM;T\sM) \subset H^2_{q,\sigma}(\sM;T\sM) $$
and \eqref{boundary condition 1}.

In order to characterize the interpolation spaces $X_\beta = [X_0, X_1]_\beta$ and $X_{\beta, p} = (X_0, X_1)_{\beta,p}$ we first include two auxiliary results.
\begin{lemma}
\label{Appendix lem interpolation non-boundary}
Given $\theta\in (0,1)$ and $p\in (1,\infty)$, let $(\cdot,\cdot)_\theta$ stand for either the complex interpolation functor $[\cdot,\cdot]_\theta$,
or the real interpolation functor $(\cdot,\cdot)_{\theta,p}$, respectively.
Then
$$
(L_{q,\sigma}(\sM;T\sM), H^2_{q,\sigma}(\sM;T\sM))_\theta  \doteq (L_q(\sM;T\sM), H^2_q(\sM;T\sM))_\theta \cap  L_{q,\sigma}(\sM;T\sM).
$$
\end{lemma}
\begin{proof}
Let $\widetilde{\bP}_H:= \PH|_{H^2_q(\sM;T\sM)}$.
Then Lemma~\ref{Apendix Lemma Helmholtz} implies
$$\widetilde{\bP}_H \in \cL( H^2_q(\sM;T\sM),  H^2_{q,\sigma}(\sM;T\sM)).$$
Moreover,   $\widetilde{\bP}_H^2=\widetilde{\bP}_H$  and $\widetilde{\bP}_H u =u $ for all $u\in H^2_{q,\sigma}(\sM;T\sM)$.
The assertion then follows from \cite[Theorem~1.17.1.1]{Tri78}.
\end{proof}
\begin{lemma}
\label{Appendix lem interpolation functor}
Given $\theta\in (0,1)$ and $p\in (1,\infty)$, let $(\cdot,\cdot)_\theta$ stand for either the complex interpolation functor $[\cdot,\cdot]_\theta$ or the real interpolation functor $(\cdot,\cdot)_{\theta,p}$.
Then
$$
(X_0, X_1)_\theta    \doteq (L_q(\sM;T\sM), H^{2  }_{q,\cB}(\sM;T\sM))_\theta \cap  L_{q,\sigma}(\sM;T\sM).
$$
\end{lemma}
\begin{proof}
Define
$$A_\cB : \sD(\Delta_\cB):=H^2_{q,\cB}(\sM;T\sM) \to L_q(\sM;T\sM)$$
by $A_\cB u:= -\mu_s(\Delta_\sM + \Ric^\sharp )u$.
It follows from analogous arguments as  in Section~\ref{Section:NS with perfect slip boundary condition strong} that there exists $\lambda_0$ such that for all $\lambda>\lambda_0$
$$
\lambda + A_\cB \in \Lis (H^2_{q,\cB}(\sM;T\sM), L_q(\sM;T\sM)) .
$$
It follows from   Lemma~\ref{surject AN} that there exists $\lambda_0$ such that for all $\lambda>\lambda_0$
$$
\lambda + A_N \in \Lis (X_1, X_0).
$$
Then the assertion follows from a similar argument to \cite[Lemma~3.2]{Ama00}.
For the reader's convenience, we will nevertheless include a proof.
Let
$$Q_1:=(\lambda+ A_N)^{-1} \PH (\lambda+ A_\cB).$$
Then $Q_1\in \cL(H^2_{q,\cB}(\sM;T\sM), X_1)$
 and
\begin{align*}
Q_1^2 u & = (\lambda+ A_N)^{-1} \PH (\lambda+ A_\cB) (\lambda+ A_N)^{-1} \PH (\lambda+ A_\cB) u \\
&=(\lambda+ A_N)^{-1}\PH (\lambda+ A_\cB) u =Q_1 u,
\end{align*}
where we have employed the relations $ \PH A_\cB |_{X_1}=A_N $  and $\PH^2=\PH$.
This further implies $Q_1|_{X_1}= I_{X_1}$, and thus $Q_1$ is a bounded projection from $H^2_{q,\cB}(\sM;T\sM)$ onto $ X_1$.
Now consider $Q_1$ as a closed densely defined operator from  $L_q(\sM;T\sM)$ to  $L_{q,\sigma}(\sM;T\sM)$ with domain $H^2_{q,\cB}(\sM;T\sM)$ and denote this operator by $Q$.
Let
\begin{equation*}
\begin{aligned}
& A_\cB^\sharp: H^2_{q',\cB}(\sM;T\sM) \to L_{q'}(\sM;T\sM), && A_\cB^\sharp:= -\mu_s(\Delta_\sM + \Ric^\sharp)  \\
& A_N^\sharp : \sD(A_{N,q'} ) \to L_{q', \sigma}(\sM;T\sM), && A_N^\sharp := -\mu_s \PH (\Delta_\sM + \Ric^\sharp ) .
\end{aligned}
\end{equation*}
Then
\begin{equation*}
\begin{aligned}
Q' &= (\lambda + A_\cB)' \PH^{\sT} [(\lambda+ A_N  )^{-1}]' \\
&= (\lambda + A_\cB^\sharp) (\lambda+ A_N^\sharp )^{-1}  \in \cL( L_{q', \sigma}(\sM;T\sM) , L_{q' }(\sM;T\sM)),
\end{aligned}
\end{equation*}
where $\PH^{\sT}$ is the dual operator of $\PH\in \cL( L_{q }(\sM;T\sM), L_{q , \sigma}(\sM;T\sM))$.
We note that $\PH^{\sT}$ is indeed the embedding operator $i^\sharp: L_{q', \sigma}(\sM;T\sM) \to L_{q' }(\sM;T\sM)$.
Therefore, $Q''\in \cL( L_q(\sM;T\sM) , L_{q , \sigma}(\sM;T\sM)  )$.
Together with the inclusion $Q\subset Q''$ and the density of $H^2_{q,\cB}(\sM;T\sM)$ in $L_q(\sM;T\sM)$, this implies that $Q  $ has a unique bounded extension $Q_0 \in \cL( L_q(\sM;T\sM) , L_{q , \sigma}(\sM;T\sM)  )$.
It is easy to check that $Q_0$ is a projection and $Q_0|_{X_0}=I_{X_0}$. Then the assertion follows from \cite[Theorem~1.17.1.1]{Tri78}.
\end{proof}
We are now ready to state the first main result of this section, providing
a characterization of the complex interpolation spaces $X_\beta:=[X_0, X_1]_\beta$.

\begin{proposition}
\label{Appendix Prop interpolation 1}
Let $\beta\in (0,1)\setminus \{ \frac{1}{2}+\frac{1}{2q} \}$. Then $X_\beta= H^{2\beta}_{q,\sigma,\cB}(\sM; T\sM)$, where
\begin{equation*}
\begin{aligned}
H^{2\beta}_{q,\sigma,\cB}(\sM; T\sM)
 =\begin{cases}
\{ u\in H^{2\beta}_{q,\sigma}(\sM;T\sM): \, \alpha u+ 2\cP_\Sigma ( D_u \nu_\Sigma )    =0  \ \   \text{on}\ \  \Sigma\},  &  \frac{1}{2}+\frac{1}{2q}< \beta<1,\\
H^{2\beta}_{q,\sigma}(\sM;T\sM), & 0<\beta<\frac{1}{2}+\frac{1}{2q}.
\end{cases}
\end{aligned}
\end{equation*}
\end{proposition}

\begin{proof}
We first observe that
\begin{equation*}
\begin{aligned}
&\cB_1 \in \cL(W^s_q(\sM; T\sM) , W^{s-1-1/q}_q(\Sigma; T\Sigma)),  && 1+1/q<s \le 2, \\
&\cB_2 \in \cL(W^s_q(\sM; T\sM) , W^{s-1/q}_q(\Sigma)), && 1/q <s \le 2,
\end{aligned}
\end{equation*}
are  normal boundary operators in the sense of \cite[Definition~3.1]{See72},
see also \cite[Section VIII.2]{Ama19}.
Then \cite[Theorem~4.1]{See72} implies,  see also \cite[Theorem 2.4.8]{Ama19} for the case $\sM= \bR^n_+$,
$$
[L_q(\sM;T\sM), H^2_{q,\cB}(\sM;T\sM)]_\beta =: H^{2 \beta}_{q,\cB}(\sM;T\sM),
$$
where
$$
H^{2 \beta}_{q,\cB}(\sM;T\sM) \doteq
\begin{cases}
\{ u\in H^{2\beta}_q(\sM;T\sM): \, \cB u    =0    \quad  \text{on } \Sigma\},  & \frac{1}{2}+\frac{1}{2q}< \beta<1,\\
\{ u\in H^{2\beta}_q(\sM;T\sM): \, \cB_2 u    =0    \quad  \text{on } \Sigma\},  & \frac{1}{2q}< \beta<\frac{1}{2}+\frac{1}{2q}, \\
H^{2\beta}_q(\sM;T\sM), &\; 0<\beta< \frac{1}{2q}.
\end{cases}
$$
Lemma~\ref{Appendix lem interpolation functor} shows that for $\beta\in (0,1)\setminus\{\frac{1}{2q}, \frac{1}{2}+\frac{1}{2q}\}$
$$
[X_0, X_1]_\beta  \doteq   H^{2 \beta}_{q,\cB}(\sM;T\sM)  \cap  L_{q,\sigma}(\sM;T\sM).
$$
By a similar argument as in~\eqref{X1 equals}, we obtain
\begin{equation*}
\begin{aligned}
& H^{2 \beta}_{q,\cB}(\sM;T\sM)  \cap  L_{q,\sigma}(\sM;T\sM) \\
=& \begin{cases}
\{ u\in H^{2\beta}_{q,\sigma}(\sM;T\sM): \, \alpha u+  2\cP_\Sigma ( D_u \nu_\Sigma )    =0    \quad  \text{on } \Sigma\}, & \beta \in (\frac{1}{2}+\frac{1}{2q}, 1), \\
H^{2\beta}_{q,\sigma}(\sM;T\sM), &  \beta \in (0,\frac{1}{2}+\frac{1}{2q}) \setminus \{\frac{1}{2q} \} .
\end{cases}
\end{aligned}
\end{equation*}
We will now pay attention to the particular case $\beta = \frac{1}{2q}$, which is currently excluded in the characterization above.
We know that $X_{1/2}= H^1_{q,\sigma}(\sM;T\sM)$.
It  follows from the reiteration theorem that
$$
[X_0,X_{1/2}]_\alpha = [X_0 , [X_0, X_1]_{1/2}]_\alpha =[X_0, X_1]_{\frac{\alpha}{2}} = X_{\alpha/2} .
$$
Taking $\alpha=1/q$ and using Lemma~\ref{Appendix lem interpolation non-boundary} and the reiteration theorem yields
\begin{align*}
X_{1/ 2q} & \doteq [X_0,X_{1/2}]_{1/ q}  = [ X_0, [X_0, H^2_{q,\sigma}(\sM;T\sM)]_{1/2}]_{1/q} \\
& = [X_0,H^2_{q,\sigma}(\sM;T\sM)]_{1/2q} \doteq H^{1/q}_{q,\sigma}(\sM;T\sM).
\end{align*}
This proves the assertion for the case $\beta=1/2q$ and thus completes the proof.
\end{proof}
\noindent
We obtain an analogous result for the real interpolation spaces $X_{\beta,p}:=(X_0,X_1)_{\beta,p}$.
\begin{proposition}
\label{Appendix Prop interpolation 2}
Let $\beta\in (0,1)\setminus \{ \frac{1}{2}+\frac{1}{2q}\}$ and $p\in (1,\infty)$. Then
$$X_{\beta,p}=(X_0, X_1)_{\beta,p} = B^{2\beta}_{qp,\sigma,\cB}(\sM;T\sM), \quad\text{where}$$
\begin{equation*}
\begin{aligned}
B^{2\beta}_{qp,\sigma,\cB}(\sM;T\sM)
  =\begin{cases}
\{ u\in B^{2\beta}_{qp,\sigma}(\sM;T\sM): \, \alpha u+ 2 \cP_\Sigma \left( D_u \nu_\Sigma \right)=0  \ \  \text{on} \ \  \Sigma\},  &\frac{1}{2}+\frac{1}{2q}< \beta<1,\\
B^{2\beta}_{qp,\sigma}(\sM;T\sM), &0<\beta<\frac{1}{2}+\frac{1}{2q}.
\end{cases}
\end{aligned}
\end{equation*}
\end{proposition}
\begin{proof}
The case $\beta\in (0,1)\setminus \{ \frac{1}{2q}, \frac{1}{2}+\frac{1}{2q}\}$  can be obtained  as in Proposition~\ref{Appendix Prop interpolation 1},
 see for instance \cite{Gri69} or \cite[Theorem 2.4.5]{Ama19} for the Euclidean case.

To treat the case $\beta=\frac{1}{2q}$, note that  the reiteration theorem for the complex and real method implies
$$
X_{1/2q, p} = (X_0,X_1)_{1/2q,p}= (X_0, [X_0,X_1]_{1/2})_{1/q,p}= (X_0, X_{1/2})_{1/q,p}.
$$
We can further utilize Lemma~\ref{Appendix lem interpolation non-boundary} to obtain
\begin{align*}
X_{1/2q, p} & = (X_0, X_{1/2})_{1/q,p} = (X_0, [X_0, H^2_{q,\sigma}(\sM;T\sM)]_{1/2} )_{1/q,p} \\
&= (X_0,H^2_{q,\sigma}(\sM;T\sM))_{1/2q,p} = B^{1/q}_{qp, \sigma} (\sM;T\sM).
\end{align*}
This completes the proof.
\end{proof}

Using a duality argument, we can also characterize the interpolation spaces between $X_{-1/2}$ and $X_{1/2}$.
\begin{proposition}
\label{Appendix Prop interpolation 3}
 Let $\beta\in (0,1) $ and $p\in (1,\infty)$. Then
\begin{align*}
[X_{-1/2}, X_{1/2}]_\beta   = H^{2\beta -1} _{q,\sigma}(\sM; T\sM)   \quad \text{and} \quad
(X_{-1/2}, X_{1/2})_{\beta,p}   = B^{2\beta -1} _{q p,\sigma}(\sM; T\sM) ,
\end{align*}
where
$$  H^{2\beta -1} _{q,\sigma}(\sM; T\sM):=\left(H^{1-2\beta}_{q',\sigma}(\sM; T\sM)\right)^\prime \!\!,
\quad  B^{2\beta -1} _{q p,\sigma}(\sM; T\sM) := \left( B^{1-2\beta} _{q' p',\sigma}(\sM; T\sM)  \right)^\prime
$$
for $\beta\in (0, 1/2)$.
\end{proposition}
\medskip
\noindent
Similar results hold for spaces with other boundary conditions. For instance, let
\begin{equation*}
\begin{aligned}
&Z_0=X_0,  && Z_1= \{ u\in H^2_{q,\sigma}(\sM;T\sM): \,   \cP_\Sigma \left( (\nabla u - [\nabla u]^\sT )  \nu_\Sigma \right)  =0 \ \ \text{on}\ \ \Sigma\},\quad\text{or} \\
&Y_0=L_q(\sM), && Y_1= \{ \phi\in H^2_q(\sM): \,    (\gd \phi | \nu_\Sigma)_g  =0  \ \ \text{on}\ \  \Sigma\}.
\end{aligned}
\end{equation*}
Then we have the following result.
\begin{proposition}
\label{Appendix Prop interpolation 4}
Let $\beta \in (0,1)$. Then the interpolation  spaces  $Z_\beta=[Z_0,Z_1]_\beta$ and $Y_\beta=[Y_0,Y_1]_\beta$ can be characterized as follows.
\begin{equation*}
\begin{aligned}
&Z_\beta=
\begin{cases}
\{ u\in H^{2\beta}_{q,\sigma}(\sM;T\sM): \,   \cP_\Sigma \left( (\nabla u - [\nabla u]^\sT )  \nu_\Sigma \right)    =0 \ \ \text{on}\ \ \Sigma\}, &\frac{1}{2}+\frac{1}{2q}< \beta<1,\\
H^{2\beta}_{q,\sigma}(\sM;T\sM), &  0<\beta<\frac{1}{2}+\frac{1}{2q}.
\end{cases} \\
& \\
&Y_\beta=
\begin{cases}
\{\phi\in H^{2\beta}_q(\sM ): \,      (\gd \phi | \nu_\Sigma)_g  =0   \ \ \text{on}\ \ \Sigma\},  & \frac{1}{2}+\frac{1}{2q}< \beta<1,\\
H^{2\beta}_q(\sM ), & 0<\beta<\frac{1}{2}+\frac{1}{2q}.
\end{cases}
\end{aligned}
\end{equation*}
\end{proposition}

\bigskip
\section{Sectorial operators and $H^\infty$-calculus}
\label{Appendix D}
In this part of the appendix, we will introduce several basic concepts concerning maximal $L_p$-regularity theory.
The reader may refer to the treatises \cite{Ama95},   \cite{DenHiePru03} and \cite{PruSim16} for more details of these concepts.

For $\theta\in (0,\pi]$, the open sector with angle $2\theta$ is denoted by
$$\Sigma_\theta:= \{\omega\in \mathbb{C}\setminus \{0\}: |\arg \omega|<\theta \}. $$
\begin{definition}\label{Def sectorial}
Let $X$ be a complex Banach space, and $\cA$ be a densely defined closed linear operator in $X$ with dense range. $\cA$ is called sectorial if $\Sigma_\theta \subset \rho(-\cA)$ for some $\theta>0$ and
$$ \sup\{\|\mu(\mu+\cA)^{-1}\|_{\cL(X)} : \mu\in \Sigma_\theta \}<\infty. $$
The class of sectorial operators in $X$ is denoted by $\cS(X)$.
The spectral angle $\phi_\cA$ of $\cA$ is defined by
$$
\phi_\cA:=\inf\{\phi:\, \Sigma_{\pi-\phi}\subset \rho(-\cA),\, \sup\limits_{\mu\in \Sigma_{\pi-\phi}} \|\mu(\mu+\cA)^{-1} \|_{\cL(X)}<\infty \}.
$$
\end{definition}

Let $\phi\in (0,\pi]$. Define
$$H^\infty(\Sigma_\phi):=\left\{f: \Sigma_\phi \to \mathbb{C}: f \text{ is analytic and } \|f\|_\infty<\infty \right\} $$
and
$$\cH_0(\Sigma_\phi) = \left\{f\in H^\infty(\Sigma_\phi): \exists s>0, c>0 \text{ s.t. } |f(z)| \leq c\frac{|z|^s}{1 +|z|^{2s}} \right\}. $$

\begin{definition}\label{Def: cHi}
Suppose that $\cA\in \cS(X)$. Then $\cA$ is said to admit a bounded $H^\infty$-calculus if there are $\phi>\phi_\cA$ and a constant $K_\phi$ such that
\begin{equation}
\label{Appendix B: cHi}
\|f(\cA) \|_{\cL(X)} \leq K_\phi \|f\|_{\infty} ,\quad f\in \cH_0(\Sigma_{\pi-\phi}).
\end{equation}
Here
\begin{equation}
\label{Def integral contour}
f(\cA):=-\frac{1}{2\pi i}\int_\Gamma (\lambda + \cA)^{-1} f(\lambda) \, d\lambda, \quad
\Gamma=
\begin{cases}
-t e^{-i \theta} \quad & \text{for } t<0, \\
t e^{ i \theta}   & \text{for } t\geq 0,
\end{cases}
\end{equation}
is a positively oriented contour
for any $\theta\in (0,\pi-\phi  ) $.
The class of such operators is denoted by $H^\infty(X)$. The $H^\infty$-angle of $\cA$ is defined by
$$\phi^\infty_\cA:=\inf\{\phi>\phi_\cA: \eqref{Appendix B: cHi} \text{ holds}\}.$$
\end{definition}
If an operator $\cA\in H^\infty(X)$ with $H^\infty$-angle $\phi^\infty_\cA<\pi/2$ and $X$ is of class {\rm UMD}, then Condition (H3) in \cite{PrWi17} is satisfied with the choices $X_0=X$ and $X_1=\sD(\cA)$.

\section*{Acknowledgements}
 We would like to thank Prof. Christian B\"ar for helpful discussions concerning Killing fields that vanish on the boundary of $\sM$,
 see Proposition~\ref{Prop: characterizing Ealpha}.

 We would also like to thank the anonymous reviewer for thoughtful and constructive suggestions.


\bigskip\noindent
{\bf Data Availability Statement:} Data sharing is not applicable, as no new datasets were generated or analyzed for this article.

\medskip\noindent
{\bf Conflict of interest:}  The authors assert that there is no conflict of interest to declare.



\begin{thebibliography}{99}


\bibitem{AKST04}
T. Akiyama, H. Kasai, Y. Shibata,  M.  Tsutsumi,
\emph{On a resolvent estimate of a system of Laplace operators with perfect wall condition}
. Funkcial. Ekvac. {\bf 47}, no. 3, 361-394  (2004). 


\bibitem{Ama95}
H.~Amann,
\emph{Linear and Quasilinear Parabolic Problems: Volume I. Abstract Linear Theory}.
Monographs in Mathematics, 89. Birkh\"auser Boston, Inc., Boston, MA,  1995.

\bibitem{Ama00}
H. Amann,
\emph{On the strong solvability of the Navier-Stokes equations.}
J. Math. Fluid Mech. \textbf{2}, no. 1, 16-98 (2000).

\bibitem{Ama13}
H. Amann,
\emph{Function spaces on singular manifolds.}
Math. Nachr. \textbf{286}, no. 5-6, 436-475 (2013).

\bibitem{Ama19}
H. Amann,
\emph{Linear and quasilinear parabolic problems. Vol. II. Function spaces}.
Monographs in Mathematics, 106. Birkh\"auser/Springer, Cham, 2019.

\bibitem{CCD17} C.H. Chan, M. Czubak, M. Disconzi,
\emph{The formulation of the Navier-Stokes equations on Riemannian manifolds.}
 J. Geom. Phys. {\bf 121}, 335-346 (2017).

\bibitem{DDHP04}
R. Denk, G. Dore, M. Hieber, J. Pr\"uss, A. Venni,
\emph{New thoughts on old results of R. T. Seeley}.
Math. Ann. {\bf 328}, no. 4, 545-583 (2004).

\bibitem{DenHiePru03}
R. Denk, M. Hieber, J. Pr\"uss,
\emph{$\mathscr{R}$-boundedness, Fourier multipliers and problems of elliptic and parabolic type}.
Mem. Amer. Math. Soc. {\bf 166}, no. 788, 2003.

\bibitem{DuoSim97}	
X.T. Duong, G. Simonett,
\emph{$H^\infty$-calculus for elliptic operators with nonsmooth coefficients}.
Differential Integral Equations {\bf 10}, no. 2, 201-217 (1997).


\bibitem{KatoFujita64}
H. Fujita, T. Kato,
\emph{On the Navier-Stokes initial value problem. I.}
Arch. Rational Mech. Anal. {\bf 16}, 269-315 (1964).


\bibitem{GHT13}	
M. Geissert, H. Heck, C. Trunk,
\emph{$H^\infty$-calculus for a system of Laplace operators with mixed order boundary conditions}.
Discrete Contin. Dyn. Syst. Ser. S {\bf 6}, no. 5, 1259-1275 (2013).


\bibitem{GeradHill}
D. G\'erard-Varet, M. Hillairet,
\emph{Regularity issues in the problem of fluid structure interaction.}
Arch. Ration. Mech. Anal. {\bf 195}, no. 2, 375-407 (2010).

\bibitem{GeradHillWang}
D. G\'erard-Varet, M. Hillairet,  C. Wang,
\emph{The influence of boundary conditions on the contact problem in a 3D Navier-Stokes flow.}
J. Math. Pures Appl. (9) {\bf 103}, no. 1, 1-38 (2015).


\bibitem{Giga}
Y. Giga,
\emph{Domains of fractional powers of the Stokes operator in $L_r$ spaces.}
Arch. Rational Mech. Anal. {\bf 89}, no. 3, 251-265 (1985).


\bibitem{GigaMiyakawa85}
Y. Giga, T. Miyakawa,
\emph{Solutions in $L_r$ of the Navier-Stokes initial value problem.}
Arch. Rational Mech. Anal. {\bf 89}, no. 3, 267-281 (1985).


\bibitem{Gri69}
P. Grisvard,
\emph{\'Equations diff\'erentielles abstraites}.
Ann. Sci. \'Ecole Norm. Sup. (4) 2, 311–395 (1969).

\bibitem{Hill07}
M. Hillairet,
\emph{Lack of collision between solid bodies in a 2D incompressible viscous flow.}
Comm. Partial Differential Equations {\bf 32}, no. 7-9, 1345-1371 (2007).

\bibitem{HillTaka}
M. Hillairet, T. Takahashi,
\emph{Collisions in three-dimensional fluid structure interaction problems.}
SIAM J. Math. Anal. {\bf 40}, no. 6, 2451-2477 (2009).

\bibitem{KatoFujita62}
T. Kato, H. Fujita
\emph{On the nonstationary Navier-Stokes system.}
Rend. Sem. Mat. Univ. Padova {\bf 32}, 243-260 (1962).

\bibitem{JaOlRe17} T. Jankuhn, M.A. Olshanskii, A. Reusken,
\emph{Incompressible fluid problems on embedded surfaces: modeling and variational formulations.}
Interfaces Free Bound. {\bf 20},  353-377 (2018).

\bibitem{LePrWi14}
J. LeCrone, J. Pr\"uss, M. Wilke,
\emph{On quasilinear parabolic evolution equations in weighted $L_p$-spaces II}.
J. Evol. Equ. {\bf 14}, no. 3, 509-533 (2014).

\bibitem{LeeBook}
J.M. Lee,
\emph{Introduction to smooth manifolds.}
Second edition. Graduate Texts in Mathematics, 218. Springer, New York, 2013.

\bibitem{Lun95}
A. Lunardi,
\emph{Analytic Semigroups and Optimal Regularity in Parabolic Problems}.
Progress in Nonlinear Differential Equations and Their Applications \textbf{16}, Springer, Basel, 1995.

\bibitem{MPS19}
    G. Mazzone,  J. Pr\"uss, G. Simonett,
    \emph{A maximal regularity approach to the study of motion of a rigid body with a fluid-filled cavity. }
    J. Math. Fluid Mech. {\bf 21},  no. 3, Paper No. 44, 20 pp. (2019).

\bibitem{MiMo09}
M. Mitrea, S. Monniaux,
\emph{On the analyticity of the semigroup generated by the Stokes operator with Neumann-type boundary conditions
on Lipschitz subdomains of Riemannian manifolds}.
Trans. Amer. Math. Soc. {\bf 361}, no. 6, 3125-3157 (2009).

\bibitem{MiMo09b}
M. Mitrea, S. Monniaux,
\emph{The nonlinear Hodge-Navier-Stokes equations in Lipschitz domains.}
Differential Integral Equations {\bf 22}  no. 3-4, 339-356 (2009).

\bibitem{OQRY18}
M.A. Olshanskii, A. Quaini,  A. Reusken, V. Yushutin,
\emph{A finite element method for the surface Stokes problem.}
 SIAM J. Sci. Comput. {\bf 40},  A2492--A2518 (2018).

\bibitem{Petersenbook}
P. Petersen,
\emph{Riemannian geometry.}
Second edition. Graduate Texts in Mathematics, 171. Springer, New York, 2006.

\bibitem{Pri94}
V. Priebe,
\emph{Solvability of the Navier-Stokes equations on manifolds with boundary}.
Manuscripta Math. {\bf 83}, no. 2, 145-159 (1994).

\bibitem{PruSim16}
J. Pr\"uss, G. Simonett,
\emph{Moving Interfaces and Quasilinear Parabolic Evolution Equations.}
Monographs in Mathematics. Birkh\"auser Verlag. 2016.


\bibitem{PrSiWi18}
J. Pr\"uss, G. Simonett,  M. Wilke,
\emph{Critical spaces for quasilinear parabolic evolution equations and applications}.
J. Differential Equations {\bf 264}, no. 3, 2028-2074 (2018).

\bibitem{PrSiWi21}
J. Pr\"uss, G. Simonett,  M. Wilke,
\emph{On the Navier-Stokes equations on surfaces}.
J. Evol. Equ. {\bf 21}, no. 3, 3153-3179 (2021).

\bibitem{PrWi17}
J. Pr\"uss,    M. Wilke,
\emph{Addendum to the paper ``On quasilinear parabolic evolution equations in weighted $L_p$-spaces II'' [MR3250797]}.
J. Evol. Equ. {\bf 17}, no. 4, 1381-1388 (2017).

\bibitem{PrWi18}
J. Pr\"uss,    M. Wilke,
\emph{On critical spaces for the Navier-Stokes equations}.
J. Math. Fluid Mech. {\bf 20}, no. 2, 733-755 (2018).

\bibitem{ReZh13}
A. Reusken,Y. Zhang,
\emph{Numerical simulation of incompressible two-phase flows with a Boussinesq-Scriven interface stress tensor.}
Internat. J. Numer. Methods Fluids {\bf 7}, 1042-1058  (2013).


\bibitem{ReVo18}
S. Reuther, A. Voigt,
\emph{Solving the incompressible surface Navier-Stokes equation by surface elements.}
 Phys. Fluids  {\bf 30}, 012107  (2018) .

\bibitem{SaTu20}
M. Samavaki, J. Tuomela,
\emph{Navier-Stokes equations on Riemannian manifolds}.
J. Geom. Phys. {\bf 148} 103543, 15 pp. (2020).

\bibitem{See72}
R. T. Seeley,
\emph{Interpolation in Lp with boundary conditions}.
Stud. Math. {\bf 44}, 47-60 (1972),

\bibitem{SiWi22}
G. Simonett,  M. Wilke,
\emph{$H^\infty$-calculus for the surface Stokes operator and applications}.
 J. Math. Fluid Mech. {\bf 24}, no. 4, no. 109, 23 pp (2022).

\bibitem{SlSaLe07}
J.C. Slattery, L. Sagis, E.-U. Oh,
\emph{Interfacial transport phenomena.}
 Second edition. Springer, New York, 2007.


\bibitem{SoSc73}
V. A. Solonnikov and V. E. Scadilov,
\emph{On a boundary value problem for a stationary system of Navier-Stokes equations.}
Proc. Steklov Inst. Math. {\bf 125}, 186--199 (1973).

\bibitem{Tri78}
H. Triebel,
\emph{Interpolation Theory, Function Spaces, Differential Operators}.
North-Holland Publishing Co., Amsterdam, New York, 1978.
	
\bibitem{Yano59}
K. Yano,
\emph{Harmonic and Killing vector fields in compact orientable Riemannian spaces with boundary}.
Ann. of Math. (2) 69, 588-597 (1959).

\bibitem{YanoAko65}
K. Yano, M. Ako,
\emph{Vector fields in Riemannian and Hermitian manifolds with boundary}.
Kōdai Math. Sem. Rep. {\bf 17}, 129-157 (1965).	
	
 \end{thebibliography}
 \end{document}